\documentclass[a4paper,11pt]{amsart}
\pdfoutput=1
\usepackage{mathrsfs}
\usepackage{tikz-cd}

\usepackage{amsmath,amssymb,color}
\usepackage{amsthm}
\usepackage{hyperref}
\usepackage{MnSymbol}
\usepackage[margin=1.0in]{geometry}
\usepackage{cleveref}
\usepackage{cancel}
\usepackage{verbatim}
\usepackage{xcolor}
\usepackage{mathtools}

\numberwithin{equation}{section}

\newtheorem{prop}{Proposition}
\newtheorem{lemma}[prop]{Lemma}

\newtheorem{thm}[prop]{Theorem}
\newtheorem{cor}[prop]{Corollary}

\numberwithin{prop}{section}

\theoremstyle{definition}
\newtheorem{defn}[prop]{Definition}

\newtheorem{rmk}[prop]{Remark}

\DeclareSymbolFont{script}{U}{eus}{m}{n}
\DeclareSymbolFontAlphabet{\amathscr}{script}
\DeclareMathSymbol{\Wedge}{0}{script}{"5E}
\DeclareMathAlphabet{\mathrmsl}{OT1}{cmr}{m}{sl}

\DeclareMathAlphabet\mathbfcal{OMS}{cmsy}{b}{n}

\newcommand{\Ome}{\alpha}
\newcommand{\A}{\beta}
\newcommand{\al}{\theta}

\newcommand{\tstM}{{\mathscr M}}% test configutation Kahler manifold

\newcommand{\tstA}{{\mathscr A}}% test configutation Kahler class
\newcommand{\tstB}{{\mathscr B}}% test configutation promitove class

\newcommand{\tstK}{{\mathscr K}}% canonical bundle
\newcommand{\Pol}{{\mathrm P}}
\newcommand{\T}{{\mathbb T}}
\newcommand{\tor}{{\mathfrak t}}

\newcommand{\DD}{\mathcal D}

\newcommand{\VV}{\mathcal V}

\renewcommand{\bar}[1]{\overline{#1}}

\newcommand{\Scal}{R}

\newcommand{\C}{{\mathbb C}}
\newcommand{\R}{{\mathbb R}}
\newcommand{\Z}{{\mathbb Z}}

\newcommand{\PP}{\mathbb P}
\newcommand{\cO}{\mathcal O}
\newcommand{\Aut}{\mathrm{Aut}}
\newcommand{\ap}{{\mathrm x}}
\newcommand{\bp}{{\mathrm y}}

\DeclareMathOperator{\Vol}{Vol}

\begin{document}

\title[Quasi-regular BHE 3-folds]{Pluriclosed 3-folds with vanishing Bismut Ricci form: General theory in the quasi-regular case}

\author{Vestislav Apostolov}
\address{V.\,Apostolov\\ D{\'e}partement de Math{\'e}matiques\\ UQAM \\
 and \\ Institute of Mathematics and Informatics\\ Bulgarian Academy of Sciences}
\email{\href{mailto:apostolov.vestislav@uqam.ca}{apostolov.vestislav@uqam.ca}}

\author{Abdellah Lahdili}
 \address{Abdellah Lahdili \\
          Campus Saint-Jean\\ University of Alberta\\ 8406\\ 91 Street\\ Edmonton (AB)\\ Canada T6C 4G9 }
 \email{\href{mailto:lahdili@ualberta.ca}{lahdili@ualberta.ca}}
 
\author{Kuan-Hui Lee}
\address{K-H. Lee\\
          Department of Mathematics and Statistic, McGill University \\
          and \\ CIRGET, UQAM}
 \email{\href{mailto:kuan-hui.lee@mcgill.ca}{kuan-hui.lee@mcgill.ca}}

%\today
%\subjectclass{53C21, 53C25, 53C55}

\begin{abstract}  We study compact complex $3$-dimensional non-K\"ahler Bismut Ricci flat pluriclosed Hermitian manifolds (BHE) via their dimensional reduction to a special K\"ahler geometry in complex dimension $2$, recently obtained in \cite{ABLS}. We show that in the quasi-regular case, the reduced geometry satisfies a 6th order non-linear PDE which has infinite dimensional momentum map interpretation,  similar to the much studied K\"ahler metrics of constant scalar curvature (cscK). We use this to associate to the reduced manifold or orbifold  Mabuchi and Calabi functionals, as well as to obtain obstructions for the existence of solutions in terms of the automorphism group, paralleling results by Futaki and Calabi--Lichnerowicz--Matsushima in the cscK case. This is used to characterize the Samelson locally homogeneous BHE geometries  in complex dimension 3  as the only non-K\"ahler BHE $3$-folds with $2$-dimensional Bott--Chern $(1,1)$-cohomology group,  for which the reduced space is a smooth K\"ahler surface. We also discuss explicit solutions  of the PDE on orthotoric K\"ahler orbifold surfaces,  motivated by  examples found by Couzens--Gauntlett--Martelli--Sparks~\cite{CGMS} in the framework of  supersymmetric ${\rm AdS}_3 \times Y_7$  type IIB supergravity. Our construction yields infinitely many non-K\"ahler BHE structures on $S^3\times S^3$  and $S^1\times S^2 \times S^3$,  which are not locally isometric to a Samelson geometry. These appear to be  the first such examples. 
\end{abstract}

\maketitle

\section{Introduction}
Non-K\"ahler,  Calabi--Yau type Hermitian metrics are subject of intense recent interest in mathematical physics and complex geometry (see e.g.  \cite{fernandez2014non, finosurvey, JFS, garciafern2018canonical,ivanov2010heterotic,phong2019geometric,picard2024strominger,st-geom, tosatti2015non}).  Given a Hermitian manifold $(N, J_N, g_N)$ of complex dimension $n$, we say that the metric is \emph{pluriclosed} if its K\"ahler form $\omega_N$ satisfies  $\sqrt{-1} \partial \bar \partial \omega_N=0$; this condition is also known in the physics literature  as \emph{strong K\"ahler with torsion} (SKT). Associated to $(g_N, J_N)$, there is a unique Hermitian connection with skew-symmetric torsion, which we shall refer to as the \emph{Bismut connection} of $(g_N, J_N)$ (cf. \cite{Bismut, StromingerSST}). The Bismut connection defines a closed $2$-form representing the first Chern class of $(N, J_N)$, the \emph{Bismut Ricci form} of $(g_N,J_N)$. Pluriclosed Hermitian metric with zero Bismut Ricci form, called \emph{Bismut--Hermitian--Einstein} (BHE) in \cite{JFS, GRFBook,ABLS}, are thus natural extensions of K\"ahler metrics with zero Ricci curvature to  complex manifolds which do not support K\"ahler metrics. They solve the  Hermite--Einstein equations on a certain holomorphic vector bundle associated to $(N, J_N)$~\cite{JFS} and arise  as fixed points of the pluriclosed flow~\cite{PCF,PCFReg}.  Thus BHE structures can be thought of as natural candidates for `canonical metrics' in complex non-K\"ahler geometry. 
It turns out that non-K\"ahler BHE geometry in $3$-complex dimensions is also related to solutions of the supersymmetric ${\rm AdS}_3 \times Y_7$  and ${\rm AdS}_2 \times Y_9$ of type IIB  and type D=11 supergravity~\cite{GK,CGMS}.

Naively, one may expect that upon allowing for non-K\"ahler geometries one should obtain a vast array of examples beyond the Calabi--Yau K\"ahler manifolds. There are hints that this might not necessarily be the case: a result of Gauduchon--Ivanov~\cite{GauduchonIvanov}  states that for complex surfaces, the only non-K\"ahler BHE manifolds are,  in fact, finite quotients of the standard Bismut-flat
Hopf surface $S^1\times S^3$. The latter is the complex $2$-dimensional manifestation of a general construction of Hermitian structures with zero Bismut curvature on any even dimensional Lie group of compact type: the metric $g_N$ is then bi-invariant whereas the complex structure $J_N$ is a left (or right) invariant Samelson's complex structure.  We shall refer to these examples as  \emph{Samelson}'s BHE examples. Remarkably, all known to us compact BHE examples discussed in the literature are in fact covered by products of Samelson's BHE manifolds with  Ricci-flat K\"ahler  manifolds (see~\cite{brienza2024cyt,ZhengBismutflat}).

\bigskip
In complex dimension $3$, Samelson's Bismut--Hermitian--Einstein metrics arise as quotients of the following two  geometries:
\begin{equation}\label{Samelson-3D} SU(2)\times SU(2) \qquad \text{and} \qquad  SU(2)\times \R^3.\end{equation}
In view of the complex $2$-dimensional Gauduchon-Ivanov's rigidity result  \cite{GauduchonIvanov} mentioned above, one might ask if there are any other examples in complex dimension $3$?
In \cite{ABLS}, the following partial result was  established:
\begin{thm}\label{thm:reduction} \cite{ABLS} Let $(N, J_N, g_N, \omega_N)$ be a compact complex 3-dimensional non-K\"ahler Hermitian manifold  such that the K\"ahler form $\omega_N$ is pluriclosed and the Ricci form of the Bismut connection is identically zero. Then  the  Bott--Chern cohomology group $H^{1,1}_{\rm BC}(N,J_N)$ is at least $2$- dimensional, $(J_N, g_N,  \omega_N)$  admits two commuting Killing fields of constant length,  $V, J_NV$, giving rise to a real two-dimensional foliation  ${\mathcal V} \subset TN$  and  a transversal K\"ahler structure  $(g_K^T, J^T, \omega_K^T)$,  such that   on any local quotient by leaves of $\mathcal{V}$
\begin{equation}\label{PDE}
 \frac{1}{2}dd^c \Scal(\omega_K^T)\wedge \omega_K^T= \al^T \wedge \al^T + \rho(\omega_K^T) \wedge \rho(\omega_K^T), \qquad \Scal(\omega_K^T)>0,
\end{equation}
where $\Scal(\omega_K^T)$ and $\rho(\omega_K^T)$ stand, respectively, for the scalar curvature and Ricci form of the transversal K\"ahler structure, and $\al^T$ is a $\omega_K^T$-primitive closed $(1,1)$-form.  Furthermore, $(N, g_N, J_N)$ is isometrically covered by a Samelson geometry \eqref{Samelson-3D} if and only if 
 $H^{1,1}_{\rm BC}(N, J_N)$ is 2-dimensional and $\Scal(\omega_K^T)=const$.
\end{thm}
As a matter of fact, the results in \cite{ABLS} also provide a canonical way to construct 3-dimensional complex BHE manifolds from  K\"ahler surfaces satisfying \eqref{PDE}: in the case when the space of leaves of $\VV$ is a smooth or an orbifold complex surface $S$, --- a situations which we will refer to in this paper respectively as \emph{regular} or \emph{quasi-regular} Bismut--Hermitian--Einstein structure, --- $N$ is the total space of a principal $2$-dimensional torus bundle over $S$, endowed with a basis $\{V, W\}$ of the Lie algebra of the torus and corresponding connection $1$-forms $\eta_V$  and $\eta_{W}$ whose curvatures are respectively $\al^T$  and $-\rho(\omega_K^T)$. The complex structure $J_N$ on $N$ is then the pull-back of the complex structure $J^T$ of $S$ to the horizontal distribution $\mathcal{H}$ with respect to $\eta=\eta_{V}V +\eta_{W}W$, completed with the action $J_N(V):=W, \, J_N (W):=-V$ on the fundamental vector fields,    whereas the Hermitian metric is $g_N= \pi^*(\frac{\Scal(\omega_K)}{2}g_K^T) + \eta_V^2 + \eta_{W}^2$.  Thus, Theorem~\ref{thm:reduction} gives a powerful tool to construct BHE manifolds from complex $2$-dimensional K\"ahler metrics satisfying \eqref{PDE}. Going back to the complex $3$-dimensional Samelson's geometries \eqref{Samelson-3D}, they correspond to the case when $\VV$ is \emph{regular} and the corresponding smooth K\"ahler surfaces are, respectively,  $\PP^1 \times \PP^1$ endowed with a product of constant scalar curvature metrics on each factor,  or a geometrically ruled elliptic surface $\PP(E)\to \Sigma$ which is the projectivisation of a polystable rank $2$-vector bundle $E$ over an elliptic curve $\Sigma$,  endowed with a locally symmetric K\"ahler metric of constant scalar curvature. 

\bigskip
Thus motivated, the first main purpose of this paper is to lay down a framework for a comprehensive study of the condition \eqref{PDE} on a general K\"ahler orbifold surface $S$,   similar to the existence problem for K\"ahler metrics of  constant scalar curvature (cscK for short) that has been a main focus of activity in K\"ahler geometry during the last 70 years or so.  To this end, notice that after applying the Hodge star  operator,   \eqref{PDE} can be rewritten as the scalar equation
\begin{equation}\label{Box}
\Delta\Scal + \frac{\Scal^2}{2} -\|{\rm Rc}\|^2-\|\al\|^2 =0, \qquad \Scal>0,\end{equation}
where $\Delta, \Scal, {\rm Rc}$ denote respectively the Laplacian, scalar curvature and Ricci tensor of the underlying K\"ahler metric $(g_K^T, \omega_K^T)$ on $S$, $\al=\al^T$ is the closed primitive (and therefore harmonic) $(1,1)$-form, and the norms are the tensorial norms with respect to $g_K^T$. We will more generally study the following geometric problem on a compact K\"ahler manifold or orbifold $(M^{2n}, J)$ in any complex dimension $n\geq 2$: Fix deRham classes $\Ome$ and $\A$ in $H^{1,1}(M, \R)$, such that $2\pi\Ome=[\omega_0]$ is a K\"ahler class  and $\A \cdot \Ome^{n-1}=0$.  We are then looking for a K\"ahler  metric $\omega_{\varphi} = \omega_0 + dd^c \varphi$ in $2\pi \Ome$ such that
\begin{equation}\label{PDEB} \Delta_{\varphi}\Scal_{\varphi} + \frac{\Scal_{\varphi}^2}{2} -\|{\rm Rc}_{\varphi}\|_{\varphi}^2-\|\al_{\varphi}\|_{\varphi}^2 =c_{\Ome, \A}, \qquad \Scal_{\varphi}>0,\end{equation}
where $c_{\Ome, \A}=2n(n-1)\left(\frac{(c_1(M)^2+ \A^2)\cdot \Ome^{n-2})}{\Ome^{n}}\right)$ is a topological constant obtained by integrating \eqref{PDEB} against $\omega_{\varphi}^n$,  
the subscript $\varphi$ denotes the geometric quantities with respect to the K\"ahler structure $(g_{\varphi}, \omega_{\varphi})$ on $M$,  and $\al_{\varphi}$ is the harmonic representative of the deRham class $2\pi \A$ with respect to $g_{\varphi}$.
Notice that \eqref{PDEB} is a 6th order 
PDE formulated in terms of the unknown $\omega_0$-relative K\"ahler potential function $\varphi$ on $M$. Even though \eqref{PDEB} is relevant to the BHE problem described above only when  $n=2$ and $c_{\Ome, \A}=0$, the motivation for studying the more general problem \eqref{PDEB} is threefold.  First, in complex dimensions $3$ and $4$,   letting  $\A=0$ and $c_{\Ome, \A}=0$, \eqref{PDEB} is the equation of motion appearing in the supersymmetric ${\rm AdS}_3 \times Y_7$ and ${\rm AdS}_2 \times  Y_9$ solutions of type ${\rm IIB}$ and ${\rm D = 11}$ supergravity, introduced  in \cite{GK} and further studied in \cite{CGMS}. This PDE is also embedded in the so-called \emph{GK geometry} formulated in any complex dimension $n \geq 3$ in \cite{GK}. Our results here apply to the quasi-regular geometries in that theory. Second, still assuming $\A=0$, \eqref{PDEB} is  an instance  of the PDE associated to  Dervan's \emph{Z-charge} (compare with \cite[Example~2.6]{DH} for $j:= (n-2)$). Finally, on a Fano variety  $(M^{2n}, J)$, taking $\Ome=c_1(M)$ and $\A=0$, solutions to \eqref{PDEB} correspond to the fixed points of the parabolic problem introduced in \cite[Eq (10.1)]{CT} and further studied in \cite{SW}, thus extending the notion of {K\"ahler--Einstein metric}.  Our general results in this paper apply to these cases too.

Loosely stated, our first main result shows that the LHS of \eqref{PDEB} is the momentum map of an  infinite dimensional formal K\"ahler manifold, similar to the interpretation  given by Fujiki~\cite{fujiki} and Donaldson~\cite{donaldson} of the scalar curvature of a K\"ahler metric.
\begin{thm}\label{thm:main1}(see Theorem~\ref{thm:GIT}) The solutions of \eqref{PDEB} correspond to the zeroes of the momentum map of a formal infinite dimensional K\"ahler manifold.
\end{thm}
In \cite{DH}, the authors gave a general method of constructing a closed $(1,1)$-form on a larger formal complex manifold, such that the zeroes of the  (formal) momentum map correspond to the PDE associated to the \emph{Z-charge}. This yields a similar momentum map interpretation of \eqref{PDEB} in the case $\A=0$. While we do not claim that the symplectic structure we construct here coincides with the closed form of \cite{DH}, we stress that the explicit definition we give allows us to verify that it is actually \emph{positive definite} on the subspace of K\"ahler structures with positive scalar curvature.

Similarly to the cscK case, the above result allows us to define a Mabuchi functional whose critical points are the solutions of \eqref{PDEB} (see Definition~\ref{d:Mabuchi}), a Futaki invariant  which obstructs the existence of solutions of \eqref{PDEB} (see Proposition~\ref{p:Futaki}), and a Calabi functional whose critical points extend \eqref{PDEB} to the Futaki unobstructed notion of \emph{extremal} solutions (see Definition~\ref{d:extremal} and Lemma~\ref{l:Calabi}). We use this familiar setup to obtain a LeBrun--Simanca type openness result for the existence of extremal solutions with respect to the cohomology classes $\Ome$ and $\A$ (see Proposition~\ref{p:LeBrun-Simanca}).  The positivity of the formal symplectic form is crucially used in order to establish the convexity of the Mabuchi functional corresponding to \eqref{PDEB} along smooth K\"ahler geodesics with \emph{positive scalar curvature}. This also allows us to obtain (similarly to the cscK case)  the following Calabi--Futaki--Lichnerowicz--Matsushima type obstruction theorem:
\begin{cor}(see Propositions~\ref{p:CLM-full} and~\ref{p:Futaki})\label{c:obstruction} If \eqref{PDEB} admits a solution $\omega_{\varphi}$, then the group $\Aut(M, J)$ of complex automorphisms of $(M,J)$ is reductive and the connected component of identity of the isometry group of $\omega_{\varphi}$ is a maximal connected compact subgroup of $\Aut(M, J)$. Furthermore, the Futaki invariant of the Lie algebra of $\Aut(M, J)$ associated to \eqref{PDEB} vanishes.
\end{cor}
Motivated by the cscK case, one would expect to find a more general obstruction to the existence of a solution of \eqref{PDEB}, expressed as a stability condition on K\"ahler test configurations of $(M, \Ome)$. To this end, one needs to show  that: (a) the derivative of the Mabuchi functional associated to \eqref{PDEB} has a well-defined limit along any ray of smooth K\"ahler metrics compatible with the given K\"ahler test configuration, and (b) the limit is non-negative should a solution of \eqref{PDEB} exist. We show in the Appendix that (a) holds for K\"ahler test configurations with  orbifold singularities, noting that the case when $\A=0$ and the central fibre is smooth already follows from the results in \cite{Dervan}. As for (b), major new challenges arise if one  tries to use the approach  \cite{BB} in the cscK case:  first, it is much harder to obtain a possible extension of the Mabuchi energy associated to \eqref{PDEB} to the space of weak $C^{1,1}$ K\"ahler potentials, and second, the convexity of the Mabuchi energy only holds along geodesics of smooth K\"ahler metrics with positive scalar curvature. Nevertheless, (b) holds true in the K\"ahler--Einstein Fano case with $\Ome=c_1(M), \A=0$,  by a result in \cite{SW}.

\bigskip
Turning back to our initial motivation to study compact BHE $3$-folds, we use Corollary~\ref{c:obstruction} in conjunction with classification results of ruled complex surfaces in order to establish the following rigidity result:
\begin{thm}\label{thm:main2} Let $(N, J_N, g_N, \omega_N)$ be a  compact complex 3-dimensional non-K\"ahler  Hermitian manifold whose  K\"ahler form $\omega_N$ is pluriclosed  and whose Bismut Ricci form is identically zero. Suppose that the foliation $\VV$ in Theorem~\ref{thm:reduction} is regular, i.e. the Killing vector fields $V,J_NV$ generate a \emph{free} $2$-dimensional torus  action on $N$. Then,  the $(1,1)$ Chern--Bott cohomology  of $(N, J_N)$ is $2$-dimensional  if and only if   $(N, g_N, J_N, \omega_N)$  is isometrically covered by a Samelson geometry \eqref{Samelson-3D}.
\end{thm}

Finally, we address the existence of solutions of \eqref{Box} on certain toric orbifold complex surfaces,  via the explicit orthotoric ansatz \cite{ACG, ACG-crelle}.  As a matter fact, such solutions appear already in a work by Couzens--Gauntlett--Martelli--Sparks~\cite[Sec.4.2]{CGMS} on supersymmetric ${\rm AdS}_3 \times Y_7$ type ${\rm IIB}$ supergravity,  where explicit complex $3$-dimensional solutions of \eqref{Box} with $\al=0$ are constructed as a product of a $2$-dimensional orthotoric K\"ahler orbifold surface $S_{\rm a, b, c}$ satisfying \eqref{Box}  and  a flat elliptic curve. In their construction, ${\rm a,b, c}$ are co-prime integers so chosen that $S_{\rm a,b,c}$ is the orbit space of a CR vector field $\xi$ on a $5$-dimensional CR manifold $\mathcal{L}^{\rm a,b,c}=(Y, \DD, J)$ diffeomorphic to $S^2 \times S^3$,  such that $\xi$ is transversal at each point to $\DD$, $(\DD, J)$  is endowed with a $\xi$-transversal orthotoric K\"ahler structure $(g^T, \omega^T)$ which satisfies \eqref{PDE} with $\al^T=0$, and the connection $1$-form $\eta_{\xi}$  of $\xi$ has curvature $c \rho(\omega_K^T)$ for a non-zero constant $c$. If we take $N=S^1\times Y$ and let $V$ be the circle generator of the $S^1$-factor whereas $J_N V:=-\frac{1}{c}\xi$, in view of Theorem~\ref{thm:reduction} and its inverse alluded to above, we obtain the following existence result:
\begin{thm}\label{thm:main3} 
$N=S^1\times S^2 \times S^3$ admits an infinite countable family of quasi-regular non-K\"ahler Bismut Ricci flat pluriclosed Hermitian metrics  with $2$-dimensional  $(1,1)$ Chern-Bott cohomology group, and which have transversal K\"ahler structure of non-constant scalar curvature. 
\end{thm}

We further investigate the existence of explicit orthotoric solutions of \eqref{Box} in the case when $\A \neq 0$. It turns out (see Proposition~\ref{p:othotoric}) that they exist on large families of complex toric  orbifold surfaces whose Delzant polytopes cover all rational compact convex quadrilaterals with no parallel  facets. Using this construction and a key fact established in \cite{ACGL} that each such orbifold can be realized as a Levi--K\"ahler quotient of the co-dimension $2$ CR manifold $S^3\times S^3 \subset \C^4$, we establish the following
\begin{thm}\label{thm:main4} $N=S^3\times S^3$ admits an infinite countable family of quasi-regular non-K\"ahler Bismut Ricci flat pluriclosed Hermitian structures with $2$-dimensional $(1,1)$ Chern-Bott  cohomology group, which have transversal K\"ahler structure of non-constat scalar curvature and are thus not isometric to the $SU(2)\times SU(2)$ Samelson's geometry.  
\end{thm}
To the best of our knowledge, Theorems~\ref{thm:main3} and \ref{thm:main4} provide the first known examples of non-K\"ahler compact BHE 3-folds which are not locally isometric to a Samelson geometry. These results also show that the hypotheses  implying the rigidity results in Theorems~\ref{thm:reduction} and \ref{thm:main2} are sharp.

\section{Quasi-regular BHE 3-folds: reduction to a 6th order PDE}

\subsection{The transversal K\"ahler geometry of a BHE and the reverse construction} 
\begin{defn}\label{d:BHE} Let $(N, J_N, g_N)$ be a hermitian manifold with fundamental $2$-form $\omega_N = g_N J_N$. Attached to $(g_N, J_N)$ is the hermitian connection
\[ \langle \nabla^{B}_X Y, Z) := \langle \nabla^{\rm LC}_X,Y, Z) - \frac{1}{2}(d^c\omega_N)_{X, Y, Z},\]
called the \emph{Bismut connection}, which in turn defines a closed $2$-form
\[ \rho_N^B(X, Y) := \sum_{i=1}^n\left\langle R^{\nabla^B}_{X, Y}, e_i, J_Ne_i\right\rangle \in 2\pi c_1(N, J_N),\]
where $R^{\nabla^B}$ is the curvature tensor of $\nabla^B$;  
we say that $(g_N, J_N)$ is \emph{pluriclosed} if $dd^c\omega_N=0$ and \emph{Bismut Ricci flat} if $\rho_N^{B} \equiv 0$. If both conditions holds simultaneously, we refer to $(N, g_N, J_N)$ as a \emph{Bismut-Hermitian-Einstein} (BHE for short).
\end{defn}
When $(N,J_N)$ is  a compact complex manifold which admits pluriclosed Hermitian metrics, a BHE Hermitian metric $g_N$ on $(N, J_N)$ is a fixed point of the \emph{pluriclosed flow}~\cite{PCF}, and thus defines a stationary \emph{soliton}  for the flow. The theory from \cite{SU2} then tells us that there exists a smooth function $f$ and an associated Killing vector field $V:= g_N^{-1}(\eta_N + df)$, where $\eta_N := J_N (d^* \omega_N)$ is the Lee form of $(g_N, J_N)$. Furthermore, assuming $(g_N, J_N)$ is not K\"ahler, it was shown in \cite{ABLS} that $V$ is $\nabla^B$-parallel (and hence of constant non-zero norm) and that $J_NV$ is also Killing. Using this,  it was established in \cite[Thm.1.1]{ABLS} (see Theorem~\ref{thm:reduction} in the Introduction) that each non-K\"ahler BHE complex 3-fold  $(N, J_N, g_N)$ admits a $2$-dimensional foliation $\VV={\rm span}(V, J_NV)\subset TN$, a $\VV$ transversal codimension $2$ CR structure $(\mathcal{H}, J)$ with a compatible $\VV$-transversal complex 2-dimensional K\"ahler structure $(M^T, J^T, \omega_K^T, \al^T)$ satisfying \eqref{PDE}.  Conversely, the BHE Hermitian  manifold $(N, g_N, J_N)$ can be locally reconstructed from the data $(M^T, J^T, \omega_K^T, \al^T)$ as follows:  Let $N$ be a principal $\R^2$-bundle over $M^T$, $V, W$ be the fundamental vector fields for the $\R$-actions,   $\eta=(\eta_V, \eta_{W})$ be the $\R^2$-valued connection $1$-form with curvature $\left(\al^T, -\rho(\omega_K^T)\right)$,   and $\mathcal{H}= {\rm ann}(\eta)\subset TN$ be the $4$-dimensional horizontal sub-bundle of $\eta$ so that, at each point, 
\begin{equation}\label{transversal}TN = \VV \oplus \mathcal{H}.\end{equation}
We can lift $J^T$ to define an almost CR structure $J$ on $\mathcal{H}$. As the curvature of $\eta$ is the pullback to $N$  of a  $(1,1)$-form  on $(M^T, J^T)$, $J$ is in fact CR-integrable. We can then extend $J$ to an integrable almost complex structure $J_N$ on $N$ by 
letting
\[ J_N(V):=W, \qquad J_N(W):= -V.\]
Furthermore, 
\begin{equation}\label{inverse}
\begin{split}
g_N &= \left(\frac{\Scal(\omega_K^T)}{2}\right)g_K^T + \eta_V\otimes \eta_V + \eta_{W} \otimes \eta_{W}
\end{split}
\end{equation}
defines a $J_N$-compatible Hermitian metric and \cite[Thm.1.1]{ABLS} tells us that $(N, g_N, J_N)$ is BHE with transversal K\"ahler structure $(M^K, g_K^T, \omega_K^T, \al^T)$.
 
\begin{defn} We say that a compact non-K\"ahler BHE 3-fold $(N,g_N, J_N)$ is \emph{quasi-regular} if $(V, J_N V)$ generate a $2$-dimensional real  torus  $\T^2_{\VV}$ of isometries of $(g_N,J_N)$. In this case, by \eqref{transversal}, $\T^2_{\VV}$ acts semi-freely on $N$, and $M^T=N/\T^2_{\VV}$ is a compact orbifold endowed with a K\"ahler orbifold structure $(g_K^T, J^T, \omega_K^T)$ satisfying \eqref{PDE}, or equivalently \eqref{Box}. We say that $(N, g_N, J_N)$ is \emph{regular} if, furthermore, $M^T$ is  smooth.
\end{defn}
In the quasi-regular case, the quotient map $\pi: N \to M^T$ is a principal $\T^2$ (orbifold) bundle and the inverse construction described above allows us to construct a global BHE metric $(g_N, J_N)$ on $N$ from a K\"ahler structure $(g_K^T,\omega^T_K)$ on $M^T$ satisfying  \eqref{PDE} as soon as the  ${\rm span}_{\R}([\omega_K^T], [\theta^T]) \subset H^2(M, \R)$  admits an integer basis. 

\subsection{Regular BHE's from compact K\"ahler surfaces}
We shall assume in this section that $S:=(M^T, J^T)$ is a smooth compact complex surface and we shall look for a K\"ahler metric $\omega$ and a closed primitive $(1,1)$-form $\al$ on $S$, satisfying \eqref{PDE}. We shall assume further that the real span of the deRham classes $\A:=\frac{1}{2\pi}[\al^T]$ and $c_1(S)=\frac{1}{2\pi} [\rho(\omega_K^T)]$ admits an integer basis, so that the inverse construction \eqref{inverse} will give rise to a \emph{compact} non-K\"ahler BHE 3-fold,  obtained as a (global) principal torus bundle over $S$ rather than just a local $\R^2$-bundle. We shall refer to this special case as a \emph{regular} BHE.

\smallskip
In this setup, we shall look for K\"ahler metrics $\omega \in 2\pi \Ome$ within a fixed K\"ahler deRham class $\Ome$ on $S$. Integrating \eqref{PDE} and using that $\al$ is $\omega$-primitive, we  obtain that $\A$ and $\Ome$ must satisfy
\begin{equation}\label{topology}
\Ome\cdot \A =0, \qquad   \A^2 + c_1^2(S)=0.
\end{equation}
We thus reduced the search of $(\omega, \al)$ on $S$  satisfying \eqref{PDE} to a 6th-order non-linear PDE problem: Let $\omega_0 \in 2\pi \Ome$ be a fixed reference K\"ahler metric and denote by $\mathcal{H}_{\omega_0}$ the space of all smooth functions $\varphi$ on $S$ such that $\omega_{\varphi}:= \omega_0 + dd^c \varphi >0$ is a K\"ahler metric in $2\pi \Ome$. Then we want to find $\varphi \in \mathcal{H}_{\omega_0}$ such that
\begin{equation}\label{PDEA} \frac{1}{2}dd^c \Scal_{\varphi}\wedge \omega_{\varphi} =  \al_{\varphi}\wedge \al_{\varphi} + \rho_{\varphi} \wedge \rho_{\varphi}, \qquad \Scal_{\varphi}>0\end{equation}
where $\Scal_{\varphi}$ and $\rho_{\varphi}$ denote, respectively, the scalar curvature and Ricci form of the K\"ahler metric $\omega_{\varphi}$, and $\al_{\varphi}$ stands for the harmonic representative of $\A$ with respect to $\omega_{\varphi}$. The condition $\Scal_{\varphi}>0$ implies $H^{2,0}(S)=\{0\}=H^{0,2}(S)$ (see \cite{Yau}), so by Hodge theory, $\al_{\varphi}$ is type $(1,1)$; as the $\omega_{\varphi}$-trace of the harmonic $2$-form $\al_{\varphi}$ is constant, the topological condition $\Ome \cdot \A=0$ ensures that $\al_\varphi$ is $\omega_{\varphi}$-primitive, i.e. we have a solution of \eqref{PDE} compatible with our choice of deRham classes $\Ome$ and $\A$.

\subsection{Kodaira's classification of complex surfaces possibly admitting solutions to \eqref{PDEA}}
\begin{lemma}\label{rough-classification} Suppose $S$ is a compact K\"ahler surface which supports a K\"ahler metric $\omega$ and a $\omega$-primitive closed $(1,1)$-form $\al$ satisfying \eqref{PDE}. Then $S$ must be one of the following
\begin{enumerate}
\item  $S=\PP^1\times \PP^1$ or is a Hirzebruch complex surface $S=\PP(\cO_{\PP^1} \oplus \cO_{\PP^1}(k)), \, k\geq 1$;
\item  $S$ is a minimal ruled complex surface of genus $1$, i.e. is the projectivisation $\PP(E)$ of a rank 2 holomorphic vector over an elliptic curve;
\item $S$ is a complex surface obtained from $\PP^1 \times \PP^1$ or a Hirzebruch surface by blowing up at most $8$ points.
\end{enumerate}
\end{lemma}
\begin{proof} Using $\Scal(\omega)>0$, it follows that the Kodaira dimension of $S$ is negative (see \cite{Yau}).
By the Enriques--Kodaira classification (see e.g. \cite[Ch.VI]{BPV}), $S$ must be either a rational surface or a ruled complex surface of genus $\geq 1$.  The topological constraint in \eqref{topology} yields $c_1^2(S) = - \A^2 \geq 0$ (as $\A$ is represented by an $\omega$-primitive form) with equality iff $\al=0$ in \eqref{PDE}. This implies that the only irrational ruled complex surfaces that can possibly support solutions to \eqref{PDE} must be minimal and of genus $0$ (otherwise $c_1^2(S) <0$, see \cite{BPV}), so we get (2). 

As for rational complex surfaces satisfying $c_1^2(S) \geq 0$, they are precisely the surfaces  covered by the cases (1) and (3) of the Lemma (see \cite{BPV}),   or $\PP^2$. However, as $b_-(\PP^2)=0$,  any solution of  \eqref{PDE} on $\PP^2$ must satisfy $\al =0$, hence by \eqref{topology}, $c_1(S)^2=0$, a contradiction. \end{proof}
\subsection{Higher-dimensional setup} 
We can extend the problem on a compact K\"ahler manifold or orbifold $(M, J)$ of arbitrary complex dimension  $n \geq 2$,  endowed with a K\"ahler class $2\pi \Ome \in H^{2}(M, \R)$ and a deRham class $\A \in H^2(M, \R)$ satisfying
\begin{equation}\label{topo-n}
\begin{split}
    &\A\cdot \Ome^{n-1}=0, \qquad \A \cdot [\gamma] \cdot \Ome^{n-2}=0 \qquad \forall \gamma \in H^{2,0}(M, \C). \end{split}\end{equation}
We then seek for a K\"ahler metric $\omega_{\varphi} \in 2\pi \Ome$ on $(M, J)$ such that
\begin{equation}\label{PDEAn}
-\frac{1}{2} dd^c \Scal_{\varphi} \wedge \omega_{\varphi}^{[n-1]} +\left( \al_{\varphi} \wedge \al_{\varphi} + \rho_{\varphi}\wedge \rho_{\varphi}\right) \wedge \omega_{\varphi}^{[n-2]} = \left(\frac{c_{\Ome, \A}}{2}\right)\omega_{\varphi}^{[n]}, \qquad \Scal_{\varphi} >0,
\end{equation}
where $\al_{\varphi}$ is the harmonic representative of $2\pi \A$ with respect to $\omega_{\varphi}$,  and $c_{\Ome, \A}$ is the topological constant
\begin{equation}\label{topo-c} c_{\Ome, \A}:= 2n(n-1)\left(\frac{(c_1^2(M) + \A^2)\cdot \Ome^{n-2}}{\Ome^n}\right),\end{equation}
obtained by integrating \eqref{PDEAn} over $M$.
Again, the cohomological conditions \eqref{topo-n}  ensure that $\al_{\varphi}$ is a $\omega_{\varphi}$-primitive harmonic $(1,1)$-form.  
Dividing  \eqref{PDEAn}  by the  Riemannian volume form $\omega_\varphi^{[n]}$,  it can be  equivalently rewritten as \eqref{PDEB} from the Introduction.

If $(\Ome, \A)$ further satisfy the cohomological relation
\[ \A^2\cdot \Ome^{n-2} = -c_1(M)^2\cdot \Ome^{n-2}, \]
then $c_{\Ome, \A}=0$ and \eqref{PDEAn} reduces to \eqref{Box} from the introduction.
\begin{rmk} When $(M, J)$ is Fano and $\Ome=c_1(M), \A=0$, a prominent special solution of \eqref{PDEAn} is given by the K\"ahler-Einstein metrics on $M$. 
\end{rmk}

\section{A momentum map interpretation} In the case when $\A=0$ and the automorphism group of $(M, J)$ is discrete, it was observed in \cite{Dervan,DH} that \eqref{PDEAn} can be interpreted as a zero of a momentum map defined on a suitable infinite dimensional submanifold $\mathcal{J}_{\omega_0}$ of the Fr\'echet manifold $\mathcal{AC}_{\omega_0}$ of all almost-complex structures compatible with a fixed symplectic form $\omega_0$,  acted by the group of hamiltonian symplectomorphisms ${\rm Ham}_{\omega_0}$;  this extends foundational work by Donaldson~\cite{donaldson} and Fujiki~\cite{fujiki} in the case of K\"ahler metrics with constat scalar curvature. However, a caveat in the approach in \cite{Dervan, DH} is that it is unclear from their construction whether or not the closed ``symplectic'' form $\boldsymbol{\Omega}$ they introduce on $\mathcal{J}_{\omega_0}$ is K\"ahler (or even non-degenerate) with respect to the natural formal complex structure $\boldsymbol{J}$.  The k\"ahlerianity of $\boldsymbol{\Omega}$ turns out to  play a key r\^{o}le in our applications of this formal GIT setup, namely for establishing a Calabi--Lichnerowicz--Matsushima type obstruction theorem for the existence of solutions to \eqref{PDEAn},   as well as for proving the convexity of the Mabuchi energy associated to \eqref{PDEAn} along certain smooth geodesics in the space $\mathcal{H}_{\omega_0}$.  

Thus motivated, we will first present the momentum map setup for \eqref{PDEAn}, extending \cite{Dervan,DH} to the case $\A\neq 0$,   and we will obtain an \emph{explicit} formula 
for the relevant symplectic form $\boldsymbol{\Omega}$,  proving that \eqref{PDEAn} corresponds to a zero of the momentum map for the action of ${\rm Ham}_{\omega_0}$ on a formal \emph{K\"ahler} manifold $(\mathcal{J}_{\omega_0}^+, \boldsymbol{\Omega}, \boldsymbol{J})$. 
We refer to \cite[Ch. 9]{gauduchon-book} for a  detailed treatment of the background material we will need.

\smallskip
In what follows, we fix a K\"ahler metric $\omega_0 \in 2\pi \Ome$ on $M$ and denote by $J_0$ the underlying $\omega_0$-compatible, integrable almost complex structure.  We then consider the space $\mathcal{AC}_{\omega_0}$ of all almost complex structures $J$ on $M$ which are $\omega_0$-compatible, i.e. satisfy  that the tensor
\[ g_J :=-\omega_0J \]
is a riemannian metric on $M$. Clearly, $J_0\in \mathcal{AC}_{\omega_0}$. The space $\mathcal{AC}_{\omega_0}$ has a structure of a formal (Fr\'echet) manifold with tangent space at a point $J\in \mathcal{AC}_{\omega_0}$
\[ \boldsymbol{T}_J(\mathcal{AC}_{\omega}) =\left\{ \dot J \in C^{\infty}(X, {\rm End}(TX)) \, \Big| \, J\dot{J}= -\dot J J, \, \, g_{J}(\dot J, \cdot) =g_J(\cdot, \dot J)\right\}. \]
Furthermore, the  Cayley  transform  gives rise to a  bijective map  from $ \mathcal{AC}_{\omega_0}$  to an open subset of an infinite dimensional complex vector space, thus  leading to an integrable  almost-complex structure $\boldsymbol{J}$ on $\mathcal{AC}_{\omega_0}$ satisfying (see \cite[Prop.9.2.1]{gauduchon-book}):
\[\boldsymbol{J}_J(\dot J) := J \dot J. \]
With this understood, we introduce a $(1,1)$-form $\boldsymbol{\Omega}$ on $(\mathcal{AC}_{\omega_0}, \boldsymbol{J})$ by the formula
\begin{equation}\label{omega}
\begin{split}
\boldsymbol{\Omega}_J(\dot J_1, \dot J_2) :=& \, \frac{1}{4}\int_M \left\langle J\dot{J}_1, \dot{J_2}\right\rangle_J \Scal_J \,\omega_0^{[n]}  + \frac{1}{2}\int_M\left\langle (\delta_J J\dot J_1), (\delta_J \dot J_2)  \right\rangle_J\omega_0^{[n]}\\
& + 4\int_M \left\langle \mathbb{G}_J\left((\al_J J\dot J_1)^{\rm skew}\right), \left(\al_J \dot J_2\right)^{\rm skew}\right\rangle_J \, \omega_0^{[n]},
\end{split}
\end{equation}
where:
\begin{itemize}
 \item $\Scal_J$  denotes the \emph{hermitian} scalar curvature of the almost-K\"ahler structure $(g_{J}, J, \omega_0)$, i.e. the $\omega_0$-trace of the Ricci $2$-form $\rho_J\in 2\pi c_1(X, \omega_0)$ of $(g_J, J, \omega_0)$,  computed from the Chern connection $\nabla^{c, J} = \nabla^J -\frac{1}{2} J \nabla^J J$,  where $\nabla^J$ stands for the Levi--Civita connection of $g_J$;
 \item $\delta_J = (\nabla^J)^*$ is the co-differential with respect to the Levi-Civita connection $\nabla^J$ of $g_J$;
 \item $\mathbb{G}_J$ is the Green operator with respect to $g_J$, acting on $2$-forms;
 \item $\langle \cdot, \cdot\rangle_J$ is the induced inner product on tensors via $g_J$;
 \item $\al_J$ is the harmonic representative of $2\pi \A$ with respect to $g_J$;
 \item $\left(\cdot \right)^{\rm skew}$ stands for the skew-symmetric part of a $(2,0)$-tensor.
\end{itemize}
By its very definition, $\boldsymbol{\Omega}$ is $\boldsymbol{J}$-invariant and
\begin{equation}\label{g}
\begin{split}
\boldsymbol{g}_J(\dot J_1, \dot J_2) :=& -\boldsymbol{\Omega}_J(\boldsymbol{J}\dot J_1, \dot J_2) \\
=& \, \frac{1}{4}\int_M \left\langle \dot{J}_1, \dot{J_2}\right\rangle_J \Scal_J \,\omega_0^{[n]}  + \frac{1}{2}\int_M\left\langle (\delta_J \dot J_1), (\delta_J \dot J_2)  \right\rangle_J\omega_0^{[n]}\\
& + 4\int_M \left\langle \mathbb{G}_J\left((\al_J\dot J_1)^{\rm skew}\right), \left(\al_J \dot J_2\right)^{\rm skew}\right\rangle_J \, \omega_0^{[n]}
\end{split} 
\end{equation}
defines a positive-definite formal Riemannian metric tensor on the open Fr\'echet subspace 
\[ \mathcal{AC}^+_{\omega_0}:= \{ J \in \mathcal{AC}_{\omega_0}\, | \, R_J >0 \, \mathrm{on} \, M\} \subset \mathcal{AC}_{\omega_0}.\]
Thus, $(\mathcal{AC}^+_{\omega_0}, \boldsymbol{J}, \boldsymbol{g})$ is a formal Hermitian manifold with fundamental form $\boldsymbol{\Omega}$. Notice that the group ${\rm Ham}_{\omega_0}$ acts on $\mathcal{AC}^+_{\omega_0}$, preserving the Hermitian structure $(\boldsymbol{g}, \boldsymbol{J}, \boldsymbol{\Omega})$.

\bigskip For the applications needed for this article, we shall further restrict the formal Hermitian structure $(\boldsymbol{g}, \boldsymbol{J}, \boldsymbol{\Omega})$ to a formal Fr\'echet complex submanifold $\mathcal{J}^+_{\omega_0} \subset \mathcal{AC}^+_{\omega_0}$, leaving it an open question to verify whether or not the further geometric properties of $(\boldsymbol{g}, \boldsymbol{J}, \boldsymbol{\Omega})$ we establish on $\mathcal{J}^+_{\omega_0}$ do actually  hold on $\mathcal{AC}^+_{\omega_0}$.

To this end, we define the subspace
\[ \mathcal{J}_{\omega_0} :=\left\{J \in \mathcal{AC}_{\omega_0}  \, \Big| \, \exists \phi \in {\rm Diff}_0(M) \, \, \mathrm{s.t.} \, \,  \phi \cdot J := \phi_* \dot J \phi_*^{-1} = J_0\right\} \subset \mathcal{AC}_{\omega_0}.\]
It is shown in \cite[Prop.9.1.1]{gauduchon-book} that $\mathcal{J}_{\omega_0}$ is a (formal) complex submanifold of $\mathcal{AC}_{\omega_0}$ whose tangent space is
\[ \boldsymbol{T}_J(\mathcal{J}_{\omega_0}) = \{ -\mathcal{L}_Z J \, \, | \, \, Z \in \mathfrak{sp}_{\omega_0} + J\mathfrak{ham}_{\omega_0}\} \subset \boldsymbol{T}_J(\mathcal{AC}_{\omega_0}),\]
where $\mathfrak{sp}_{\omega_0}$ stands for the infinite dimensional Lie algebra of symplectic vector fields on $(M, \omega_0)$ and $\mathfrak{ham}_{\omega_0} < \mathfrak{sp}_{\omega_0}$ is the sub-algebra of Hamiltonian symplectic vector fields:
\[ \mathfrak{sp}_{\omega_0} := \{ X \in C^{\infty}(M, TM) \, | \, \mathcal{L}_X \omega_0 =0\}, \qquad \mathfrak{ham}_{\omega_0} : = \{ Y= -\omega_0^{-1}(df), \ | \, f\in C^{\infty}(M)\}. \]
Note that any tangent vector $-(\mathcal{L}_{Z} J) \in \boldsymbol{T}_J(\mathcal{J}_{\omega_0})$ is the velocity at $t=0$ of a smooth curve $J(t) \in \mathcal{J}_{\omega_0}$ with $J(0)=J$. This follows by writing $Z= X - J\omega_0^{-1}(df)$ and letting $J(t)= \phi_t^X \circ \psi_t^f \cdot J$, where $\phi^X_t \in {\rm Diff}_0(M)$ is the flow of the symplectic field $X$ and $\psi_t^f$ is the Moser lift of $\omega_t = \omega_0 + t d Jd f$.

We shall  further restrict to the open subset
\[\mathcal{J}_{\omega_0}^+ := \{ J \in \mathcal{J}_{\omega_0} \, | \, \Scal_J >0 \} \subset \mathcal{AC}^+_{\omega_0}. \]
What will be important for us is that any $J\in \mathcal{J}^+_{\omega_0}$ is, in fact,  an integrable almost-complex structure, thus giving rise to a K\"ahler structure $(g_J, J, \omega_0)$ and $J_0\in \mathcal{J}^+_{\omega_0}$. The pullback of the Hermitian structure $(\boldsymbol{g}, \boldsymbol{J}, \boldsymbol{\Omega})$ to $\mathcal{J}^+_{\omega_0}$ thus makes it a formal Hermitian manifold  on which  ${\rm Ham}_{\omega_0}$ acts isometrically. 

\begin{thm}\label{thm:GIT} In the setup above, assume that $H^1(M, \R)=\{0\}$. Then $\boldsymbol{\Omega}$ is a closed $2$-form on $\mathcal{J}^+_{\omega_0}$ and 
 ${\rm Ham}_{\omega_0}$ acts in a hamiltonian way on $(\mathcal{J}^+_{\omega_0}, \boldsymbol{\Omega})$ with momentum map  \[ \boldsymbol{\mu}(J):= -\left(-\frac{1}{2}\left(dd^c_J \Scal_J\right)\wedge \omega_0^{[n-1]} + \rho_J^2 \wedge\omega_0^{[n-2]} + \al_J^2\wedge \omega_0^{[n-2]}-\frac{c_{\Ome, \A}}{2}\omega_0^{[n]}\right) \in (C^{\infty}(M)/\R)^{*}.\]
\end{thm}
\begin{rmk}
The assumption $H^1(M, \R)=\{0\}$ is used only to make sure that the tangent space $\boldsymbol{T}_J(\mathcal{J}_{\omega_0})=\{-(\mathcal{L}_V J) \, | \, V= \nabla^J f + J\nabla^J h\}$ is generated by \emph{hamiltonian} vector fields.  One can instead state the result for the restriction of the $2$-form $\boldsymbol{\Omega}$ to the $\boldsymbol{J}$-invariant foliation $\boldsymbol{D}_J = \{-(\mathcal{L}_V J) \, | \, V= \nabla^J f + J\nabla^J h\} \subset \boldsymbol{T}_J(\mathcal{J}_{\omega_0})$,  where the r\^{o}le of  ``symplectic manifold'' should be played by a ``leaf'' of $\boldsymbol{D}$, i.e. the ``complexified orbit'' for the holomorphic action of ${\rm Ham}_{\omega_0}$ on $\mathcal{J}_{\omega_0}$. It turns out that this formal interpretation is enough for our applications. 
\end{rmk}
The proof of Theorem~\ref{thm:GIT} will be established in three steps, corresponding to the three Lemmas below, in which we \emph{do not assume} $H^1(M, \R)=\{0\}$. The first proves a general \emph{momentum map property} of $\boldsymbol{\mu}$, without assuming that $\boldsymbol{\Omega}$ is closed:
\begin{lemma}\label{l:mentum-property}
Let $h$ be a smooth function on $(M, \omega_0)$,  giving rise to a hamiltonian vector field $X_h=-\omega_0^{-1}(dh)$,  and  $J(t)\in \mathcal{AC}_{\omega_0}$  a smooth curve such that $J=J(0)$ is integrable. Denote by $\dot{J} := \frac{d}{dt}_{t=0} J(t) \in \boldsymbol{T}_J(\mathcal{AC}_{\omega_0})$ the the corresponding tangent vector.  Then the following identity holds:
\begin{equation}\label{momentum-property}-\frac{d}{dt}_{|_{t=0}}\int_M h \boldsymbol{\mu}(J(t))  = \boldsymbol{\Omega}_J((-\mathcal{L}_{X_h}J), \dot J). \end{equation}
\end{lemma}
\begin{proof} We have 
\[
\begin{split}
  -2 \left(\frac{d}{dt}_{|_{t=0}}\, \int_M h \boldsymbol{\mu}(J(t))\right) =& \frac{d}{dt}_{|_{t=0}}\,\int_M h \left( (\Delta_{J(t)}\Scal_{J(t)}) \omega_0^{[n]} +2\rho_{J(t)}^2\wedge\omega_0^{[n-2]} + 2\al_{J(t)}^2\wedge \omega_0^{[n-2]}\right)\\
  =& \frac{d}{dt}_{|_{t=0}}\,\int_M (\Delta_{J(t)} h) \Scal_{J(t)}\omega_0^{[n]} +2h\rho_{J(t)}^2\wedge\omega_0^{[n-2]} + 2h\al_{J(t)}^2\wedge \omega_0^{[n-2]}
 \end{split} \] 
 and we will next develop each of the three terms of the last line separately. To this end, we shall use the variational formulae for $\dot \Scal_J, \dot \rho_J, \dot \Delta_J$ and $\dot\al_J$ which we detail first. By \cite[Prop.9.5.2]{gauduchon-book}, we know
\begin{equation}\label{ak-variation}
\dot{R}_J=-\delta_J J\delta_J\dot{J},\quad\dot{\rho}_J=-\frac{1}{2}d\delta_J\dot{J}. 
\end{equation}
The variation of the Laplace operator on 
functions is given by
\begin{equation}\label{variation-laplace}
\dot{\Delta}_J(\psi)=  -\Lambda_{\omega_0}(d\dot J d \psi) =\delta^c_J(\dot{J}d\psi),
\end{equation}
where $\dot J$ acts on $1$-forms by the opposite of its dual, $\delta_J^c =J\delta_J J$ is the adjoint of $d^c_J=JdJ^{-1}$,  and for the final equality we have used  the general commutation identities (see e.g. \cite{gauduchon-book}):
\begin{equation}\label{commuting} [\Lambda_{\omega_0}, d]= -\delta^c_J, \qquad [\Lambda_{\omega_0}, d_J^c]= \delta_J.\end{equation}
To compute $\dot \al_J$, notice that  $\dot \al_J$ is $d$-exact (as $[\al_J] = 2\pi\A$) so by the Hodge theorem 
\[ \dot \al_J = -\mathbb{G}_J(d\dot \delta_J \al_J)= - d \mathbb{G}_J(\dot \delta_J \al_J). \]
Using \eqref{commuting}, 
any closed primitive (1,1)-form  $\Psi$ relative to $J$ satisfies
\[
\begin{split}
 \dot{\delta}_J \Psi= & \frac{d}{dt}_{|_{t=0}} \left([\Lambda_{\omega_0}, d_J^c]\Psi \right)
    =[\Lambda_{\omega_0},\dot{d}_J^c]\Psi
=\Lambda_{\omega_0}(\dot d^c_J\Psi)
=\Lambda_{\omega_0} \left(\frac{d}{dt}_{t=0} \left(JdJ^{-1}\Psi \right)\right) \\
=&\Lambda_{\omega_0} \left(Jd \left(\frac{d}{dt}_{t=0}\left(J^{-1}\Psi\right)\right)\right)= -2\Lambda_{\omega_0} \left(J d (\Psi J \dot J)^{\rm skew}\right).
\end{split}
\]
For the last equality, we used that $\Psi$ is of type $(1,1)$,  so that for any tangent vectors  $a,b$ 
\[
\begin{split}
   \frac{d}{dt}  (J^{-1}\Psi)(a,b)=& \frac{d}{dt}(\Psi(Ja,Jb))
=\Psi(\dot{J}a,Jb)+\Psi(Ja,\dot{J}b)\\
=&-\Psi(J\dot{J}a,b)+\Psi(J\dot{J}b,a)\\
=-2(\Psi J\dot{J})^{\rm skew}(a,b).
\end{split}
\]
As $(\Psi J \dot J)^{\rm skew}$ is a $2$-form of type $(2,0)+ (0,2)$ wrt $J$, 
\[J d (\Psi J \dot J)^{\rm skew}= - J d J(\Psi J \dot J)^{\rm skew}= -J d J^{-1}(\Psi J \dot J)^{\rm skew} = - d_J^c (\Psi J \dot J)^{\rm skew}. \]
By \eqref{commuting} again, we conclude
\[\dot \delta_J \Psi =  2\Lambda_{\omega_0} d^c_J (\Psi J \dot J)^{\rm skew} = 2[\Lambda_{\omega_0}, d_J^c](\Psi J \dot J)^{\rm skew} = 2\delta_J (\Psi J \dot J)^{\rm skew}.\]
Taking $\Psi=\al_J$, we infer
\begin{equation}\label{variation-alpha}
\begin{split}
    \dot{\al}_J=&-2d\delta\mathbb{G}(\al_J J\dot{J})^{\rm skew}.
\end{split}
\end{equation}
We now return to the computation of $\dot{\boldsymbol{\mu}}$.
The first term is 
 \[\begin{split}
  \frac{d}{dt}_{|_{t=0}}\,\int_M (\Delta_{J(t)} h) \Scal_{J(t)}\omega_0^{[n]} &=\int_M \dot{\Delta}_J (h)\Scal_J  + (\Delta_J h)\dot\Scal_J \omega_0^{[n]}  \\
  &=\int_M \delta^c_J(\dot J dh) \Scal_J -(\Delta_J h) (\delta^J J \delta^J \dot J) \omega_0^{[n]} \\
  &= \int_M \left(-\left\langle \dot J(dh), d_J^c\Scal_J\right\rangle_J +\left\langle (\delta_J \dot J), d^c(\Delta_J h)\right\rangle_J\right)\omega_0^{[n]}.
 \end{split}\]
 In the above computation, we have used \eqref{variation-laplace}  and 
 \eqref{ak-variation} in the second line,   and integration by parts  taking into account  that $\dot{J}$ is $g_{J}$ self-adjoint,  in the third line.

 The second term is
 {
 \[
 \begin{split}
 & \frac{d}{dt}_{|_{t=0}}\int_M 2h \rho_{J(t)}^2 \wedge \omega_0^{[n-2]}=-2\int_M h(d\delta_J\dot{J}) \wedge \rho_J\wedge \omega_0^{[n-2]} \\
    &=\int_M\left(-h\left(\Lambda_{\omega_0}d\delta_J\dot{J}\right) \Scal_J+2\left\langle d\delta_J\dot{J} , h\rho_J\right\rangle_J \right)\omega_0^{[n]}=\int_M\left(h (\delta_J^c\delta_J\dot{J}) R_J+2\left\langle \delta_J\dot{J} , \delta_J(h\rho_J)\right\rangle_J\right)\omega_0^{[n]}\\
   &=\int_M\left(\left\langle\dot{J},\nabla^J (d_J^c R_Jh)\right\rangle_J   -h\left\langle \delta_J\dot{J} , d_J^cR_J\right\rangle_J -2\left\langle \delta_J\dot{J} , \rho_J(\nabla^{J} h)\right\rangle_J\right)\omega_0^{[n]}\\
   &=\int_M\left(\left\langle\dot{J},\nabla^J d_J^ch\right\rangle_J R_J +\cancel{\left\langle\dot{J},\nabla^{J} d_J^cR_J\right\rangle_J h}-\cancel{\left\langle\dot{J},\nabla^J(hd_J^cR_J)\right\rangle_J} \right)\omega_0^{[n]}\\
   & \, \, \, \, +\int_M\left(\left\langle \dot{J}, dR_J\otimes d_J^ch\right\rangle_J-\cancel{\left\langle \dot{J}, dh\otimes d_J^cR_J\right\rangle_J}\right)\omega_0^{[n]}\\
   &\, \, \, \,+2\int_M\left\langle (\delta_J\dot{J}) , \delta_J(\nabla^{J,-}d_J^ch)\right\rangle_J\omega_0^{[n]}-\int_M\left\langle d^c_J(\Delta_J h), \delta_J\dot{J}\right\rangle_J \omega_0^{[n]}\\
&=\int_M\left\langle\dot{J},\nabla^J d_J^ch\right\rangle_J R_J\omega_0^{[n]} + 2\int_M\left\langle \delta_J\dot{J} , \delta_J\nabla^{J,-}d_J^ch\right\rangle_J\omega_0^{[n]}-\int_M\left\langle d^c_J(\Delta_J h),  \delta_J\dot{J}\right\rangle_J\omega_0^{[n]}\\
&\, \, \, \,+\int_M\left\langle \dot{J},  dR_J\otimes d^ch\right\rangle_J \omega_0^{[n]}.
\end{split}
\]
}
In the above sequence of equalities, we have used:
\begin{itemize}
    \item \eqref{ak-variation} for the first line,
    \item the general algebraic identity $\left(\Lambda_{\omega_{0}}\Phi\right)\left(\Lambda_{\omega_{0}}\Psi\right) - \left(\frac{\Phi\wedge \Psi \wedge \omega_{0}^{[n-2]}}{\omega_{0}^{[n]}}\right) = \langle \Phi, \Psi\rangle_{J}$ for  $(1,1)$-forms $\Phi, \Psi$ with respect to $J$, 
\eqref{commuting} and integration by parts on the second line;
\item the Ricci identity $\delta_J \rho_J = -\frac{1}{2}d^c_J \Scal_J$
on the third line; 
\item the Bochner identity (see \cite[Lemma 1.23.10]{gauduchon-book})
\[\delta_J\left(\nabla^{J,-}d^c_J h\right)=\frac{1}{2}\Delta_J( d^c_Jh)-\rho_J(\nabla^J h)\]
on the forth line.
\end{itemize}
We finally consider the term $
\frac{d}{dt_{|_{t=0}}} \int_M h\al_J\wedge\al_J\wedge \omega_0^{[n-2]}$. Using \eqref{variation-alpha} and  that $\al_J$ is a primitive  harmonic form of type $(1,1)$),  we get
\[
\begin{split}
   & \frac{d}{dt} \int_M 2h\al_J\wedge\al_J\wedge \omega_0^{[n-2]} = 4\int_M h\dot{\al}_J\wedge\al_J \wedge \omega_0^{[n-2]}\\
    & = -8 \int_M hd\delta_J\mathbb{G}_J(\al_J\circ J\dot{J})^{\rm skew}\wedge\al_J\wedge \omega_0^{[n-2]}
 =8 \int_M h\left\langle d\delta_J\mathbb{G}_J(\al_JJ\dot{J})^{\rm skew},\al_J\right\rangle_J\omega_0^{[n]}\\
    &= -8 \int_M \left\langle \mathbb{G}_J\delta_J(\al_JJ\dot{J})^{\rm skew},\al_J(\nabla^J h)\right\rangle_J\omega_0^{[n]}
    =8\int_M \left\langle \mathbb{G}_J\delta_J\left( J^{*}(\al_J\dot{J})^{\rm skew}\right),\al_J(\nabla^J h)\right\rangle_J\omega_0^{[n]}\\
    &=-8\int_M \left\langle \mathbb{G}_J J \delta_J \left((\al_J\dot{J})^{\rm skew}\right),\al_J(\nabla^J h)\right\rangle_J\omega_0^{[n]}
    = -8\int_M \left\langle J\mathbb{G}_J \delta_J \left((\al_J\dot{J})^{\rm skew}\right),\al_J(\nabla^J h)\right\rangle_J\omega_0^{[n]}\\
    &={8 \int_M \left\langle \mathbb{G}_J \delta_J \left((\al_J\dot{J})^{\rm skew}\right),\al_J(J \nabla^J h)\right\rangle_J\omega_0^{[n]}} 
    = {8\int_M \left\langle\mathbb{G}_J(\al_J\dot J)^{\rm skew} , \left(\mathcal{L}_{J\nabla^J h} \al_J\right)^{J, -}\right\rangle_J \omega_0^{[n]}} \\
     & = { -8 \int_M \left\langle\mathbb{G}_J(\al_J\dot J)^{\rm skew} , \left(\al_J J (\mathcal{L}_{J\nabla^J h} J)\right)^{\rm skew}\right\rangle_J \omega_0^{[n]}}\\ 
    &=-8\int_M \left\langle\mathbb{G}_J(\al_JJ\dot J)^{\rm skew} , \left(\al_J (- (\mathcal{L}_{J\nabla^J h} J))\right)^{\rm skew}\right\rangle_J \omega_0^{[n]}.
\end{split}
\]
Substituting back the three terms and using that $\dot J$ is $g_{J}$ self-adjoint and anti-commutes with $J$, we  get
\[
\begin{split}
  -2 \left(\frac{d}{dt}_{|_{t=0}}\, \int_M h \boldsymbol{\mu}(J(t))\right) =& \int_M\left\langle\dot{J},(\nabla^{J,-}( Jdh))\right\rangle_J R_J\omega_0^{[n]} {+} 2\int_M\left\langle \delta_J\dot{J} , \delta_J \left(\nabla^{J,-}J dh\right)\right\rangle_J\omega_0^{[n]} \\
  &- 8 \int_M \left\langle \mathbb{G}_J\left(\al_J J \dot J\right)^{\rm skew}, \left(\al_J (-\mathcal{L}_{J\nabla^J h}J)\right)^{\rm skew} \right\rangle_J \omega_0^{[n]} \\
  =& {-\frac{1}{2} \int_M \left\langle J\dot{J}, (-\mathcal{L}_{J\nabla h} J) \right\rangle_J  R_J \, \omega_0^{[n]} -\int_M\left\langle J(\delta_J\dot{J}) , \delta_J\left(-\mathcal{L}_{J\nabla h} J\right)\right\rangle_J\omega_0^{[n]}} \\
  &-8 \int_M \left\langle \mathbb{G}_J\left(\al_J J \dot J\right)^{\rm skew}, \left(\al_J (-\mathcal{L}_{J\nabla^J h}J)\right)^{\rm skew} \right\rangle_J \omega_0^{[n]}.
\end{split}
\]
For the last equality we used the identity (see \cite[Lemma~9.2.1]{gauduchon-book})
\[ - J \mathcal{L}_Z J = 2(\nabla^J Z)^{\rm sym}, \]
applied to the symplectic vector field $Z= J\nabla^J h$, and,  as $J$ is integrable,  
\[\nabla^{J,-}(J dh)=(\nabla^{J}( Jdh))^{\rm sym}= -\frac{1}{2}J\mathcal{L}_{J\nabla^J h} J.\]
\noindent 
The equality \eqref{momentum-property} is thus established. \end{proof}

\bigskip
The ideal $\mathfrak{ham}_{\omega_0} < \mathfrak{sp}_{\omega_0}$ gives rise to a   $\boldsymbol{J}$-invariant distribution on $\mathcal{J}_{\omega_0}$ \[\boldsymbol{D}:= \{ -\mathcal{L}_V J \, | \, V \in \mathfrak{ham}_{\omega_0} + J \mathfrak{ham}_{\omega_0}\}  \subset \boldsymbol{T}(\mathcal{J}_{\omega_0}).\] We can think of $\boldsymbol{D}$ as the complexification of $\mathfrak{ham}_{\omega_0}$ obtained via the infinitesimal holomorphic action of $\mathfrak{ham}_{\omega_0}$ on $(\mathcal{J}_{\omega_0}, \boldsymbol{J})$, i.e we have complex Lie algebra homomorphism \[ \mathfrak{ham}_{\omega_0}\otimes \C \ni f+ ih \longmapsto -\mathcal{L}_{J\nabla^J f} J - J(\mathcal{L}_{J\nabla^J h}J) \in \boldsymbol{D},\] 
where $\boldsymbol{D}$ is endowed with the complex $\boldsymbol{J}$.

\begin{lemma} \label{l:symplectic}  The restriction of $\boldsymbol{\Omega}$ to $\boldsymbol{D}$ is a closed form, i.e \[ (d\boldsymbol{\Omega})_J(\boldsymbol{X}, \boldsymbol{Y}, \boldsymbol{Z})=0 \qquad \forall \boldsymbol{X}, \boldsymbol{Y}, \boldsymbol{Z} \in \boldsymbol{D}.\] \end{lemma} \begin{proof} This is a consequence of the momentum map property \eqref{momentum-property}. Indeed, as $\boldsymbol{\Omega}$  is invariant under the action of ${\rm Diff}(M)$, it is in particular ${\rm Ham}_{\omega_0}$ invariant. Thus,  for any $h\in C^{\infty}(M)$ if  $X=-\omega_0^{-1}(dh)$ is the corresponding vector on $M$ and $\hat X (J):= -(\mathcal{L}_{X}J)$ the induced vector field on $\mathcal{J}_{\omega}$, Cartan's formula and \eqref{momentum-property} tel us \[(d\boldsymbol{\Omega})(\hat X, \cdot, \cdot)= d\left(\boldsymbol{\Omega}(\hat X, \cdot)\right) = - d\left(d \langle \boldsymbol{\mu}, h\rangle\right) =0,  \] where $\langle \boldsymbol{\mu}, h\rangle$ stands for the function on $\mathcal{J}_{\omega_0}$ \[ J \longmapsto \int_M \boldsymbol{\mu}_J h \omega_0^{[n]}.\] We thus only need to show that for any hamiltonian vector fields $X, Y, Z \in \mathfrak{ham}_{\omega_0}$,  \[ (d\boldsymbol{\Omega})(\boldsymbol{J}\hat X, \boldsymbol{J}\hat Y, \boldsymbol{J}\hat Z)=0. \] As  $\boldsymbol{\Omega}$ is $\boldsymbol{J}$-invariant and $\boldsymbol{J}$ has zero Nijenhueis tensor, $d\boldsymbol{\Omega}$ is a $3$-form of type $(2,1)+(1,2)$ with respect to $\boldsymbol{J}$: it thus follows that \[(d\boldsymbol{\Omega})(\boldsymbol{J}\hat X, \boldsymbol{J}\hat Y, \boldsymbol{J}\hat Z)= (d\boldsymbol{\Omega})(\boldsymbol{J}\hat X, \hat Y, \hat Z)+(d\boldsymbol{\Omega})(\hat X, \boldsymbol{J}\hat Y, \hat Z)+(d\boldsymbol{\Omega})(\hat X, \hat Y, \boldsymbol{J}\hat Z)=0. \] \end{proof}

\begin{lemma}\label{l:equivariant} The map $\boldsymbol{\mu}: \mathcal{J}_{\omega_0} \to \left(C^{\infty}(M)/\R\right)^*$ is equivariant with respect to the infinitesimal action of $\mathfrak{ham}_{\omega_0}$ on $\mathcal{J}_{\omega_0}$ and its (adjoint) infinitesimal action  on $\left(C^{\infty}(M)/\R\right)^*$  by the $\omega_0$-Poisson bracket, i.e. for any two hamiltonian vector fields $X=-\omega_0^{-1}(df), Y=-\omega_0^{-1}(dh) \in \mathfrak{ham}_{\omega_0}$ the following identity holds
\[ -\boldsymbol{\Omega}_J(\hat X, \hat Y) = d \left(\langle \boldsymbol{\mu}, f \rangle\right)(\hat Y(J)) = \langle \boldsymbol{\mu}(J), \{ f, h \}_{\omega_0}\rangle,\]
where $\hat X (J) := -\mathcal{L}_X J$ and $\hat Y(J):= -\mathcal{L}_Y J$ denote the fundamental vector fields on $\mathcal{J}_{\omega_0}$ induced by $X, Y$, respectively.
\end{lemma}
 \begin{proof} The flow $\phi_t^Y$ of $Y$ indices a flow line curve $J(t)= \phi_t^Y \cdot J \in \mathcal{J}_{\omega_0}$. We know that
 \[\int_M (\phi_t^Y)^*(f) \boldsymbol{\mu}(J(t)) = \int_M (\phi_t^Y)^*(f\mu(J))\]
 is independent of $t$. We thus compute (using also \eqref{momentum-property}):
 \[
 \begin{split}
d \left(\langle \boldsymbol{\mu}, f \rangle\right)(\hat Y)_J  &= \frac{d}{dt}_{|_{t=0}} \left(\int_M f \boldsymbol{\mu}(J(t))\right)  =-\int_M \left(\mathcal{L}_Y f\right)\boldsymbol{\mu}(J) \\
&= \int_M\{f, h\}_{\omega_0} \boldsymbol{\mu}(J) = \langle \boldsymbol{\mu}, \{f, h\}_{\omega_0} \rangle. 
 \end{split}
  \]
  The claim now follows by \eqref{momentum-property}.
 \end{proof}

 \begin{proof}[Proof of Theorem~\ref{thm:GIT}]  This is now an immediate corollary of Lemmas~\ref{l:mentum-property}, \ref{l:symplectic} and \ref{l:equivariant}, taking into account that $\boldsymbol{D}=\boldsymbol{T}(\mathcal{J}_{\omega_0})$ under the assumption $H^1(M, \R)=\{0\}$.
 \end{proof}
\begin{rmk} In \cite{DH}, in the case $\A=0$, it is proposed a general method of constructing a closed $(1,1)$-form $\boldsymbol{\Omega}_{DH}$ on $(\mathcal{AC}_{\omega_0}, \boldsymbol{J})$ such that $\boldsymbol{\mu}$ satisfies the momentum map property \eqref{momentum-property} with respect $\boldsymbol{\Omega}_{DH}$. It is not clear from the methods in \cite{DH} whether or not $\boldsymbol{\Omega}_{DH}>0$ on the subspace $\mathcal{AC}^+_{\omega_0}$. On the other hand,  we do not show the closeness of $\boldsymbol{\Omega}$ on the whole space $\mathcal{AC}_{\omega_0}$. What will be important for us  in this paper is that $\boldsymbol{\Omega}$ is positive definite on $(\mathcal{J}^+_{\omega_0}, \boldsymbol{J})$ and satisfies the momentum map and equivariant properties  established in Lemmas~\ref{l:mentum-property} and \ref{l:equivariant}.
\end{rmk}

\section{The Calabi--Lichnerowicz--Matsushima obstruction}
We recall (see e.g. \cite{gauduchon-book}) that the reduced group $\Aut_r(M, J)$ of complex automorphisms of $(M,J)$ is the connected closed subgroup of complex automorphisms,  whose Lie algebra $\mathfrak{h}_{\rm red}(M,J)$ consists of holomorphic vector fields with zeroes.
\begin{prop}\label{p:Calabi-Lichne-Mats} If \eqref{PDEAn} admits a solution $\omega_{\varphi}$, then  the K\"ahler structure $(g_{\varphi}, J, \omega_{\varphi})$ is invariant under a maximal connected compact subgroup $K < \Aut_r(M,J)$ of the reduced group of complex automorphisms of $(M,J)$. Moreover, $\Aut_r(M, J)$ is reductive, i.e. $K^{\C}=\Aut_r(M, J)$.
\end{prop}
\begin{proof}
The above statement can be derived  from the formal momentum map picture for \[ (\mathcal{J}^+_{\omega_0}, \boldsymbol{\Omega}, \boldsymbol{J}, \mathfrak{ham}_{\omega_0}, \boldsymbol{\mu}), \] (see \cite{LSW,ratiu,ASU}). One delicate point in our situation is that we  did not establish  the closeness of $\boldsymbol{\Omega}$ on the formal Hermitian manifold $(\mathcal{J}^+_{\omega_0}, \boldsymbol{J}, \boldsymbol{\Omega})$, but only of the restriction of $\boldsymbol{\Omega}$ to the $\boldsymbol{J}$-invariant distribution $\boldsymbol{D}$ obtained as the complexification of the infinitesimal action of $\mathfrak{ham}_{\omega_0}$ on $(\mathcal{J}_{\omega_0}, \boldsymbol{J})$ (see Lemma~\ref{l:symplectic}). However, as noticed in \cite[Appendix A]{ASU},  Lemmas~\ref{l:mentum-property} and \ref{l:equivariant} are precisely the ingredients needed for the proof of our claim.
\end{proof}

\begin{cor}\label{c:P^1xP^1} The only solutions of \eqref{PDEA} on $S=\PP^1 \times \PP^1$ are the K\"ahler products of constants scalar curvature metrics on each factor $\PP^1$.
\end{cor}
\begin{proof} A solution of \eqref{PDEA} on $\PP^1\times \PP^1$ must be $PU(2)\times PU(2)$-invariant by Proposition~\ref{p:Calabi-Lichne-Mats}.
\end{proof}
\section{The Mabuchi functional and Futaki character} In this section,  we fix the underlying complex structure $J=J_0$ on $M$,  and vary the K\"ahler form $\omega_0 \in 2\pi \Ome$  within the deRham K\"ahler class $\Ome$. We thus consider the  spaces 
\[ \mathcal{H}_{\omega_0}:=\{ \varphi \in C^{\infty}(X) \, | \, \omega_0 + dd^c \varphi >0\} \]
of $\omega_0$ relative  K\"ahler potentials, and its quotient by the action of $\R$ via additive constants
\[ \mathcal{K}_{\omega_0}=\{ \omega_\varphi:=\omega_0 + dd^c \varphi >0 \, | \, \varphi \in \mathcal{H}_{\omega_0}\} \cong \mathcal{H}_{\omega_0}/\R. \]  
Both $\mathcal{H}_{\omega_0}$ and $\mathcal{K}_{\omega_0}$ are Fr\'echet manifolds, respectively modelled on the vector spaces $C^{\infty}(M)$ and $C^{\infty}(M)/\R$.   The tangent spaces of $\mathcal{H}_{\omega_0}$ and $\mathcal{K}_{\omega_0}$ are respectively
\[ \boldsymbol{T}_{\varphi} (\mathcal{H}_{\omega_0})= \{ \dot \varphi \in C^{\infty}(X)\}, \qquad  \boldsymbol{T}_{\omega_{\varphi}}(\mathcal{K}_{\omega_0}) = \left\{ \dot\varphi \in C^{\infty}(M) \, \Big| \,  \int_M \dot \varphi \omega_{\varphi}^{[n]} =0\right\}.\]
We introduce a $1$-form  $\boldsymbol{\Theta}_\varphi(\dot\varphi)$ on $\mathcal{H}_{\omega_0}$ by the formula
\[
\boldsymbol{\Theta}_\varphi(\dot{\varphi})=-\int_M\dot{\varphi}\left(-\frac{1}{2}dd^cR_\varphi\wedge\omega_\varphi^{[n-1]}+\al_\varphi\wedge\al_\varphi\wedge \omega_{\varphi}^{[n-2]}+\rho_\varphi\wedge\rho_\varphi\wedge \omega_{\varphi}^{[n-2]}-\frac{c_{\Ome, \A}}{2}\omega_{\varphi}^{[n]}\right).
\]
\begin{rmk} By \eqref{topo-c}, $\boldsymbol{\Theta}$ also defines a $1$-form on $\mathcal{K}_{\omega_0}$.
\end{rmk}
\begin{lemma}\label{l:mabuchi} $\boldsymbol{\Theta}$ is a closed $1$-form on $\mathcal{H}_{\omega_0}$.
\end{lemma}
\begin{proof} For any $\dot \varphi \in C^{\infty}(X)$, we consider the constant vector field $\boldsymbol{V}_{\dot\varphi}$ on $\mathcal{H}_{\omega_0}$, defined by
\[ \boldsymbol{V}_{\dot{\varphi}}(\varphi):= \dot \varphi.\]
The vector field $\boldsymbol{V}_{\dot \varphi}$ admits a (local) flow: the curve $\varphi(t) =\varphi + t \dot \varphi$ is a flow curve of $\boldsymbol{V}_{\dot \varphi}$ through $\varphi$. It follows that constant vector fields mutually commute.  As they also generate each tangent space $\boldsymbol{T}_{\varphi}(\mathcal{H}_{\omega_0})$,   the closeness of $\boldsymbol{\Theta}$ then reads as
\begin{equation}\label{closenness} \boldsymbol{V}_{\dot{\varphi}} \cdot \boldsymbol{\Theta}(\boldsymbol{V}_{\dot \psi}) = \boldsymbol{V}_{\dot{\psi}} \cdot \boldsymbol{\Theta}(\boldsymbol{V}_{\dot \varphi}), \qquad \forall \, \dot \varphi, \dot\psi \in C^{\infty}(X), \end{equation}
an identity we are going to check in what follows. Notice that
\[ 
\begin{split}
\boldsymbol{V}_{\dot \varphi} \cdot \boldsymbol{\Theta}(\boldsymbol{V}_{\dot\psi}) &= \frac{d}{dt}_{|_{t=0}} \boldsymbol{\Theta}_{\varphi(t)}(\dot\psi)\\
&= -\frac{d}{dt}_{|_{t=0}}\int_M \dot\psi \left(\frac{1}{2}\left(\Delta_{\varphi(t)}\Scal_{\varphi(t)}\right) \omega_{\varphi(t)}^{[n]} + \rho_{\varphi(t)}^2\wedge \omega_{\varphi(t)}^{[n-2]}+\al_{\varphi(t)}^2\wedge  \omega_{\varphi(t)}^{[n-2]}- \frac{c_{\Ome, \A}}{2}\omega_{\varphi(t)}^{[n]}\right). \end{split}\]
We thus need to rewrite the above quantity as a symmetric expression in $\dot \psi$ and $\dot \varphi$.  As we shall recall below, this is a corollary of the momentum map property of $\boldsymbol{\mu}$ that we have proved in Lemma~\ref{l:mentum-property}.

Indeed, let $\omega_{\varphi(t)}$ be the smooth path of K\"ahler metrics in $\mathcal{K}_{\omega_0}$ with $\omega_{\varphi(0)}=\omega_{\varphi}$,  and $h(t)$ a smooth path of functions. We shall compute below the quantity $\frac{d}{dt}_{|_{t=0}} \boldsymbol{\Theta}_{\varphi(t)}(h(t))$.

To lighten the notation, we denote $\omega:=\omega_{\varphi}$. By Moser's lemma, there is an isotopy of diffeomorphisms $\phi_t$, defined by the time dependent vector field 
\[ Z_t =-J\omega_{ \varphi(t)}^{-1}(d \dot\varphi (t))= -(\nabla^{\varphi(t)} \dot \varphi(t)),\] such that $\phi_t^*(\omega_{\varphi(t)}) =\omega$.  It follows that  $J(t):= (\phi_t^{-1})_* J (\phi_t)_* \in \mathcal{J}_{\omega}$. Furthermore (see \cite[Sec.~ 4.6]{gauduchon-book}):
\[ \dot{J} = J\mathcal{L}_{J\nabla^J \dot \varphi} J.\]
By the invariance of $\boldsymbol{\Theta}$ under symplectomorphsms and using \eqref{momentum-property}, it  follows that 
\begin{equation}\label{variational-formula}
\begin{split} &-\frac{d}{dt}_{|_{t=0}} \boldsymbol{\Theta}_{\varphi(t)}(h(t))\\
&=\frac{d}{dt}_{|_{t=0}}\int_M h(t) \left(\frac{1}{2}\left(\Delta_{\varphi(t)}\Scal_{\varphi(t)}\right) \omega_{\varphi(t)}^{[n]} + \rho_{\varphi(t)}^2\wedge \omega_{\varphi(t)}^{[n-2]}+\al_{\varphi(t)}^2\wedge  \omega_{\varphi(t)}^{[n-2]}-\frac{c_{\Ome, \A}}{2}\omega_{\phi(t)}^{[n]}\right) \\
&= \frac{d}{dt}_{|_{t=0}}\int_M \phi_t^*(h(t))\left(\frac{1}{2}\left(\Delta_{J(t)}\Scal_{J(t)}\right) \omega^{[n]} + \rho_{J(t)}^2\wedge \omega^{[n-2]}+\al_{J(t)}^2\wedge  \omega^{[n-2]}-\frac{c_{\Ome, \A}}{2}\omega^{[n]}\right) \\
&=\int_M \left(\dot{h}-\left\langle dh, d\dot \varphi\right\rangle_J\right)\left(\frac{1}{2}\left(\Delta_{J}\Scal_{J}\right) \omega^{[n]} + \rho_{J}^2\wedge \omega^{[n-2]}+\al_{J}^2\wedge  \omega^{[n-2]} -\frac{c_{\Ome, \A}}{2}\omega^{[n]}\right)  \\
& \, \, \, \, \, \, \, \, -\boldsymbol{\Omega}_{J}((\mathcal{L}_{J \nabla h}J), J(\mathcal{L}_{J\nabla \dot \varphi} J)) \\
&=\int_M \left(\dot h -\left\langle dh, d\dot \varphi\right\rangle_\omega\right)\left(\frac{1}{2}\left(\Delta_{\omega}\Scal_{\omega}\right) \omega^{[n]} + \rho_{\omega}^2\wedge \omega^{[n-2]}+\al_{\omega}^2\wedge  \omega^{[n-2]} -\frac{c_{\Ome, \A}}{2}\omega^{[n]}\right)  \\
& \, \, \, \, \, \, \, \, -\boldsymbol{g}_{\omega,J}((\mathcal{L}_{J \nabla^\omega h}J), (\mathcal{L}_{J\nabla^{\omega} \dot \varphi} J)).
\end{split} 
\end{equation}
In the case when $h(t)= \dot\psi$ is independent of $t$, we get an expression
which is manifestly symmetric in $\dot\psi$ and $\dot\varphi$. \end{proof}

As $\mathcal{H}_{\omega_0}$ is a contractible space,  we can find a primitive of
$\boldsymbol{\Theta}$:
\begin{defn}[Mabuchi energy]\label{d:Mabuchi} There exists a functional $\mathcal{E}:\mathcal{H}_{\omega_0}\to\R$  such that
\[
(D_{\varphi}\mathcal{E})(\dot{\varphi})=-\int_M\dot{\varphi}\left(-\frac{1}{2}\left(dd^cR_\varphi\right)\wedge\omega_\varphi^{[n-1]}+\al_\varphi^2\wedge\omega_\varphi^{[n-2]}+\rho_\varphi^2\wedge\omega_\varphi^{[n-2]}-\frac{c_{\Ome, \A}}{2}\omega_{\varphi}^{[n]}\right), \qquad \mathcal{E}(0)=0,
\]
which we call the \emph{Mabuchi functional} associated to \eqref{PDEAn}.
Furthermore, by \eqref{topo-c}, 
\[ \mathcal{E}(\varphi+c)= \mathcal{E}(\varphi),\] i.e $\mathcal{E}$ is also defined on the space $\mathcal{K}_{\omega_0}$ of K\"ahler metics in $2\pi\Ome$.   The critical points of $\mathcal{E}$ are the K\"ahler metrics verifying \eqref{PDEAn}.
\end{defn}

\subsection{The Futaki character}
The existence of a well-defined Mabuchi functional associated to \eqref{PDEAn} yields the definition of a character
\[ \mathcal{F} : \mathfrak{h}(M, J) \to \R \]
defined on the Lie algebra $\mathfrak{h}(M, J)$ of (real) holomorphic vector fields on $Z$ as follows:
If $Z\in C^{\infty}(M, TM)$ is such a vector field, its flow $\phi_t^Z$ preserves the complex structure and the K\"ahler class $\Ome$ on $X$. It follows that $Z$ gives rise to a vector field $\boldsymbol{V}_Z$ on the space of K\"ahler metrics in $2\pi \Ome$, with a flow line $(\phi_t^Z)^* \omega_{\varphi}$ at $\omega_{\varphi}\in 2\pi \Ome$; this can be lifted to a vector field (still denoted $\boldsymbol{V}_Z$) on $\mathcal{K}_{\omega_0}$, such that
\[ \frac{d}{dt}_{|_{t=0}} (\phi_t^Z)^*(\omega_\varphi) = dd^c h_{\varphi},  \qquad \int_M h_{\varphi} \omega_{\varphi}^{[n]}=0,\]
where $h_\varphi$ is the so-called \emph{real holomorphy potential} of $Z$ with respect to $\omega_\varphi$, see \cite[Lemma~2.1.2]{gauduchon-book}.
We then let
\begin{equation}\label{e:Futaki} \begin{split}
\mathcal{F}_{\varphi}(Z)&:=  \boldsymbol{\Theta}(\boldsymbol{V}_Z)(\varphi)\\
&= -\int_M h_{\varphi} \left( -\frac{1}{2}(dd^cR_\varphi)\wedge\omega_{\varphi}^{[n-1]}+\al_{\varphi}^2\wedge\omega_{\varphi}^{[n-2]}+\rho_{\varphi}^2\wedge\omega_{\varphi}^{[n-2]}-\frac{c_{\Ome, \A}}{2}\omega_{\varphi}^{[n]}\right).
\end{split}
\end{equation}
\begin{prop}\label{p:Futaki} The expression \eqref{e:Futaki} is independent of $\omega_\varphi \in \mathcal{K}_{\omega_0}$ and thus defines a linear map
$ \mathcal{F} : \mathfrak{h} \to \R$,  satisfying
\[ \mathcal{F}([Z, Y])=0 \qquad \forall Z, Y \in \mathfrak{h}.\]
Furthermore, $\mathcal{F} \equiv 0$ should \eqref{PDEAn} admit a solution.
\end{prop}
\begin{proof} We  first show that the function
\[ \varphi \to \boldsymbol{\Theta}(\boldsymbol{V}_Z)(\varphi)\]
is constant. 
Notice that the $1$-form $\boldsymbol{\Theta}$ is invariant under (the flow of) $\boldsymbol{V}_Z$,  as it is induced by the action of ${\rm Diff}(X)$ on the entire structure;  as $\boldsymbol{\Theta}$ is a closed $1$-form (see Lemma~\ref{l:mabuchi}), it follows by Cartan's formula that $\boldsymbol{\Theta}(\boldsymbol{V}_Z)(\varphi)$ is constant.

If $Y\in \mathfrak{h}$ is another real holomorphic vector field, inducing a vector field $\boldsymbol{V}_Y$ on $\mathcal{K}_{\omega_0}$, then
\[ \mathcal{F}([Z,Y])= \boldsymbol{\Theta}([\boldsymbol{V}_Z, \boldsymbol{V}_Y]) = -(d \boldsymbol{\Theta})(\boldsymbol{V}_Z, \boldsymbol{V}_Y) - \boldsymbol{V}_Y \cdot \boldsymbol{\Theta}(\boldsymbol{V}_Z) + \boldsymbol{V}_Z \cdot \boldsymbol{\Theta}(\boldsymbol{V}_Y) =0,\]
as claimed.

Finally, if $\omega_{\varphi}\in 2\pi \Ome$ solves \eqref{PDEAn}, then clearly $\mathcal{F}=\mathcal{F}_{\varphi} \equiv 0$.
\end{proof}
\begin{defn}[Futaki invariant]\label{d:futaki} The character
\[ \mathcal{F} : \mathfrak{h}(M,J) \to \R \]
defined by \eqref{e:Futaki} is called the \emph{Futaki invariant}  associated to \eqref{PDEAn}. The condition $\mathcal{F} \equiv 0$  is necessary for \eqref{PDEAn} to admit a solution.
\end{defn}

\subsection{The Calabi--Lichnerowicz--Matsushima obstruction revisited}
The computation \eqref{variational-formula} allows us to extend Proposition~\ref{p:Calabi-Lichne-Mats} to the full automorphism group.
\begin{prop}\label{p:CLM-full} If \eqref{PDEAn} admits a solution $\omega$, then the Lie algebra of holomorphic $\mathfrak{h}(M,J)$ vector fields of $(M, J)$ decomposes as
\[\mathfrak{h}(M, J) = \mathfrak{a} \oplus \mathfrak{k} \oplus J \mathfrak{k}, \]
where $\mathfrak{a}$ is the abelian Lie algebra of $\omega$-parallel (Killing) vector fields,   $\mathfrak{k}={\rm Lie}(K)$ is the Lie algebra of hamiltonian Killing vector fields of $\omega$,  and $\mathfrak{h}_{\rm red}(M,J)= \mathfrak{k}\oplus J\mathfrak{k}$ is the ideal of holomorphic vector fields with zeroes of $(M,J)$. In particular, $\Aut(M, J)$ is reductive and the connected component of identity of the  K\"ahler isometry group $G$ of $\omega$ is a maximal compact subgroup of $\Aut_0(M, J)$.
\end{prop}
\begin{proof} Proposition~\ref{p:Calabi-Lichne-Mats} gives 
\[\mathfrak{h}_{\rm red}(M, J) = \mathfrak{k}\oplus J\mathfrak{k},\]
where $\mathfrak{h}_{\rm red}(M, J) < \mathfrak{h}(M, J)$ is the ideal of holomorphic vector fields with zeroes. In order to extend the argument for  arbitrary holomorphic vector fields, let $Z$ be real a holomorphic vector field and $\alpha^Z= g(Z)$ its riemannian dual $1$-form. It is well-known (see e.g. \cite[Lemma~2.2.1]{gauduchon-book}) that $\alpha^Z$ has Hodge decomposition
\[ \alpha^Z = \psi_H^Z + d^cf^Z + dh^Z, \]
where $\psi_H^Z$ is a  $\omega$-harmonic $1$-form and $f^Z, h^Z$ are smooth functions normalized by $\int_M f^Z\omega^{[n]}=0=\int_M h^Z \omega^{[n]}$.  Acting with the flow $\phi_t^Z$ of $Z$ on $\omega$, we obtain a smooth  path of solutions of \eqref{PDEAn} (which all have $\Scal_{\omega_t}>0$):
\[ \omega_{t} := (\phi_t^Z)^* \omega \in \mathcal{K}_{\omega},  \qquad \theta_t := (\phi_t^Z)^* \theta, \qquad \dot \omega_t = dd^c\left((\phi_t^Z)^*h^Z\right).\]
Denote by $\dot \varphi_t = (\phi_t^Z)^*h^Z$ the velocity of 
$\omega_t$.  As $\omega_t=\omega_0 + dd^c\varphi(t)$ solves \eqref{PDEAn}, $\boldsymbol{\Theta}_{\varphi(t)} \equiv 0$. The computation \eqref{variational-formula} then gives (using that $\omega_0=\omega$ solves \eqref{PDEAn}):
\begin{equation*}
\begin{split}
0=&\frac{d}{dt}_{|_{t=0}} \boldsymbol{\Theta}_{\varphi(t)}(\dot \varphi(t)) = \|\mathcal{L}_{J\nabla h^Z}J\|^2_{\boldsymbol{g}_{\omega,J}}.
\end{split}
\end{equation*}
This shows that $h^Z$ is a Hamiltonian potential of  a Killing vector field $\omega^{-1}(dh^Z)$ of $\omega$. Similarly, considering $JZ$ instead of $Z$, we conclude that $\omega^{-1}(df^Z)$ is a Killing vector field. Thus,
\[ Y:=Z - \nabla f^Z - J\nabla h^Z = g^{-1}(\psi_H^Z)=g^{-1}(\psi_H^Y)\]
is a real holomorphic vector field whose riemannian dual form is harmonic.  By the Bochner--Lichnerowicz formula (\cite[1.20.4]{gauduchon-book})
\[ 0 = \delta \nabla^{-} \psi_H^Y = \frac{1}{2}\Delta \psi_H^Y -{\rm Ric}(\omega)(Y) = - {\rm Ric}(\omega)(Y). \]
By Bochner's formula, $\psi_H^Y$ and, equivalently, $Y$ is parallel. Clearly, if $Y$ is parallel so is $JY$, so the algebra $\mathfrak{a}$ is complex. The claim follows easily. 
\end{proof}

\subsection{Convexity of the Mabuchi functional along smooth geodesics and conditional uniqueness}
Another immediate consequence of the formal momentum map property \eqref{momentum-property} of $\boldsymbol{\mu}$ is the convexity of the Mabuchi energy $\mathcal{E}$ along a smooth Mabuchi geodesic $\omega_{\varphi(t)}\in \mathcal{K}_{\omega_0}$ of K\"ahler metrics with positive scalar curvatures. Recall that such a geodesic is defined (see \cite[Sec. 4.6]{gauduchon-book}) by a curve $\varphi(t) \in \mathcal{H}_{\omega_0}$ such that:
\begin{equation}\label{geodesic} \ddot \varphi (t) = || \nabla^{\varphi(t)} \dot \varphi(t)||^2_{\omega_{\varphi(t)}}.\end{equation}
We are going to prove that
\[ \frac{d^2}{dt^2} \mathcal{E}(\varphi(t)) = \frac{d}{dt} \boldsymbol{\Theta}_{\varphi(t)}(\dot \varphi(t))\geq 0,\]
as soon as $\omega_{\varphi(t)}$ has positive scalar curvature $\Scal_{\varphi(t)}$.
As $\varphi(t+c)$ is also geodesic, it is enough to check the above inequality at $t=0$.  Given a smooth geodesic $\varphi(t)$, we apply \eqref{variational-formula} with $h(t)=\dot\varphi(t)$ and use \eqref{geodesic} to compute
\begin{equation*}\label{Moser-geodesic}
\begin{split}
&\frac{d^2}{dt^2}_{|_{t=0}}\mathcal{E}(\varphi(t))=\frac{d}{dt}_{|_{t=0}} \boldsymbol{\Theta}_{\varphi(t)}(\dot \varphi(t)) =\|\mathcal{L}_{J\nabla \dot \varphi}J\|^2_{\boldsymbol{g}_{\omega,J}}.
\end{split}
\end{equation*}
The RHS is non-negative as soon as $\Scal_\omega>0$
We thus have an analog of a fundamental result in the cscK case
\begin{prop}\label{p:mabuchi-convex} 
\begin{enumerate}
    \item[(a)] The Mabuchi functional $\mathcal{E}$ associated to \eqref{PDEAn} is convex along smooth Mabuchi geodesics of K\"ahler metrics with positive scalar curvature in $\mathcal{K}_{\omega_0}$. 
    \item[(b)] If $\varphi(t)$ is a smooth geodesic emanating from $\varphi$ such that $\Scal_{\varphi(t)}>0$, then $\frac{d^2}{dt^2}\mathcal{E}(\varphi(t))\equiv 0$ holds if and only if $\varphi(t)$ is induced by the flow of  $-JX$, where $X$ a Hamiltonian Killing vector field of $(M, J, \omega_{\varphi})$. 
    \item[(c)] If $\omega_{\varphi_1}, \omega_{\varphi_2} \in \mathcal{K}_{\omega_0}$ solve \eqref{PDEAn} and can be connected by a smooth geodesic of K\"ahler metrics with positive scalar curvature, then there exists a reduced biholomorphism $\phi \in \Aut_r(M, J)$ such that $\phi^* \omega_{\varphi_2}=\omega_{\varphi_1}$.
    \end{enumerate}
\end{prop}
\begin{proof} (a) follows from \eqref{Moser-geodesic}. 

(b) \eqref{Moser-geodesic} yields that the second derivative of $\mathcal{E}$ along a smooth geodesic $\varphi(t)$ identically vanishes iff the smooth path of symplectic fields $X_t:=J (\nabla^{\varphi(t)} \dot \varphi(t))$ preserve $J$, and hence are hamiltonian Killing.  By \cite[Prop.4.6.3]{gauduchon-book},  $X_t=X$ is time independent and $\varphi(t) = (\phi_t^{-JX})^* \varphi(0)$ where $\phi_t^{-JX}$ is the flow of $-JX$.

(c) This follows from (a) and (b) as the solutions of \eqref{PDEAn} are critical points of $\mathcal{E}$. \end{proof}

\section{The linearization of the 6th order PDE}
For $\omega \in \Ome$, we let
\begin{equation}\label{calS}
\begin{split}
\mathcal{S}(\omega) &:= \left(-\frac{1}{2} dd^c \Scal_{\omega} \wedge \omega^{[n-1]} + \rho_{\omega}^2 \wedge \omega^{[n-2]} + \al_{\omega}^2 \wedge \omega^{[n-2]}\right)\Big/ \omega^{[n]} \\
&= \frac{1}{2} \Delta_{\omega} \Scal_{\omega} + \frac{\Scal_{\omega}^2}{4} -\frac{1}{2}\|{\rm Rc_\omega} \|^2_{\omega} -\frac{1}{2}\|\al_{\omega}\|^2_{\omega}.
\end{split}\end{equation}
For $\varphi \in \mathcal{H}_{\omega_0}$, we denote
\[ \mathcal{S}({\varphi}) := \mathcal{S}(\omega_{\varphi}),\]
so that $\mathcal{S}(\varphi)=c_{\Ome, \A}$ is the PDE \eqref{PDEB} from the Introduction.
We give below the linearization of $\mathcal{S}(\varphi)$, viewed as a $6$-th order quasilinear PDE on $\varphi$.~\footnote{We were kindly informed by R. Dervan that Mahmoud Abdelrazek has independently obtained the linearization formula for $\mathcal{S}$ in the framework of existence theory of  $Z$-critical metrics on projective bundles.}
\begin{prop}\label{prop:linearization}
For a path $\omega_{\varphi(t)} \in \mathcal{K}_{\omega_0}$ with $\omega_{\varphi(0)}=\omega_{\varphi}$ and  $\frac{d}{dt}_{|_{t=0}} \omega_{\varphi(t)} =dd^c \dot \varphi$, the linearization of $\mathcal{S}({\varphi(t)})$ at $t=0$ is given by 
\begin{equation}\begin{split}
\label{eq:linearization}
(D_{\varphi} \mathcal{S})(\dot \varphi)&=\frac{d}{dt}_{|_{t=0}}\mathcal{S}(\varphi(t))= \langle d(\mathcal{S}({\varphi})), d\dot \varphi \rangle_{\omega_\varphi} -  2\delta_\varphi\left(\delta_\varphi\nabla^{\varphi,-}
\left(\delta_\varphi\nabla^{\varphi,-}d\dot{\varphi}\right)\right)\\ &- \delta_\varphi\delta_\varphi\left(R_\varphi\nabla^{\varphi,-}d\dot{\varphi}\right)  {+} 16 \delta_{\varphi} \delta_{\varphi} \left(\al_{\varphi} \circ\mathbb{G}_{\varphi}(\al_{\varphi} \circ \nabla^{\varphi, -} d\dot \varphi)^{\rm skew}\right).
\end{split}\end{equation}
\end{prop} 
\begin{proof}
   By \eqref{variational-formula}, for any smooth function $h$, we have
\[
\begin{split} 
&\frac{d}{dt}_{|_{t=0}}\int_M h \mathcal{S}({\varphi(t)}) \omega_{\varphi(t)}^{[n]}\\
=&-\int_M \left\langle dh, d\dot \varphi\right\rangle_{\omega_\varphi}\mathcal{S}({\varphi}) \omega_{\varphi}^{[n]} -  \boldsymbol{g}_{\omega_\varphi,J} ((\mathcal{L}_{J \nabla^{\omega_\varphi}h} J) , (\mathcal{L}_{J\nabla^{\omega_\varphi}\dot\varphi} J)) \\
=&-\int_M \left\langle dh, d\dot \varphi\right\rangle_{\omega_\varphi}\mathcal{S}({\varphi}) \omega_{\varphi}^{[n]}-  \int_M \left\langle \nabla^{\varphi,-}dh, \nabla^{\varphi,-}d\dot{\varphi}\right\rangle_J \Scal_\varphi \,\omega_\varphi^{[n]}  \\
&- 2\int_M\left\langle \delta_\varphi(\nabla^{\varphi,-}dh), \delta_\varphi(\nabla^{\varphi,-}d\dot{\varphi})  \right\rangle_J\omega_\varphi^{[n]}\\
& {-} 16\int_M \left\langle \mathbb{G}_\varphi\left((\al_\varphi\circ\nabla^{\varphi,-}dh)^{\rm skew}\right), \left(\al_\varphi\circ \nabla^{\varphi,-}d\dot{\varphi}\right)^{\rm skew}\right\rangle_\varphi \, \omega_\varphi^{[n]}
\end{split} 
\]
where we use that $\nabla^{\varphi,-}dh=-\frac{1}{2}\mathcal{L}_{J \nabla^{\omega_\varphi}h} J$. Using that
\[
\begin{split}
    \frac{d}{dt}_{|_{t=0}}\int_M h \mathcal{S}({\varphi(t)}) \omega_{\varphi(t)}^{[n]}=&-\int_M \left\langle dh, d\dot \varphi\right\rangle_{\omega_\varphi}\mathcal{S}({\varphi}) \omega_{\varphi}^{[n]}-\int_M\left\langle d(\mathcal{S}({\varphi})), d\dot \varphi\right\rangle_{\omega_\varphi}h \omega_{\varphi}^{[n]}\\
    &+\int_M h L(\dot{\varphi})\omega_{\varphi}^{[n]},
\end{split}
\]
where $L(\dot{\varphi}):= \frac{d}{dt}_{|_{t=0}} \mathcal{S}({\varphi(t)})$, we obtain
\[
\begin{split}
    \int_M  L(\dot{\varphi})h\, \omega_{\varphi}^{[n]}=\int_M&\Big(\langle d(\mathcal{S}({\varphi})), d\dot \varphi \rangle_{\omega_\varphi} - 2\delta_\varphi\left(\delta_\varphi\nabla^{\varphi,-}
\left(\delta_\varphi\nabla^{\varphi,-}d\dot{\varphi}\right)\right)\\ &-\delta_\varphi\delta_\varphi\left(R_\varphi\nabla^{\varphi,-}d\dot{\varphi}\right)  + 16\, \delta_{\varphi} \delta_{\varphi} \left(\al_{\varphi} \circ\mathbb{G}_{\varphi}(\al_{\varphi} \circ \nabla^{\varphi, -} d\dot \varphi)^{\rm skew}\right)\Big) h\,\omega_\varphi^{[n]}.
\end{split}
\]
This completes the proof.
\end{proof}

\section{Relaxing \eqref{PDEAn}: the extremal K\"ahler metrics for \eqref{PDEAn} and the LeBrun--Simanca openness result} As in the case of cscK metrics, the momentum map interpretation of \eqref{PDEAn} leads to a natural generalization of \eqref{PDEAn},  for which the Futaki obstruction will be trivially satisfied.  
\begin{defn}[Extremal K\"ahler metric for \eqref{PDEAn}]\label{d:extremal} Suppose $(M, J)$ is a compact complex $n$-dimensional manifold or orbifold,   endowed with a K\"ahler class $\Ome$ and a deRham class $\A\in H^2(M, \R)$,  satisfying \eqref{topo-n}.
A K\"ahler metric $\omega \in \Ome$ is called \emph{extremal} for \eqref{PDEAn} if  the smooth function $\mathcal{S}(\omega)$ (see \eqref{calS})
is a Killing potential, i.e. $V=-\omega^{-1}(d\mathcal{S}(\omega))$ is a Killing vector field for the K\"ahler structure $(\omega, J)$.
\end{defn}

\begin{lemma}\label{l:Calabi} An extremal K\"ahler metric for \eqref{PDEAn} is a critical point of the Calabi energy ${\mathcal Ca} : \mathcal{K}_{\omega_0} \to \R$
\[ {\mathcal Ca}(\omega):= \int_M \mathcal{S}(\omega)^2 \omega^{[n]}.\]
Conversely, a K\"ahler metric $\omega \in \Ome$ with positive scalar curvature  which is a critical point of ${\mathcal Ca}$ is  extremal  for \eqref{PDEAn}.
\end{lemma}
\begin{proof} For a path $\omega_{\varphi(t)} \in \mathcal{K}_{\omega_0}$ with $\omega_{\varphi(0)}=\omega_{\varphi}$ and  $\frac{d}{dt}_{|_{t=0}} \omega_{\varphi(t)} =dd^c \dot \varphi$,  we use the Moser lift $\phi_t$ as in the proof of Lemma~\ref{l:mabuchi}. This  gives rise to a path $J(t):=\phi_t \cdot J \in \mathcal{J}_{\omega_0}$ such that
\[ {\mathcal Ca}(\omega_{\varphi(t)})= \int_M \left(\frac{\boldsymbol{\mu}(J(t))}{\omega_0^{[n]}}\right) \boldsymbol{\mu}(J(t)).\]
Differentiating and using \eqref{momentum-property} with $\dot J=J \mathcal{L}_{J \nabla^J \dot \varphi}J$, we obtain
\[
\begin{split}
    \frac{d}{dt}_{|_{t=0}} {\mathcal Ca}(\omega_{\varphi(t)}) &= 2 \boldsymbol{\Omega}_{J}(J\mathcal{L}_{J\nabla^J \dot \varphi}J, \mathcal{L}_{J \nabla^J \mathcal{S}_{\varphi}} J)=-2\boldsymbol{g}_J(\mathcal{L}_{J\nabla^J \dot \varphi}J, \mathcal{L}_{J \nabla^J \mathcal{S}_{\varphi}} J).
    \end{split}\]
The claim follows by the positive definiteness of $\boldsymbol{g}$ on $\mathcal{J}_{\omega_0}^+$.    \end{proof}

We notice that Proposition~\ref{p:Calabi-Lichne-Mats} extends to the extremal case (see \cite[Thm.A.5]{ASU}):

\begin{prop}\label{p:Calabi-ext} An extremal K\"ahler metric for \eqref{PDEAn}  with $\Scal_{\omega}>0$ is invariant under a maximal compact (connected) subgroup $K<\Aut_{r}(M, J)$. 
\end{prop}
Because of this, we can fix once for all a  maximal compact torus $\T < \Aut_r(M, J)$, take a $\T$-invariant base point $\omega_0 \in 2\pi \Ome$,  and search for $\T$-invariant extremal K\"ahler metrics in 
$\mathcal{K}_{\omega_0}^\T$. In this setup the momentum map property \eqref{momentum-property} yields
\begin{lemma}\label{l:extremal-field} Suppose $\omega \in \mathcal{K}^{\T}_{\omega_0}$ and let $P_{\omega}$ denote the finite dimensional vector space of all Killing potentials for Killing vector fields in $\tor = {\rm Lie}(\T)$. Then the $L_2(M, \omega)$-orthogonal projection of $\mathcal{S}(\omega)$ to $P_{\omega}$ is the Killing potential of a vector field $V^{\rm ext} \in \tor$ which is independent of the choice of $\omega$.
\end{lemma}
\begin{proof} Uses the momentum map property \eqref{momentum-property} and the fact that the $\T$-equivariant Moser lift of a normalized $\omega_{\varphi(t)}$-Killing potential for $V\in \tor$ is a (fixed) normalized $\omega_0$-hamiltonian  of $V$. \end{proof}
\begin{defn} The vector field $V^{\rm ext}$ is called the \emph{extremal vector field} of $(\Ome, \A, \T)$.
\end{defn}
We then have the following
\begin{lemma}\label{l:extremal-reduction} A K\"ahler metric $\omega\in \mathcal{K}_{\omega_0}^{\T}$ is extremal  for \eqref{PDEAn} iff $V^{\rm ext}= J \nabla^{\omega} \mathcal{S}(\omega)$.
\end{lemma}
\begin{proof} If $J \nabla^{\omega} \mathcal{S}(\omega)= V^{\rm ext}$ , then $\omega$ is extremal as $V^{\rm ext}\in \tor$ is a Killing vector field for $\omega \in \mathcal{K}_{\omega_0}^\T$. Conversely, if  $\omega \in \mathcal{K}_{\omega_0}^\T$ is extremal, then  $V:= J \nabla^{\omega} \mathcal{S}(\omega)$  is Killing and central for $\tor$. By the maximality of $\T< \Aut_r(M, J)$, $V\in \tor$, so $V=V^{\rm ext}$ by Lemma~\ref{l:extremal-field}. \end{proof}

\begin{rmk}\label{r:Futaki-Mabuchi} Similarly to \cite{lahdili,FM},  for each $\omega_{\varphi} \in \mathcal{K}_{\omega_0}^{\T}$,  we can consider the corresponding \emph{normalized} momentum map 
\[ \mu_{\varphi} : M \to \tor^*, \qquad \int_M \mu_{\varphi} \omega_{\varphi}^{[n]}=0.\]
This gives rise to a (normalized)  $\T$-momentum polytope $\Pol_\Ome:=\mu_{\varphi}(M) \subset \tor^*$, which is independent of the $\T$-invariant K\"ahler metric $\omega_{\varphi}$.  Then,  \eqref{topo-c}, Lemmas~\ref{l:extremal-field}  and \ref{l:extremal-reduction} tel us that the  affine linear function $\ell^{\rm ext}(x) =\langle V^{\rm ext}, x \rangle +c_{\Ome, \A}$ on $\Pol_{\Ome}$,  determined by the data $(M, \Ome, \A, \T)$, is such that the extremal K\"ahler metrics $\omega_{\varphi}$ for \eqref{PDEAn} in $\mathcal{K}_{\omega_0}^{\T}$ are  the solutions of the 6-th order PDE for the unknown $\T$-invariant potential  $\varphi\in \mathcal{H}_{\omega_0}$
\[ \mathcal{S}(\omega_{\varphi}) = \ell^{\rm ext}(\mu_{\varphi}), \qquad \mu_{\varphi} = \mu_0 + d^c \varphi.\]
When $V^{\rm ext}=0$ (a condition on $(\Ome, \A)$), the extremal K\"ahler metrics  $\omega \in \Ome$ satisfy \eqref{PDEAn}.
These are the analogue of constant scalar curvature K\"ahler metrics.  If $(\Ome, \A)$ are such that  $V^{\rm ext}=0$ and  $c_{\Ome, \A}=0$ (i.e. $(c_1(M)^2 + \A^2)\cdot \Ome^{n-2}=0$), then the extremal K\"ahler metrics in $2\pi \Ome$ of positive scalar curvature are precisely the solutions of \eqref{Box}.  These are analogues of scalar-flat K\"ahler metrics.
\end{rmk}

We next prove an analog of LeBrun--Simanca~\cite{LS} openness result for extremal K\"ahler classes.
\begin{prop}\label{p:LeBrun-Simanca} Suppose $(\Ome_0, \A_0) \in H^2(M, \R)$ satisfy \eqref{topo-n} and $\omega_0 \in \Ome_0$ is  an extremal K\"ahler metric for \eqref{PDEAn} with  $\Scal_{\omega_0}>0$.  Then, for any $(\Ome, \A) \in H^2(M, \R)$ which are sufficiently close to $(\Ome_0, \A_0)$ and satisfy \eqref{topo-n},  there exists an extremal K\"ahler metric $\omega\in 2\pi \Ome$ for \eqref{PDEAn} with $\Scal_{\omega}>0$.
\end{prop}
\begin{proof} This is an application of the IFT.  We shall use the formula for the linearization $(D_{\varphi}\mathcal{S})(\dot \varphi)$ of $\mathcal{S}({\varphi})$ given by Proposition~\ref{prop:linearization}.
We shall consider $\T$-invariant K\"ahler metrics $\omega \in 2\pi \Ome$ (see Lemma~\ref{p:Calabi-ext}) and $\T$-invariant functions in $\mathcal{H}_{\omega_0}$, and normalize the $\T$-momentum $\mu_{\omega}$ of $\omega$ by
\[ \int_M \mu_{\omega} \omega^{[n]}= 0 \in \tor^*.\]
This gives rise to a canonically normalized polytope $\Pol_{\Ome}$ associated to $\Ome$,  and to an extremal affine-linear function $\ell^{\rm ext}(x)$ associated to $(\Ome, \A)$ (see Remark~\ref{r:Futaki-Mabuchi}).  We can then define the reduced operator 
\[ \mathring{\mathcal{S}}(\varphi):= \mathcal{S}(\varphi) -\ell^{\rm ext}(\mu_{\varphi}), \]
 and notice that, by Proposition \ref{prop:linearization}, if $\omega$ is extremal for \eqref{PDEAn}, i.e. $\mathring{\mathcal{S}}(\omega) = 0$, then  
\begin{equation} \label{lineraization}
\begin{split} (D_{\omega}\mathring{\mathcal{S}})(\dot \varphi)&=-2\delta_\omega\delta_\omega\left(\nabla^{\omega,-}
\left(\delta_\omega\nabla^{\omega,-}d\dot{\varphi}\right)\right)\\ &-\delta_\omega\delta_\omega\left(R_\omega\nabla^{\omega,-}d\dot{\varphi}\right)  + 16 \delta_{\omega} \delta_{\omega} \left(\al_{\omega} \circ\mathbb{G}_{\omega}(\al_{\omega} \circ \nabla^{\omega, -} d\dot \varphi)^{\rm skew}\right).
\end{split}\end{equation}

We denote by $\mathcal{S}_{\Ome, \A}(\omega)$ the expression \eqref{calS}, emphasizing the dependence on $(\Ome, \A)$  (which we are going to vary), and consider its  reduced part
\[ \mathring{\mathcal{S}}_{\Ome, \A}(\omega) :=\mathcal{S}_{\Ome, \A}(\omega) -\ell^{\rm ext}_{\Ome, \A}(\mu_{\omega}). \]
We can thus define a map
\[ \mathring{\mathcal{S}}: \mathcal{U} \times U  \to C^{\infty}(M)^{\T}, \qquad \mathring{\mathcal{S}}(\omega, \al, \varphi):=\mathring{\mathcal{S}}_{\frac{[\omega]}{2\pi}, \frac{[\al]}{2\pi}}(\omega + dd^c \varphi), \]
where:
\begin{itemize}
\item   $0\in \mathcal U \subset C^{\infty}(M)^{\T}$ is an open subset in the Fr\'echet topology;
\item $(\omega_0, \al_{\omega_0}) \in U$ is a relative open subset of the hyper-surface $\mathcal{H}=\{(\omega, \al) \in \mathcal{H}^{1,1}_{\omega_0} \times \mathcal{H}^{1,1}_{\omega_0} \, | \, \int_M \al \wedge \omega^{[n-1]}  =0\}$,  where  $\mathcal{H}^{1,1}_{\omega_0}$ stands for the vector space of harmonic $(1,1)$-forms with respect to $\omega_0$;
\item $\mathcal{U}$ and $U$ are so chosen that for any  $(\varphi, \omega, \al) \in \mathcal{U}\times U$, $\tilde \omega_{\varphi} := \omega + dd^c \varphi >0$ and $\Scal_{{\tilde \omega}_{\varphi}}>0$.
\end{itemize}
 $\mathring{\mathcal{S}}$ can be extended as a Fr\'echet differentiable map between domains in the Sobolev spaces
\[ \mathring{\mathcal{S}} : \mathcal{V} \times U \to W^{2,k}(M)^{\T}, \qquad \mathcal{V} \subset W^{2, k+6}(M)^{\T}.\]
We next denote by $\mathring{W}^{2, p}(M)^{\T}$ the $L^2(M, \omega_0)$-orthogonal complement in  $W^{2, p}(M)^{\T}$ of the subspace $P_{\omega_0}$ of $\omega_0$-Killing potentials of vector fields in $\tor$ (i.e. $P_{\omega_0} = \{ \ell(\mu_{\omega_0})  \, | \, \ell(x) \, \textrm{affine-linear on} \, \tor^*\}$) and let $\mathring{\mathcal{V}} = \mathcal{V}\cap \mathring{W}^{2, k+6}(M)^{\T}$. Notice that
\[ \mathring{\mathcal{S}}(0, \omega_0, \al_{\omega_0}) : \mathring{\mathcal{V}} \times U \to \mathring{W}^{2,k}(M)^\T \]
by construction. Furthermore, by \eqref{lineraization}, the $\mathring{\mathcal{V}}$-partial derivative of $\mathring{\mathcal{S}}$ at $(\omega_0, \al_0)$ is
\[\begin{split}  \mathbb{S}(\dot \varphi):= (D_{\varphi} \mathring{\mathcal{S}})(\dot \varphi, \omega_0, \al_0) =&-2\delta_{\omega_0}\delta_{\omega_0}\left(\nabla^{\omega_0,-}
\left(\delta_{\omega_0}\nabla^{\omega_0,-}d\dot{\varphi}\right)\right)\\ &-\delta_{\omega_0}\delta_{\omega_0}\left(R_\omega\nabla^{\omega,-}d\dot{\varphi}\right)  +16 \delta_{\omega_0} \delta_{\omega_0} \left(\al_{\omega_0} \circ\mathbb{G}_{\omega_0}(\al_{\omega_0} \circ \nabla^{\omega_0, -} d\dot \varphi)^{\rm skew}\right).
\end{split} \]
Notice that 
\[ \mathbb{S} : \mathring{W}^{2,k+6}(M)^{\T} \to \mathring{W}^{2, k}(M)^{\T}\]
is a linear,  $L^2(M, \omega_0)$-self-adjoint 6th 
order  elliptic pseudo-differential operator.  Using $\Scal_{\omega_0}>0$, the kernel  of $\mathbb{S}$ consists of $\T$-invariant functions in $\mathring{W}^{2,k+6}(M)^{\T}$ which satisfy $\delta_{\omega_0} \delta_{\omega_0} \nabla^{\omega_0,-}(d\dot \varphi) =0$, i.e. which are Killing potentials for Killing vector fields in the center of $\tor$. By maximality of $\T$, such Killing fields are trivial, so $\mathbb{S}$ is invertible.

By the implicit function theorem applied to $({\rm Id} -\Pi_{\omega_0})(\mathring{\mathcal{S}})$, where $\Pi_{\omega_0}$ is the $L^2(M, \omega_0)$-orthogonal projection of $\mathring{W}^{2, k}(M)$ to $P_{\omega_0}$, we obtain a differentiable family of K\"ahler metrics $\omega_{\Ome, \A}=\omega + dd^c \varphi_{\alpha, \beta}$ (recall $\alpha :=\frac{[\omega]}{2\pi}, \, \beta :=\frac{[\theta]}{2\pi}$) of regularity $C^r(M), \, r\geq 6$ for which $\mathring{S}_{\Ome, \A}(\omega_{\alpha, \beta})$ belongs to $P_{\omega_0}$. These metrics must satisfy  $\mathring{S}_{\Ome, \A}(\omega_{\alpha, \beta})=0$  for $(\Ome, \A)$ sufficiently close to $(\Ome_0, \A_0)$ as, by definition,   $\mathring{S}_{\Ome, \A}(\omega_{\alpha,\beta})$ projects trivially to $P_{\omega_{\alpha,\beta}}$ with respect to the $L^2(M, \omega_{\alpha,\beta})$-orthogonal projection. 

The proof concludes by showing that $\omega_{\alpha, \beta}$ are smooth via a standard boot-strapping argument, using the ellipticity of ${\mathcal{S}}_{\Ome, \A}$  and the fact that $\ell^{\rm ext}_{\alpha, \beta}(x)$ is smooth whereas $\mu_{\omega_{\alpha,\beta}} =\mu_{\omega} + d^c\varphi_{\alpha,\beta}$ by the normalization condition for $\Pol_\Ome$.
\end{proof}

\section{Applications to BHE 3-folds: Proof of Theorem~\ref{thm:main2}}
We shall assume through this section that $(M, J)$ is a compact orbifold complex surface and $(\Ome, \A)$ satisfy
\begin{equation}\label{c=0} c_{\Ome, \A} = 4(c_1^2(M)+ \A^2)=0.\end{equation}
We are thus solving \eqref{PDE} on $M$,  which in turn will produce non-K\"ahler BHE 3-folds via the inverse of Theorem~\ref{thm:reduction}.
\subsection{Ruled complex surfaces with vanishing Futaki invariant}
Here we consider the minimal ruled complex surfaces  described in Lemma~\ref{rough-classification} (1)-(2) for which the underlying vector bundle $E=L_1 \oplus L_2$ splits as a sum of line bundles and will characterize those for which the Futaki invariant associated to \eqref{PDEB} (see Proposition~\eqref{p:Futaki}) vanishes. More generally, we consider complex orbifold surfaces $M$ which are the total space of  a $\PP^{1}_{v,w}$-bundle over a Riemann surface $\Sigma$, where $\PP^1_{v,w}$ (with $v,w\in \Z_{>0}$ co-prime) denotes the weighted projective line (seen as a $1$-dimensional toric orbifold) attached through a principal $S^1$-bundle $P$ over $\Sigma$,  endowed with a connection $1$-form $\eta_{P}$  with curvature
\[ d\eta_{P} = k\omega_{\Sigma},\]
where $\omega_{\Sigma}$ is a Riemannian metric of constant scalar curvature $s_{\Sigma}$ on $\Sigma$. We can then use the Calabi Ansatz construction~\cite{HS,ACGTF-inventiones} to represent (a scale of) $\Ome$ with special K\"ahler metrics, allowing also to explicitly write down the the harmonic $2$-form $\al$ in \eqref{Box}.

To this end,  we consider K\"ahler metrics of the form
\begin{equation} \label{ansatz1}
g=(kz+ c)g_{\Sigma} + \frac{1}{\Theta(z)} dz^2 + \Theta(z) \eta_P^2, \qquad \omega = (kz+c)\omega_{\Sigma} + dz\wedge \eta_{P}, 
\end{equation}
where:
\begin{itemize}
\item $\Theta(z)>0$ is a smooth function on defined on the interval $(-1, 1)$,  which satisfies the end points conditions
\begin{equation} \label{boundary-k} \Theta(-1)=0=\Theta(1), \qquad \Theta'(-1) = 2/w, \qquad \Theta'(1)=-2/v. \end{equation}
\item $(kz+ c)$ is a strictly positive function on $[-1,1]$.
\end{itemize}
The ansatz \eqref{ansatz1}  describes an $S^1$-invariant K\"ahler structure on $\mathring{M}:= (-1, 1) \times P \to \Sigma$, with corresponding momentum map $z: M \to (-1, 1)$. It is well-known (see e.g. \cite{HS,ACGTF-jdg}) that the end-point conditions \eqref{boundary-k} for $\Theta$ ensure that this K\"ahler metric compactifies on $M$ as a smooth orbifold K\"ahler metric.
Thus, the ansatz~\eqref{ansatz1}  describes an orbifold K\"ahler structure on the  orbifold
\[ M=S_{v, w, k}(\Sigma) := \PP^1_{v, w} \times_{S^1} P \to \Sigma, \]
which has a Killing $S^1$-symmetry with momentum map $z: S_{p,q, k}(\Sigma) \to [-1, 1]$. Furthermore, as explained in \cite{ACGTF-inventiones}, the constant $k$ fixes the principal bundle $P$ (and therefore the topology of $S_{v, w, k}(\Sigma)$) whereas the constant $c$ parametrizes the K\"ahler class $2\pi\Ome_c=[\omega]$ on $S_{v,w, k}(\Sigma)$: up to a positive scale, $\alpha_c$ exhaust the K\"ahler cone of $S_{v,w,k}(\Sigma)$. 

\bigskip
The computations of Ricci form and scalar curvature of \eqref{ansatz1} are made in  e.g. \cite{ACG-jdg}: Letting
\[ V(z):= \Theta(z) (kz + c), \]
we have 
\[ 
\begin{split}
\rho(\omega) &= \rho_{\Sigma} -\frac{1}{2} dd^c \log V =-\frac{1}{2}\left(\frac{kV'(z)}{(kz + c)} -s_{\Sigma}\right)\omega_{\Sigma} - \frac{1}{2}\left(\frac{V'(z)}{(kz + c)}\right)' dz\wedge \theta , \\\Scal(\omega)&= 4 \left(\frac{\rho(\omega)\wedge \omega}{\omega\wedge \omega}\right)= \frac{s_{\Sigma}}{c+kz}-\frac{V''(z)}{(kz + c)},
    \end{split} \]
whereas 
\[ \al_\lambda := \frac{\lambda}{2}\Big(\frac{1}{(kz + c)}\omega_{\Sigma} - \frac{1}{(kz+c)^2}dz \wedge \theta \Big) \]
is a primitive closed $(1,1)$-form. Furthermore,  for a smooth function $f=f(z)$ we compute
\[ dd^c f = \left(\frac{f'(z)V(z)}{(kz+ c)}\right)' dz \wedge \theta +k\left(\frac{f'(z)V(z)}{(kz+ c)}\right) \omega_{\Sigma}.\]
Putting all these together, we obtain the following extension for $\mathcal{S}(\omega)$
\begin{equation} \label{ODE-k}
\begin{split}
2\mathcal{S}(\omega)&=-\left[\left(\frac{s_{\Sigma}}{c+kz} - \frac{V''}{kz+c}\right)'\left(\frac{V}{kz+c}\right) \right]' - \frac{k}{(kz + c)}\left[\left(\frac{s_{\Sigma}}{c+kz} - \frac{V''}{kz+c}\right)'\left(\frac{V}{kz+c}\right) \right] \\
& +\frac{1}{(kz+c)}\left(\frac{kV'}{kz+c} - s_{\Sigma}\right)\left(\frac{V'}{kz +c}\right)' - \frac{\lambda^2}{(kz + c)^4}.
\end{split}
\end{equation}

\begin{lemma}\label{l:k neq 0} Suppose $k\neq 0$, $v=w$,  and $(2\pi \Ome_c=[\omega], 2\pi \A_{\lambda}=[\al_{\lambda}])$ satisfy \eqref{c=0}. Then the Futaki invariant \eqref{e:Futaki} associated to \eqref{PDEAn} on $(S_{v,w,k},\Ome_c, \A_{\lambda})$ is non-zero. 
\end{lemma}
\begin{proof}
We first use that the integral of \eqref{ODE-k} on [-1,1] against $(kz+c)dz$ is zero. This corresponds to the topological condition \eqref{c=0} we imposed on $(\Ome_c, \A_{\lambda})$ in order to ensure $c_{\Ome_c, \A_{\lambda}}=0$.  Integration by parts using  \eqref{boundary-k} gives
 \begin{align*}
   \int_{-1}^1 \frac{\lambda^2}{(kz+c)^3}dz=\lambda^2 \Big( -\frac{1}{2k(kz+c)^2} \Big)\Big|_{-1}^1=\frac{2c \lambda^2}{(c^2-k^2)^2}, 
\end{align*}
\begin{align*}
    \int_{-1}^1 \left[\left(\frac{s_{\Sigma}}{kz+c} - \frac{V''}{kz+c}\right)'\left(\frac{V}{kz+c}\right) \right]' (kz+c)dz=-\int_{-1}^1 \left(\frac{s_{\Sigma}}{kz+c} - \frac{V''}{kz+c}\right)'\left( \frac{kV}{kz+c} \right)dz,
\end{align*}
\begin{align*}
    &\int_{-1}^1 \left(\frac{kV'}{kz+c} - s_{\Sigma}\right)\left(\frac{V'}{kz +c}\right)' dz=\left( \frac{k}{2}\left( \frac{V'}{kz+c}-\frac{s_{\Sigma}}{k} \right)^2  \right)\Bigg|_{-1}^1
    \\&=\frac{1}{2k}\left(\frac{-2k}{v}-s_{\Sigma} \right)^2-\frac{1}{2k}\left(\frac{2k}{w}-s_{\Sigma} \right)^2
    =2\left(  \frac{1}{v}+\frac{1}{w} \right)\left(k\left(  \frac{1}{v}-\frac{1}{w}\right)+s_{\Sigma} \right).
\end{align*}
We conclude that
\begin{align}\label{Futaki1}
    \lambda^2 =\frac{(c^2-k^2)^2}{c}\left(  \frac{1}{v}+\frac{1}{w} \right)\left(k\left(  \frac{1}{v}-\frac{1}{w}\right)+s_{\Sigma} \right).
\end{align}
Next, we integrate \ref{ODE-k} on $[-1,1]$ against $(kz+c)^2 dz$. This  computes (up to a non-zero scale) the Futaki invariant \eqref{e:Futaki} of the holomorphic vector field $kJK$, where $K= -\omega^{-1}(dz)$ is the Killing vector field defined by the $S^1$ fibre-wise action of $P_k$.  Integration by parts using \eqref{boundary-k} gives
\begin{align*}
    \int_{-1}^1 \frac{\lambda^2}{(kz+c)^2}dz=\frac{2\lambda^2}{c^2-k^2},
\end{align*}
\begin{align*}
   \int_{-1}^1& (kz+c)^2\left[\left(\frac{s_{\Sigma}}{c+kz} - \frac{V''}{kz+c}\right)'\left(\frac{V}{kz+c}\right) \right]' + kV\left(\frac{s_{\Sigma}}{c+kz} - \frac{V''}{kz+c}\right)'dz
   \\&=\int_{-1}^1 \left(\frac{s_{\Sigma}}{c+kz} - \frac{V''}{kz+c}\right)'' (kz+c)V+\left(\frac{s_{\Sigma}}{c+kz} - \frac{V''}{kz+c}\right)' (kz+c)V' dz
   \\&=-\int_{-1}^1 kV\left(\frac{s_{\Sigma}}{c+kz} - \frac{V''}{kz+c}\right)' dz
   =\int_{-1}^1 kV'\left(\frac{s_{\Sigma}}{c+kz} - \frac{V''}{kz+c}\right) dz.
\end{align*}
Assuming $\mathcal{F}(-kJK)=0$, further integration by parts yields:
\begin{align*}
    \frac{2\lambda^2}{c^2-k^2}&=-\int_{-1}^1kV'\left(\frac{s_{\Sigma}}{c+kz} - \frac{V''}{kz+c}\right) dz+\int_{-1}^1 \big(kV'-(kz+c)s_{\Sigma}\big)\left( \frac{V'}{kz+c}\right)' dz
    \\&=\int_{-1}^1 \frac{2kV'V''}{kz+c}-\frac{k^2(V')^2}{(kz+c)^2}-V''s_{\Sigma} dz
    =\left(\frac{k(V')^2}{kz+c}\right)\Bigg|_{-1}^1 -s_{\Sigma}(V')\Big|_{-1}^1
    \\&=4k^2\left(\frac{1}{v^2}+\frac{1}{w^2}\right)+4kc\left(\frac{1}{v^2}-\frac{1}{w^2}\right)+s_{\Sigma}\left[ 2c\left(\frac{1}{v}+\frac{1}{w}\right)+2k\left(\frac{1}{v}-\frac{1}{w}\right)\right].
\end{align*}
In summary, 
\begin{align}\label{Futaki2}
    \frac{\lambda^2}{c^2-k^2}=2k^2\left(\frac{1}{v^2}+\frac{1}{w^2}\right)+2kc\left(\frac{1}{v^2}-\frac{1}{w^2}\right)+s_{\Sigma}\left[ c\left(\frac{1}{v}+\frac{1}{w}\right)+k\left(\frac{1}{v}-\frac{1}{w}\right)\right].
\end{align}
Combining (\ref{Futaki1}) and (\ref{Futaki2}), we get 
\[
\left[k\left(\frac{1}{v}+\frac{1}{w}\right)+c\left(\frac{1}{v}-\frac{1}{w}\right)  \right]\left[s_{\Sigma}+k\left(\frac{1}{v}-\frac{1}{w}\right)+c\left(\frac{1}{v}+\frac{1}{w}\right)\right]=0. 
\]
If $v=w$ we get
\[
s_{\Sigma}=-c\left(\frac{1}{v}+\frac{1}{w}\right), \qquad \lambda^2=-(c^2-k^2)^2\left(\frac{1}{v}+\frac{1}{w}\right)^2.
\]
It follows that $\lambda=0$ and $c^2=k^2$. However, $kz+c$ has to be strictly positive on $[-1,1]$, a contradiction. We thus conclude that $\mathcal{F}(kJK)\neq 0$. 
\end{proof}
In the case when  $k=0$ \eqref{ansatz1} describes a product K\"ahler metric on $S_{v, w,0}= \PP^1_{v,w} \times \Sigma$, where the metric on $\Sigma$ is the constant scalar curvature metric $(g_{\Sigma},\omega_{\Sigma})$ whereas the K\"ahler metric on $\PP^1_{v,w}$ is the $S^1$-invariant K\"ahler structure 
\[ g_{\PP^1_{v,w}}=\frac{dz^2}{\Theta(z)} + \Theta(z) dt^2, \qquad  \omega_{\PP^1_{v,w}}= dz\wedge dt, \] described in momentum/angular coordinates $(z,t)$. In this case, \eqref{ODE-k} simplifies to 
\begin{equation}\label{ODE-k=0} 
\begin{split}
2\mathcal{S}(\omega)=-\left( \left( \frac{s_{\Sigma}}{c}-\frac{V''}{c}\right)'\left(\frac{V}{c}\right)  \right)'-\frac{s_{\Sigma}}{c}\left(\frac{V'}{c} \right)'-\frac{\lambda^2}{c^4}.
\end{split}
\end{equation}
\begin{lemma} Suppose $k= 0$, $v\neq w$, and $(2\pi \Ome_c=[\omega], 2\pi \A_{\lambda}=[\al_{\lambda}])$ satisfy \eqref{c=0}. Then the Futaki invariant \eqref{e:Futaki} associated to \eqref{PDEAn} on $(S_{v,w,0},\Ome_c, \A_{\lambda})$ is non-zero. 
\end{lemma}
\begin{proof}
In this case, \eqref{Futaki1} simplifies to 
\begin{align} \label{k=0 Futaki}
    \lambda^2=c^3 s_{\Sigma}\left(\frac{1}{v}+\frac{1}{w}\right).
\end{align}
Integrating \eqref{ODE-k=0} on $[-1,1]$ against $zdz$ computes $-4\mathcal{F}(JK)$. Similarly to the case $k\neq 0$, we find that $\mathcal{F}(JK)=0$ iff
\begin{align}\label{k=0 Futaki1}
  s_{\Sigma}\left(\frac{1}{v}-\frac{1}{w}\right) =c\left(\frac{1}{v^2}-\frac{1}{w^2}\right).
\end{align}
If $v\neq w$, the relations \eqref{k=0 Futaki} and \eqref{k=0 Futaki1} yield 
\[s_{\Sigma}=-c\left(\frac{1}{v}+\frac{1}{w}\right),\quad \lambda^2=-c^4\left(\frac{1}{v}+\frac{1}{w}\right)^2,\]
a contradiction.
Therefore, if $v\neq w$,  $\mathcal{F}(JK) \neq 0$.     \end{proof}
\begin{rmk}\label{r:ruled Futaki=0} If $k=0$ and $v=w=1$, $s_{\Sigma}\geq 0$ by \eqref{k=0 Futaki}. In this case $V(z)=c - cz^2, \, c>0$ solves $\mathcal{S}(\omega)=0$ and corresponds to the product of  constant scalar curvature metrics on $\PP^1$  and $\Sigma$.
We thus have $\mathcal{F}\equiv 0$ by Proposition~\ref{p:Futaki}. These solutions of \eqref{PDE} give rise to BHE $3$-folds covered by the Samelson geometries \eqref{Samelson-3D}: in the case when $s_{\Sigma}>0$ we obtain the $SU(2)\times SU(2)$ geometry and in the case when $s_{\Sigma}=0$ we get the $SU(2)\times \R^3$ geometry.
\end{rmk}

\subsection{Proof of Theorem~\ref{thm:main2}}
\begin{proof}[Proof of Theorem~\ref{thm:main2}] The quotient space $S:=N/\T^2$ is a smooth compact complex surface $S$ endowed with a K\"ahler metric $(g_K^T, \omega_K^T)$ verifying \eqref{PDE}.  It follows that $S$ must be a complex surface listed in Lemma~\ref{rough-classification}. By \cite[Prop.2.21]{ABLS}, the hypothesis $h^{1,1}_{\rm BC}(N, J_N)=2$ yields  $b_2(S) \leq 2$. However, the K\"ahler surfaces  in Lemma~\ref{rough-classification} all have $b_+(S)=1$ and a non-positive signature (see e.g. \cite{BPV}). Thus  $b_-(S) \geq 1$ and $b_2(S)=b_+(S)+b_-(S)\geq 2$. We conclude that $b_2(S)=2$. This  singles out the surfaces described in the cases (1) and (2) of Lemma~\ref{rough-classification}. We analyse  below the two cases separately.

\bigskip
\noindent
{\bf Case 1.} $S=\PP(\mathcal{O}_{\PP^1} \oplus \mathcal{O}_{\PP^1}(k))$, $k\geq 0$. Then $S$ admits a Calabi-type K\"ahler metric given by \eqref{ansatz1} with $v=w=1, \, s_{\Sigma}=4$ in each rescaled K\"ahler class. By Lemma~\ref{l:k neq 0}, we must have $k=0$. By Corollary~\ref{c:P^1xP^1}, we conclude that the K\"ahler metric 
is a product of constant Gauss curvature metrics on each factor (see Remark~\ref{r:ruled Futaki=0}). This leads to the Samelson geometry $SU(2)\times SU(2)$ on $N$.

\bigskip
\noindent
{\bf Case 2.} $S=\PP(E) \to \Sigma$ is a ruled surface over an elliptic curve. In this case $c_1^2(S)=0$, so we are  looking for solutions of \eqref{PDEA} with $\A=0$.
Atiyah~\cite{atiyah} has shown that there are 3 possible cases for $S$:

\smallskip
\noindent
{\bf Case 2(a).} $S=\PP(\mathcal{O}_{\Sigma} \oplus L_k)$ for a line bundle $L_k \to \Sigma$ of degree $k\geq 0$. Such surfaces admit in each suitably rescaled K\"ahler class metrics given by the Calabi ansatz with $s_{\Sigma}=0$.  Lemma~\ref{l:k neq 0} yields $k=0$, i.e.  $\mathcal{O}_{\Sigma} \oplus L_0$ is  polystable. By Narasimhan--Seshadri Theorem~\cite{NS}, $\cO_{\Sigma} \oplus L_0$ is a projectively-flat Hermitian bundle,  defined by a representation $\epsilon : \pi_1(\Sigma) \to PU(2)$, 
where $\pi_1(\Sigma) \cong \mathbb{Z} \oplus \mathbb Z$ is the fundamental group of the elliptic curve $\Sigma$.
It follows that $S$ admits locally symmetric solutions of \eqref{PDE}, corresponding to  regular complex BHE $3$-folds  $N$ covered by $SU(2)\times \R^3$. We  thus need to obtain  the uniqueness of the locally symmetric solution. Suppose $\omega$ is any solution of \eqref{PDEA}. If $L_0=\cO_{\Sigma}$ is trivial, $S= \Sigma\times \PP^1$  is a product ruled surface. By Proposition~\ref{p:Calabi-Lichne-Mats}, the K\"ahler metric must be invariant under the action of $PU(2) < \Aut_{r}(S)$. Writing  point-wisely
\[ \omega({c, p}) = \omega_{\Sigma}(c,p) \oplus  \omega_{\PP^1}(c,p), \qquad (c,p)\in \Sigma \times \PP^1, \]
and using that $PU(2)$ acts transitively on $\PP^1$, we get that $\omega_{\Sigma}(c,p)=\omega_{\Sigma}(c)$ is a K\"ahler metric on $\Sigma$ independent of $p$ whereas
$\omega_{\PP^1}(c,p)= f(c) \omega_{FS}(p)$ is a multiple of the Fubini--Study metric on $\PP^1$ (the unique riemannian metric up to scale on $\PP^1$ invariant under $PU(2)$). By the closeness of $\omega$, we conclude $f(c)=const$, so $\omega = \omega_{\Sigma} + \omega_{FS}$ is a product K\"ahler metric. In this case, 
\[ \rho(\omega) = \frac{\Scal_{\Sigma}}{2} \omega_{\Sigma} + \frac{\Scal_{FS}}{2} \omega_{FS}, \qquad \Scal(\omega) = \Scal_{\Sigma} + \Scal_{FS}, \] 
so \eqref{PDEA} reduces to 
\[ -\Delta \Scal_{\Sigma} = \frac{1}{2}\Scal_{FS} \Scal_{\Sigma}.\]
Integrating by parts the above equality against $\Scal_{\Sigma}$ we get $\Scal_{\Sigma}\equiv 0$, so the metric is locally symmetric.

Suppose now that $L_0$ is a non-trivial degree zero line bindle. In this case, by \cite[Thm.~3]{Suwa}, the Lie algebra $\mathfrak{h}(S)$ of holomorphic vector fields is $2$-dimensional.  Its reduced ideal $\mathfrak{h}_{\rm red}(S)$, which consists of the holomorphic vector fields with zeroes, projects trivially to the base $\Sigma$: it follows that $\mathfrak{h}_{\rm red}(S)= H^0(\Sigma, \mathfrak{sl}(\cO_{\Sigma} \oplus L_0))\cong \C \oplus H^0(\Sigma, L_0) \oplus H^0(\Sigma, L_0^{-1})$.  As $L_0$ is non-trivial and of zero degree, we conclude that $\mathfrak{h}_{\rm red} \cong \C$ is spanned by the generator of fibre-wise $\C^*$-action. By Proposition~\ref{p:CLM-full}, $\omega$ admits non-trivial parallel vector fields $X, JX$. It thus follows that the Ricci tensor has zero eigenvalue of multiplicity $2$, so that $\rho(\omega)\wedge \rho(\omega)=0$ at each point. Thus, \eqref{PDEA} becomes $\Delta \Scal(\omega)=0$, showing that the scalar curvature is a positive constant.  It follows that ${\rm Rc}(\omega)$ has constant eigenvalues $\left(0,0, \frac{\Scal(\omega)}{2}, \frac{\Scal(\omega)}{2}\right)$.  The K\"ahler metric is then locally product and hence locally symmetric, see \cite[Thm.~1.1]{ADM}. 

\smallskip
\noindent
{\bf Case 2(b).} $S=\PP(E_0)$ with $E_0$ a degree $0$ indecomposable vector bundle, which is the extension
\[ 0\to \mathcal{O}_{\Sigma} \to E_0 \to \mathcal{O}_\Sigma \to 0.\]
In this case, the reduced automorphism group $\Aut_r(S) =\Aut_0(E)/\C^* \cong \C$ is not reductive (see \cite[Thm.3]{Suwa} and \cite[Example~4.4]{Daly}), so we rule out the existence of a solution by Proposition~\ref{p:Calabi-Lichne-Mats}.

\smallskip
\noindent
{\bf Case 2(c).}  $S=\PP(E_1)$ with $E_1$ indecomposable of degree $1$ and stable (see \cite[Appendix A]{Tu}). By Narasimhan--Seshadri Theorem~\cite{NS}, $E_1$  is a projectively-flat Hermitian bundle,  defined by a representation  $\epsilon : \pi_1(\Sigma) \to PU(2)$, and therefore $S$ admits locally symmetric solutions  corresponding to BHE geometries on $N$ covered by $SU(3)\times \R^3$. 

We have reduced again the problem to show the uniqueness of the locally symmetric solutions in this case. Suppose $\omega$ is any solution of \eqref{PDEA} defined on $S$. As $E_1$ is simple, by the arguments in Case 2(a),  $\mathfrak{h}_{\rm red}(S)=H^0(\Sigma, \mathfrak{sl}(E_1))=0$ whereas $\mathfrak{h}(S)=\C$ (see \cite[Thm.~3]{Suwa}). Proposition~\ref{p:CLM-full} yields that $\omega$ admits non-trivial parallel vector fields $X, JX$ and we conclude as in Case 2(a) that $\omega$ is locally symmetric. \end{proof}

\section{Explicit solutions on  toric orbifold surfaces with $b_2=2$: Proof of Theorems~\ref{thm:main3} and \ref{thm:main4}} 
\subsection{The orthotoric ansatz} We consider the case when $M$ is a toric K\"ahler orbifold complex surface whose Delzant polytope is a generic quadrilateral, i.e. has no parallel faces.  By \cite{legendre}, any such quadrilateral can be realized as the momentum map image of a toric K\"ahler metric given by the orthotoric ansatz~\cite{ACG}:
\begin{equation}\label{orthotoric}
\begin{split}
g&= (x-y)\left(\frac{dx^2}{A(x)}  + \frac{dy^2}{B(y)}\right) + \frac{1}{x-y}\Big(A(x)(dt_1 + y dt_2) + B(y)(dt_1 + x dt_2) \Big),\\
\omega&= dx \wedge (dt_1 + ydt_2) + dy \wedge (dt_1 + x dt_2),
\end{split}
\end{equation}
where $x\in (\ap_1, \ap_2)$ and $y\in (\bp_1, \bp_2)$ are two intervals with
\[ \bp_1 < \bp_2 < \ap_1 <\ap_2.\]
The variables $x>y$ are called  \emph{orthotoric coordinates},  and $t_1, t_2$ are angular coordinates on $\T^2$. The smooth functions $A(x)>0$ and $B(y)>0$ are  defined respectively on $(\ap_1, \ap_2)$ and $(\bp_1, \bp_2)$. Thus, \eqref{orthotoric} describes a $\T^2$-invariant K\"ahler metric on $\mathring{M}:= (\ap_1, \ap_2) \times (\bp_1, \bp_2) \times \T^2$, whose momentum coordinates (written with respect to the Killing fields $\frac{\partial}{\partial t_1}$  and $\frac{\partial}{\partial t_2}$) are 
\[\mu_1 = x+y, \qquad \mu_2= xy.\]
Notice that the image of $(\ap_1, \ap_2)\times (\bp_1, \bp_2)$ under $\mu=(\mu_1, \mu_2)$ is the convex compact 
quadrilateral $\Delta_{\ap, \bp} \subset \R^2$ defined by the inequalities 
\begin{equation}\label{ortho-labels}  
\begin{split}
& \ell_{\ap_1}(\mu):= -(\ap_1^2 - \mu_1 \ap_1 + \mu_2)>0 \qquad \ell_{\ap_2}(\mu):=+(\ap_2^2 - \mu_1 \ap_2 + \mu_2) >0\\
&\ell_{\bp_1}(\mu):=+ (\bp_1^2 - \mu_1 \bp_1 + \mu_2)>0 \,   \qquad \ell_{\bp_2}(\mu):=-(\bp_2^2 - \mu_1 \bp_2 + \mu_2)>0.\end{split}
\end{equation}
It is shown in \cite[Prop.9]{ACGTF-jdg} that \eqref{orthotoric} smoothly extends to an orbifold toric K\"ahler structure defined on the simply connected compact toric symplectic 4-dimensional orbifold $M$  constructed from the labelled Delzant polytope $(\Delta_{\ap, \bp}, {\bf L}=\{r_1\ell_{\ap_1}, r_2\ell_{\ap_2},p_1\ell_{\bp_1}, p_2\ell_{\bp_2}\})$ (see \cite{LT, ACGTF-jdg})~\footnote{If $r_i$ and $p_j$ are positive real numbers such that vectors  $r_i(-\ap_i, 1), p_j(-\bp_j, 1), \, \, i,j=1,2$ in $\R^2$ span  over $\Z$ a two dimensional lattice $\Lambda \subset \R^2$ then the orbifold singularities of $M$ are determined  by $({\bf L}, \Lambda)$.} if and only if $A(x)$ and $B(x)$ are  smooth functions on $[\ap_1, \ap_2]$ and $[\bp_1, \bp_2]$, respectively, and satisfy the boundary conditions
\begin{equation}\label{boundary-orho}
A(\ap_i)=0=B(\bp_j), \qquad A'(\ap_i) =(-1)^{i-1}\left(\frac{2}{r_i}\right),  \qquad B'(\bp_j)=(-1)^{j-1}\left(\frac{2}{p_j}\right),
\end{equation}
for some positive real numbers $r_i$ and $p_j$ (which encode the orbifold singularities of $M$). Equivalently,  and without assuming that $(\Delta_{\ap, \bp}, {\bf L})$ is rational, the toric K\"ahler metric~\ref{orthotoric} $(g, \omega)$ satisfies the Abreu--Guillemin boundary conditions with respect to the label ${\bf L}$ (see \cite{ACGTF-jdg,legendre} for more details).

A key feature of the ansatz \eqref{orthotoric} is that, similarly to \eqref{ansatz1},  we can write down a closed primitive $(1,1)$-form
\begin{equation}\label{alpha} \al_{\lambda} = \frac{\lambda}{(x-y)^2}\Big(dx \wedge (dt_1 + ydt_2) - dy\wedge(dt_1 + x dt_2) \Big),\end{equation}
thus allowing us to  express the quantity
\begin{equation*}\label{quantity}
\mathcal{S}(\omega)=\left(-\frac{1}{2}dd^c \Scal(\omega)\wedge \omega + \left(\al_{\lambda} \wedge \al_{\lambda} + \rho(\omega) \wedge \rho(\omega)\right)\right)\Big/\omega^{[2]}.
\end{equation*}
Using the explicit form of the K\"ahler structure we compute (see also \cite{ACG-crelle}):
\[ \omega^2 = (x-y)\, dx\wedge dy\wedge dt_1\wedge dt_2, \qquad \al_{\lambda}^2 = -\frac{\lambda^2}{(x-y)^3} dx\wedge dy \wedge dt_1 \wedge dt_2, \]
\[ 
\begin{split}
dd^c \left(f(x,y)\right) = &\frac{1}{2(x-y)} \Big[\left(A(x)f_x\right)_x +\left(B(y)f_y\right)_y \Big]\omega \\
&+\frac{(x-y)^3}{2}\left[\left(\frac{A(x)f_x}{(x-y)^2}\right)_x -\left(\frac{B(y)f_y}{(x-y)^2}\right)_y \right]\al_1  \\
& + f_{xy}\left(dx \wedge d^c y + dy \wedge d^c x\right),
\end{split}
\]
\begin{equation}\label{orthotoric-Ricci} 
\begin{split}
\rho(\omega) =&  -\frac{1}{4(x-y)}\Big[ A''(x) + B''(y) \Big]\omega \\ 
&-\frac{(x-y)^3}{4}\left[\left(\frac{A'(x)}{(x-y)^2}\right)_x- \left(\frac{B'(y)}{(x-y)^2}\right)_y \right]\al_1. \\
    \end{split} \end{equation}
We thus find
\[ (x-y)^4 \left(\frac{\al_\lambda \wedge \al_{\lambda}}{\omega \wedge \omega}\right)= -\lambda^2.\]
\[ \Scal(\omega)=4\left(\frac{\rho(\omega)\wedge \omega}{\omega \wedge \omega}\right)= -\left(\frac{A''(x) + B''(y)}{x-y}\right).\]
\[ \frac{dd^c f \wedge \omega}{\omega^2}= \frac{1}{2(x-y)}\Big( \big(A(x)f_x\big)_x + \big(B(y)f_y\big)_y\Big).\]
\[ \begin{split} 
-(x-y)^4\left(\frac{dd^c\Scal(\omega)\wedge \omega}{\omega^2}\right) =&  
\frac{(x-y)^2}{2}\left(\Big(A(x)A'''(x)\Big)' + \Big(B(y)B'''(y)\Big)'\right) \\
&- (x-y)\Big((A(x)A'''(x)-B(y)B'''(y)\Big) \\
& -\frac{(x-y)}{2}\left(\Big(A'(x)-B'(y)\Big)\Big(A''(x)+ B''(y)\Big)\right) \\
&+ \Big((A(x)+B(y)\Big)\Big(A''(x) + B''(y)\Big).
\end{split}\]

\[\begin{split} 2(x-y)^4\left(\frac{\rho^2(\omega)}{\omega^2}\right)  = &
 \frac{(x-y)^2}{2} \Big(A''(x)B''(y)\Big) + \frac{(x-y)}{2}\Big(A'(x)+B'(y)\Big)\Big(A''(x)-B''(y)\Big) \\ &-\frac{1}{2}\left(A'(x)+B'(y)\right)^2.
\end{split}\]
We then get
\begin{equation}\label{orthotoric-reduction}
\begin{split} 
&(x-y)^4\mathcal{S}(\omega)=(x-y)^4\left(\frac{-dd^c \Scal(\omega)\wedge \omega + 2\left(\al_{\lambda}^2 + \rho(\omega)^2\right)}{\omega^2}\right)  \\
=&\frac{(x-y)^2}{2}\left(\Big(A(x)A'''(x)\Big)' + \Big(B(y)B'''(y)\Big)' + \Big(A''(x)B''(y)\Big)\right) \\
&-(x-y)\Big(A(x)A'''(x) - B(y)B'''(y) +A'(x)B''(y)-B'(y)A''(x)\Big) \\
&+\Big(A(x) + B(y)\Big)\Big(A''(x)+B''(y)\Big) -\frac{1}{2}\Big(A'(x) + B'(y)\Big)^2 -2\lambda^2.
\end{split} 
\end{equation}
We now look for polynomial solutions $A(x)$ and $B(y)$ 
of the  equation $\mathcal{S}(\omega)=0$.
If the degree of $A(x)$ is $k$, the LHS of \eqref{orthotoric-reduction} contains a maximal degree monomial
\[ 
\begin{split}
\frac{k(k-2)^2(2k-3)}{2} a_0^2 x^{2k-2}.
\end{split}\]
 It thus follows that a polynomial solution can only occur for ${\rm deg}(A)\leq 2, \, {\rm deg}(B) \leq 2$.  We next analyse this case in detail.
Letting
\[ A(x)=a_0x^2 + a_1x+a_2, \qquad B(y)=b_0y^2 +b_1y+ b_2\]
and substituting in \eqref{orthotoric-reduction}, we get
\begin{equation}\label{quadratic-solution} \begin{split} &(x-y)^4\mathcal{S}(\omega) = -2a_0b_0(x-y)^2 -2(x-y)(b_0a_1 -a_0b_1) \\
&+2(a_0x^2 +a_1x + a_2 + b_0y^2 +b_1y+b_2)(a_0+b_0) -\frac{1}{2}\left(2a_0x + a_1 + 2b_0y+b_1\right)^2-2\lambda^2\\
=&-2a_0b_0(x-y)^2-2(x-y)(b_0a_1-a_0b_1) +2a_0b_0(x^2+y^2) +2b_0a_1x + 2a_0b_1y \\
&+2(a_2+b_2)(a_0+b_0)-4a_0b_0xy-2a_0b_1x-2b_0a_1y-\frac{1}{2}(a_1 + b_1)^2-2\lambda^2 \\
&= 2(a_0+b_0)(a_2+b_2)-\frac{1}{2}(a_1+b_1)^2 -2\lambda^2.
\end{split} 
\end{equation}
The above shows that we can find explicit solutions of $\mathcal{S}(\omega)=0$ as soon as 
\begin{equation}\label{admissible} 
4(a_0+b_0)(a_2+b_2) \geq (a_1+b_1)^2.
\end{equation}
To address this last point, let us scale $A(x)$ by a real constant $t>0$. Then  the coefficients of  $(t A, B)$ will satisfy \eqref{admissible} for the values of $t$ such that
\begin{equation*}\label{label-criterion} 
\begin{split}
f(t)&:=4(t a_0+b_0)(t a_2+b_2) -(t a_1+b_1)^2\\
&=-(a_1^2 -4a_0a_2)t^2 +2(2a_0b_2+ 2b_0a_2- a_1b_1)t -(b_1^2-4b_0b_2) \geq 0.
\end{split}\end{equation*}
The case of interest for constructing compact orbifold examples is when $A(x)$ and $B(y)$ both have two distinct real roots, i.e. the discriminants  of $A$ and $B$ satisfy
\[ a_1^2 -4a_0a_2 >0 \qquad b_1^2-b_0b_2 >0.\]
We can then find positive $t$'s with $f(t)\geq 0$ iff the discriminant $D(f)$ of $f(t)$ is non-negative. A direct computation shows 
\[\begin{split}
\frac{D(f)}{4}=R(A,B)=&a_0^2b_2^2 + b_0^2a_2^2 - 2a_0a_2b_0b_2-a_0a_1b_1b_2 -a_1a_2b_0b_1 +a_1^2b_0b_2 +b_1^2a_0a_2 \geq 0,
\end{split}\]
where $R(A,B)$ is the resultant of $A$ and $B$. If $\ap_1<\ap_2$ and $\bp_1<\bp_2$ are the real roots of $A$ and $B$, respectively,  we know that 
$R(A,B)=A(\bp_1)A(\bp_2)B(\ap_1)B(\ap_2)$. This quantity is strictly positive provided that $\bp_1 <\bp_2 < \ap_1 <\ap_2$ and $A(x)>0$ on $(\ap_1, \ap_2)$ and $B(x)>0$ on $(\bp_1, \bp_2)$.
Noting that  the scalar curvature is 
\begin{equation}\label{ortho-scalar} 
\Scal(\omega)= -\left(\frac{a_0+b_0}{x-y}\right)>0,\end{equation}
the conclusion of this analysis can be summarized in the following
\begin{prop}\label{p:othotoric} Suppose $\Delta$ is a compact convex quadrilateral in $\R^2$, which does not have a pair of parallel facets. Then there exists a 2-parameter family of labels ${\bf L}_{t_1, t_2}$ of $\Delta$
%\footnote{Generic quadrilaterals are parametrized by $2$ real parameters. For any quadrilateral, there is a $2$-parameter family of labels ${\bf L}_{t_1,t_2}$ for which the explicit solutions exist. As an overall scale of ${\bf L}_{t_1,t_2}$ will produce homothetic solutions of \eqref{PDEA} (and hence isometric BHE 3-folds), we get solutions parametrized by 3 effective real parameters.}, 
such that $(\Delta, {\bf L}_{t_1, t_2})$ admits an orthotoric K\"ahler metric $\omega$ verifying the Abreu-Guillemin boundary conditions and satisfying \eqref{PDEA}. For two different values of $t_1/t_2$, the solution satisfies \eqref{PDEA} with $\al_{\varphi}=0$. 
\end{prop}
\begin{proof} E. Legendre~\cite{legendre} has shown that given such a polytope $\Delta$, there exists a  choice of real numbers $\bp_1 <\bp_2 < \ap_1 <\ap_2$ such that the  polytope $\Delta_{\ap, \bp}$ given by \eqref{ortho-labels} is affine isomorphic to $\Delta$. We can thus assume that $\Delta=\Delta_{\ap, \bp}$. We then consider the quadrics 
\[ A_0(x) = -\frac{2}{(\ap_2-\ap_1)}(x-\ap_1)(x-\ap_2), \qquad B_0(y):= -\frac{2}{(\bp_2-\bp_1)}(y-\bp_1)(y-\bp_2),\]
which define an orthotoric metric on the interior of $\Delta_{\ap,\bp}$,   satisfying the Abreu--Guillemin boundary conditions with respect to the labels 
\[{\bf L}_{1,1} := \left\{ \ell_{\ap_1}(\mu), \ell_{\ap_2}(\mu), \ell_{\bp_1}(\mu), \ell_{\bp_2}(\mu) \right\}.\]
Letting
\[{\bf L}_{t_1,t_2} := \left\{ t_1 \ell_{\ap_1}(\mu), t_1\ell_{\ap_2}(\mu), t_2\ell_{\bp_1}(\mu), t_2\ell_{\bp_2}(\mu) \right\}, \, \, t_1, t_2>0,\]
we obtain a two-parameter family of labels of $\Delta_{\ap, \bp}$ with respect to which the orthotoric metric with quadratic functions $(t_1A_0(x), t_2B_0(y))$ satisfies the Abreu--Guillemin boundary conditions. As $R(A_0, B_0)>0$, we can choose $t=t_1/t_2$ in an interval so that $f(t)\geq 0$, i.e. the inequality  \eqref{admissible} holds for $t_1A(x), t_2B(y)$. We can then determine $\lambda$ from the equality 
\begin{equation}\label{lambda}
2(a_0+b_0)(a_2+b_2) -\frac{1}{2}(a_1+b_1)^2=2 \lambda^2.\end{equation}
By \eqref{quadratic-solution}, this yields a solution of \eqref{PDEA}  satisfying the Abreu--Guillemin boundary conditions on $(\Delta=\Delta_{\ap, \bp}, {\bf L}_{t_1, t_2})$. $\theta_{\lambda}=0$ corresponds to the case when $t_1/t_2$ equals to one of the two roots of $f(t)=0$. \end{proof}
\begin{rmk} Letting all the data $\ap_i, \bp_j, t_k$ be rational numbers, Proposition~\eqref{p:othotoric} gives rise to a familly of non-trivial solutions of \eqref{PDEA} on compact complex orbifold surfaces.
\end{rmk}
\subsection{Proof of Theorem~\ref{thm:main4}}
\begin{proof}[Proof of Theorem~\ref{thm:main4}] We start with a polytope $(\Delta_{\ap, \bp}, {\bf L}_{t_1,t_2})$ given by Proposition~\ref{p:othotoric}, where $\bp_1< \bp_2<\ap_1<\ap_2$ are integers, and $(t_1,t_2)$ are  co-prime integers such that $f(t_1/t_2)>0$. 
By construction, $(\Delta_{\ap, \bp}, {\bf L}_{t_1,t_2})$  is then a \emph{rational} Delzant polytope, meaning that  the  differentials of the affine linear functions  in ${\bf L}_{t_1,t_2}$ span a rank two lattice $\Lambda \subset \R^2$. It follows from \cite{LT} (see also the treatment in \cite{ACGTF-jdg}) that $(\Delta_{\ap, \bp}, {\bf L}, \Lambda)$ defines a simply connected compact toric orbifold $(M, \omega)$, whose  Delzant image is $\Delta_{\ap, \bp}$. By Proposition~\ref{p:othotoric} and \cite[Prop.9]{ACGTF-jdg}, $M$ admits a toric K\"ahler orbifold metric $(g, \omega, J)$ obtained by the orthotoric  ansatz \eqref{orthotoric} and degree $2$ polynomial functions $A(x)=t_1A_0(x), \, B(y)=t_2B_0(y)$, such that $(g,\omega, J)$ satisfies \eqref{PDEA} with $\al_{\lambda} \neq 0$.  By \cite[Rem.2.16]{ABLS}, changing the sign of $\al_{\lambda}$ (or equivalently the sign of $\lambda$  determined via \eqref{lambda}) leads to BHE's on the same riemannian manifold $(N, g_N)$ with complex structures related by changing sign on the orthogonal compliment of $\VV$, so we shall assume without loss that $\lambda>0$. By \cite{ACG,ACG-crelle}, 
\[\omega_-:=\al_{\lambda} \qquad g_-:=\frac{\lambda}{(x-y)^2}g\] define another toric K\"ahler metric on $M$, with complex structure denoted $J_-$ inducing the opposite orientation of the one induced by $(g, \omega)$. In the terminology of \cite{ACG-crelle}, $(g_-, \omega_-, J_-)$ is a (regular) ambitoric structure obtained by the \emph{Segre factorization structure} from the same pair of quadrics $A(x)$ and $B(y)$. As $A(x)$ and $B(x)$ are polynomials of degree $\leq 3$, it follows by \cite[Cor.2 \& Prop.6]{ACGL} that  $(M, g_-, \omega_-, J_-)$ can be identified with the Levi--K\"ahler quotient  of the co-dimension $2$ CR-manifold \[ N= (S^3\times S^3, {{\mathcal D}}, J_0) \subset \C^2 \times \C^2.\] More concretely, $(N, {{\DD}}, J_0)$ is invariant under the linear  action on $\C^4$ of the $4$-dimensional torus $\T^4$, and there exists a $2$-dimensional sub-torus $\T^2\subset \T^4$  with Lie algebra $\mathcal V \subset {\rm Lie}(\T^4)$ (determined from the toric geometry of $(M, g_-, \omega_-, J_-)$),  and an element $\lambda^* \in \mathcal{V}^*\setminus \{0\}$ with the  following properties: 
\begin{itemize}
    \item The sub-algebra $\mathcal V \subset {\rm Lie}(\T^4)$ gives rise to a point-wise splitting
\[ T_pN = \mathcal{V}_p \oplus {{\mathcal D}}_p,\]
where  the subspace $\mathcal{V}_p\subset T_pN$ is generated by the fundamental vector fields (evaluated at $p$) of elements of $\mathcal V$. In the analogy with the much studied Sasaki geometry, the above decomposition allows one to  define a connection $1$-form $\eta \in \Omega^1(N, \mathcal V)$, acting tautologically on $\mathcal{V}_p$. 
\item The choice of $\lambda^* \in {\mathcal V}^* \setminus \{0\}$ is such that  the real-valued $1$-form $\eta_{\lambda}:= \langle \eta, \lambda^*\rangle$ on $N$ gives rise to a $\mathcal{V}$-basic $2$-form $d\eta_{\lambda}$ and $({\DD}, J_0, d\eta_{\lambda})$ gives rise to a transversal K\"ahler structure on $(N,\mathcal{V})$,  which induces the K\"ahler structure $(g_-, \omega_-, J_-)$ on the orbifold  $M=N/\T^2$. 
\end{itemize}
The second Betti number of the orbifold $M$ equals $2$: this follows for instance from the Hodge decomposition, using that $M$ is a toric K\"ahler orbifold surface with respect of either orientation, and thus $H^{2,0}(M, J)=0=H^{2,0}(M, J_-)$ and ${\rm dim}(H^{1,1}(M, J))={\rm dim}(H^{1,1}(M, J_-))=2$. Equivalently, the second basic deRham cohomology of $(N, \mathcal{V})$ is $2$-dimensional.  We can then find another non-zero element $c^*\in \mathcal{V}^*\setminus\{ 0\}$ such that the $1$-form $\eta_c:= \langle \eta, c^* \rangle$ satisfies
\[ [d\eta_c]_B = -[\rho(\omega)]_B, \]
where $\rho(\omega)$ is the transversal Ricci form of the initial orthotoric structure $(g, \omega, J)$ on $M$,  seen as a basic $2$-form on $(N, \mathcal{V})$ by the pulling it back via the quotient map. The vectors $(\lambda^*, c^*)$  form a basis of $\VV^*$ (this follows as $\rho(\omega)^2$ and $\al_{\lambda}^2$ have opposite signs on $M$) whose dual basis of $\VV$ will be denoted $\{V, W\}$.  Identifying the $\VV$-basic $2$-form $d\eta_c$ with a $2$-form on $M$, a key point is that both $d\eta_c$ and $\rho(\omega)$ are of type $(1,1)$ with respect $J_-$, see \cite{ACGL} and \eqref{orthotoric-Ricci}.  By the  $dd^c_{J_-}$-lemma, we can find a smooth function $\varphi$ on $M$ such that 
\[ -\rho(\omega) = d\eta_c + dd^c_{J_-} \varphi.\]
Letting 
\[ \tilde \eta_c:= \eta + d^c_{J_-} \varphi, \qquad   \mathcal{H}_p := {\rm ann}(\eta_{\lambda}(p), \tilde \eta_{c}(p)),\]
we have a $\VV$-invariant splitting
\[ T_pN = \mathcal{V}_p \oplus \mathcal{H}_p.\]
The  pullback of the initial orthotoric structure $(g, J, \omega)$ on $M$ to  $N$ defines a $\VV$-invariant (transversal) K\"ahler structure $(g_K^T, J^T, \omega_K^T)$ on $\mathcal{H}$.  We can further extend  $J^T$ to an almost-complex structure $J_N$ on $N$  by 
$J_N(V):=W, \, J_N (W)=-V$,  and define a compatible Hermitian metric on $N$ by $g_N :=\frac{\Scal(\omega_K^T)}{2}g_K + \eta_{\lambda}^2 + \tilde \eta_c^2$.
By  \cite[Thm.1.1]{ABLS}, we obtain a 
quasi-regular Bismut Ricci flat pluriclosed Hermitian metric on $N=S^3\times S^3$. By \cite[Prop.2.21]{ABLS} and using that $b_2(M)=2$, we have that the $(1,1)$ Bott-Chern cohomology group of $(N,J_N)$ is $2$-dimensional.
Notice that the transversal scalar curvature of $g_K^T$ is given by \eqref{ortho-scalar}, and therefore is not constant. It follows from Theorem~\ref{thm:reduction} that $(g_N, J_N)$ is not isometric to the Samelson geometry $SU(2)\times SU(2)$.  \end{proof}
\begin{rmk}  Proposition~\ref{p:othotoric} and the construction in \cite{ACGL} lead more generally to a continuous family of (not quasi-regular in general) BHE examples on $S^3\times S^3$, parametrized by generic quadrilaterals and two parameter family of labels associated to each quadrilateral. As on overall scale of the labels lead to the same structure on $S^3\times S^3$, this gives rise to BHE structures on $S^3\times S^3$ parametrized by $3$ effective real parameters. The construction detailed in the proof of Theorem~\ref{thm:main4} above corresponds to a countable infinite subfamily of quasi-regular ones. When the generic quadrilaterals converge to a parallelogram, the corresponding BHE geometries on $S^3\times S^3$ converge to the homogeneous Samelson geometry $SU(2)\times SU(2)$. We postpone a detailed study of irregular BHE $3$-folds and their deformations to a future work. \end{rmk}

\subsection{Proof of Theorem~\ref{thm:main3}}
\begin{proof}[Proof of Theorem~\ref{thm:main3}]
In \cite[Sec.4.2]{CGMS},  the authors found  orthotoric  metrics with quadratic functions $A(x)$ and $B(y)$ as above, satisfying the relations
\[ a_1+b_1=0 \qquad a_2+b_2 =0.\]
Notice that such solutions satisfy \eqref{PDEA} with $\al_{\varphi}=0$, and thus are solutions of the ``box equation'' arising in the \emph{GK geometry}~\cite{GK},  and further studied in \cite{CGMS}. A class of solutions can be parametrized by $3$ positive real numbers ${\rm a,b, c}$ as follows:
\[ A_{\rm a,b,c}(x)=-(x-1)\left(x-\frac{\rm c}{\rm d}\right), \qquad B_{\rm a,b,c}(y) =-\left(\frac{\rm ab}{\rm cd}\right)\left(y+ \frac{\rm c}{\rm a}\right)\left(y-\frac{\rm c}{\rm b}\right), \]
where $ {\rm d} :=-{\rm a}  + {\rm b} - {\rm c}$.
For suitable (infinite) choices of positive integers ${\rm a, b, c}$, it is shown in \cite{CGMS} that \eqref{orthotoric} compactifies to an orbifold orthotoric complex surface $S_{\rm a,b,c}$ obtained as the orbit space of a semi-free $S^1$-action on a \emph{smooth} $5$-dimensional CR-manifold $(Y, \DD, J)= \mathcal{L}^{\rm a,b,c}$; furthermore, if $\xi$ denotes the generator of this $S^1$-action, $\xi$ is a CR vector field transversal to $\DD$ at each point, and the connection $1$-form $\eta_{\xi}$ it defines has curvature $d\eta_{\xi} = c \pi^*(\rho(\omega))$ for some non-zero constant $c$. By arguments similar to the proof of Theorem~\ref{thm:main4}, one can show that the manifold $\mathcal{L}^{\rm a,b,c}$ is obtained as a quotient of $S^3\times S^3$ by a suitable circle subgroup of $\T^4$: using Gysin's sequence and Smale's classification of compact imply connected spin $5$-manifolds, this implies that ${\mathcal L}^{\rm a,b,c}$ are all diffeomorphic to $S^2 \times S^3$, a fact  established in \cite{CGMS}.

Let $N:=S^1 \times Y$, $V$ be the generator of the $S^1$-factor,  $J_NV:= -c\xi$  and $\VV := {\rm span}(V, J_NV) \subset TN$. Pulling the orthotoric metric to $N$ via the splitting $TN = \VV \oplus \DD$, 
we define a BHE $6$-manifold via the inverse construction of Theorem~\ref{thm:reduction} (see \cite[Thm.~1.1]{ABLS} for details).  To show that that $H^{1,1}_{\rm BC}(N, J_N)$ is $2$-dimensional, we use the exact sequence \cite[(2.17)]{ABLS}. In our case,  $H^{1,0}_{\rm BC}(N, \VV)=H^{1,0}(S_{\rm a,b,c}, \C)=0$ by the toric property of the orbifold quotient and $H^{1,1}_{\rm BC}(N, \VV,\R)= H^{1,1}(S_{\rm a,b,c}, \R)=2$ by the ambitoric property of the orbifold quotient. Note that the scalar curvature of the transversal orthotoric K\"ahler structure is not constant (see \eqref{ortho-scalar}). 
\end{proof}

\appendix
\section{Stability condition associated to \eqref{PDEAn}} 
In the case $\A=0$, Dervan~\cite{Dervan} computed the slope of the Mabuchi functional $\mathcal{E}$  along smooth rays of K\"ahler metrics defined by a K\"ahler test configuration of $(M, \Ome)$ with smooth central fibre, and showed that it is a topological invariant of the test configuration, independent of the choice of the ray. Here we extend this result to $\A\neq 0$ and possibly a singular central fibre.

\smallskip
\begin{defn} A \emph{K\"ahler test configuration} $({\tstM}, {\tstA})$ associated to a K\"ahler manifold $(M, \Ome)$ is a  \emph{normal} compact K\"ahler space ${\tstM}$ endowed with 
\begin{itemize}
\item a flat morphism $\pi : \mathcal{\tstM} \to \mathbb{P}^1$;
\item a $\C^*$-action $\lambda$  on $\tstM$ covering the standard $\C^*$-action on $\mathbb{P}^1$;
\item 
 a $\C^*$-equivariant biholomorphism  $\Pi_0 : (\mathcal{\tstM} \setminus \pi^{-1}(0)) \cong (M\times (\PP^1\setminus \{0\}))$ where  the $\C^*$-action $\lambda_0$ on the RHS is the trivial action on  $M$ and the standard action on $\PP^1\setminus\{0\}=\C$;
\item a K\"ahler class $\mathcal{A}\in H^{1,1}({\tstM}, \R)$ such that $(\Pi_0^{-1})^*(\mathcal{\tstA})_{|_{M \times \{\tau\}}} = \Ome$.
\end{itemize}
We say that $(\mathcal{M}, \mathcal{A})$ is \emph{dominating} if $\Pi_0$  extends to
a $\C^*$-equivariant  morphism 
\begin{equation}\label{dominate}
 \Pi : \mathcal{\tstM} \to M \times \PP^1.
 \end{equation}
\end{defn}
\noindent
We now introduce K\"ahler test configurations adapted to our problem \eqref{PDEAn}.

\begin{defn}\label{Def:SKTC} A K\"ahler test configuration $(\mathcal{\tstM}, \mathcal{\tstA}, \mathcal{\tstB})$ associated to $(M, \Ome,\beta)$ is a K\"ahler test configuration $(\mathcal{M}, \mathcal{\tstA})$  for $(M, \Ome)$ endowed with a deRham class $\mathcal{\tstB}\in H^{1,1}(M,\R)$ such that $(\Pi_0^{-1})^*(\mathcal{\tstB})_{|_{M \times \{\tau\}}} = \beta$.
\end{defn}
\begin{rmk} Any dominating K\"ahler test configuration $(\tstM, \tstA)$ of $(M, \Ome)$ can be associated to $(M, \Ome, \A)$ by letting  $\tstB:=\Pi^*(\widehat\A)$ where $\Pi$ is the dominating map \eqref{dominate} and $\widehat\A$ is the pullback of $\A$ to   $M\times \PP^1$ by the projection $p_M: M \times \PP^1 \to M$.
\end{rmk}
Let $(\mathcal{\tstM}, \mathcal{\tstA},\mathcal{\tstB})$ be a dominating K\"ahler test configuration of 
$(M, \Ome, \A)$. We assume, for simplicity, that $\tstM$ is smooth or  an orbifold.  Let $\Omega \in 2\pi \mathcal{A}$ be an ${S}^1$-invariant K\"ahler metric on $\mathcal{\tstM}$ with $\omega_0:=\Omega_{|_{M\times\{1\}}}$ the corresponding K\"ahler metric on $M$. We denote $\omega_t \in \Ome$ the smooth ray of K\"ahler metrics on $M$ defined by
\begin{equation}\label{eq:ray-met}
\omega_t := \lambda(e^{-t + i s})^*(\Omega)_{|_{M\times\{1\}}}.    
\end{equation} 
We shall establish the following slope formula for the Mabuchi functional introduced in Definition~\ref{d:Mabuchi}:
\begin{thm}\label{E-Slope}
Let $({\tstM}, {\tstA},{\tstB})$ be an orbifold dominating K\"ahler test configuration (see Definition \ref{Def:SKTC}) associated to $(M,\Ome,\beta)$, where  $\Ome$ is a K\"ahler class on $M$ and  $\beta\in H^{1,1}(M,\R)$ satisfy the cohomological conditions \eqref{topo-n}. If the central fibre ${M}_0$ is reduced, then
\begin{equation*}
\mathcal{E}_{\rm NA}(\tstM, \tstA, \tstB):=\underset{t\rightarrow +\infty}{\lim}\frac{\mathcal{E}(\omega_t)}{t}={\color{red} -}(2\pi)^{n+1} \left[\left(c_1(\tstK_{{\tstM}/{\mathbb{P}^1}})^2+{\tstB}^2\right)\cdot{\tstA}^{[n-1]}- \left(\frac{c_{\Ome, \A}}{2}\right)\tstA^{[n+1]}\right]_{{\tstM}},
\end{equation*}
where $c_{\Ome,\A}= 2\left(\frac{(c_1^2(K_M) + \A^2)\cdot \Ome^{[n-2]}}{\Ome^{[n]}}\right)_M$.
\end{thm}
The  proof will be given in  the reminder of the section.

\smallskip
We first decompose the Mabuchi functional   $\mathcal{E} : \mathcal{H}_{\omega_0} \to \R$ 
\[ \mathcal{E}=\mathcal{E}^0+\mathcal{P}^{\A} + \left(\frac{c_{\Ome,\A}}{2}\right)\mathcal{I},
\]
where
\begin{equation}
\begin{split}\label{eq:E0H}(d\mathcal{E}^0)_{\varphi}(\dot{\varphi}):=&-\int_M\dot{\varphi}\left(-\frac{1}{2}\left(dd^c\Scal_\varphi\right)\wedge\omega_{\varphi}^{[n-1]}+\rho_{\varphi}^2\wedge\omega_{\varphi}^{[n-2]}\right), \qquad \mathcal{E}^0(0)=0,\\
(d{\mathcal P}^{\A})_\varphi(\dot{\varphi}):=&-\int_M\dot{\varphi}\left(\theta_\varphi^2\wedge\omega_\varphi^{[n-2]}\right), \qquad \qquad \qquad \qquad  \qquad  \qquad   \mathcal{P}^{\A}(0)=0,\\
(d\mathcal{I})_{\varphi}(\dot \varphi) :=& \int_M \dot \varphi \, \omega_{\varphi}^{[n]}, \qquad \qquad \qquad \qquad \qquad \qquad \qquad \qquad \, \, \,  \, \,  \, \mathcal{I}(0)=0.
\end{split}    
\end{equation}
\begin{lemma} $\mathcal{E}^0$ and $\mathcal{P}^{\A}$ are equivalently given by
\begin{align}\label{eq:CTE0}
\mathcal{E}^0(\varphi)=&-\frac{1}{(n-1)!}\sum_{j=0}^{n-2}\int_M \varphi\omega_\varphi^{j}\wedge\rho_0^{2}\wedge\omega_0^{n-j-2}+\frac{1}{2}\int_M\log\left(\frac{\omega_\varphi^n}{\omega_0^n}\right)\omega_\varphi^{[n-1]}\wedge\rho_0 \\ \nonumber
&+\frac{1}{2}\int_M\log\left(\frac{\omega_\varphi^n}{\omega_0^n}\right)\omega^{[n-1]}_\varphi\wedge\rho_\varphi, \\\label{eq:CTH}
\mathcal{P}^{\A}(\varphi)=&-\frac{1}{(n-1)!}\sum_{j=0}^{n-2}\int_M \varphi\omega_\varphi^{j}\wedge\theta_0^{2}\wedge\omega_0^{n-j-2}+\int_M\mathbb{G}_0\left(\Lambda_0\theta_\varphi\right)\theta_0\wedge\omega_\varphi^{[n-1]}.
\end{align}
\end{lemma}
\begin{proof}
The identity \eqref{eq:CTE0} follows from \cite[Proposition 3.9]{Dervan}.  To obtain the expression for the ``potential'' energy $\mathcal{P}^{\A}$, we compute below the derivative of the RHS of \eqref{eq:CTH} and compare it with \eqref{eq:E0H}. Similarly to \eqref{variation-alpha}, we have
\[ \dot{\theta}_\varphi=-dd^c\mathbb{G}_{\varphi} \langle\theta_\varphi, dd^c\dot{\varphi}\rangle_{\varphi},\]
so we compute,  using integration by parts, 
\[
\begin{split}
-(d\mathcal{P}^{\A})_\varphi(\dot{\varphi})
=&\frac{1}{(n-1)!}\sum_{j=0}^{n-2}\int_M \dot\varphi\omega_\varphi^{j}\wedge\theta_0^{2}\wedge\omega_0^{n-j-2}+\int_M j\varphi dd^c\dot\varphi\wedge\omega_\varphi^{j-1}\wedge\theta_0^{2}\wedge\omega_0^{n-j-2}\\\
&+\int_M\mathbb{G}_0\underbrace{\left(\Lambda_0 dd^c\mathbb{G}_{\varphi} \langle\theta_\varphi, dd^c\dot{\varphi}\rangle_{\varphi}\right)}_{=-\Delta_0( \mathbb{G}_{\varphi}\langle \al_{\varphi}, dd^c\varphi\rangle_{\varphi})}\omega_\varphi^{[n-1]}\wedge\theta_0 -\int_M\mathbb{G}_0\left(\Lambda_0\theta_\varphi\right)\omega_\varphi^{[n-2]}\wedge dd^c\dot\varphi\wedge\theta_0\\
=&\frac{1}{(n-1)!}\sum_{j=0}^{n-2}\int_M\dot\varphi\omega_\varphi^{j}\wedge\theta_0^{2}\wedge\omega_0^{n-j-2}  +\int_M j\dot\varphi(\omega_\varphi-\omega_0)\wedge\omega_\varphi^{j-1}\wedge\theta_0^{2}\wedge\omega_0^{n-j-2}\\
&-\int_M \left(\mathbb{G}_{\varphi} \langle\theta_\varphi, dd^c\dot{\varphi}\rangle_{\varphi}-\frac{1}{\Vol(\Ome)}\int_M\mathbb{G}_{\varphi} \langle\theta_\varphi, dd^c\dot{\varphi}\rangle_{\varphi}\omega_0^{[n]}\right)\, \omega_\varphi^{[n-1]}\wedge\theta_0\\
&-\int_M \dot\varphi dd^c\mathbb{G}_0\left(\Lambda_0\theta_\varphi\right)\wedge \omega_\varphi^{[n-2]}\wedge\theta_0 \\
=& \frac{1}{(n-1)!}\sum_{j=0}^{n-2}\int_M(j+1)\dot\varphi\omega_\varphi^{j}\wedge\theta_0^{2}\wedge\omega_0^{n-j-2}  -\int_M j\dot\varphi\omega_\varphi^{j-1}\wedge\theta_0^{2}\wedge\omega_0^{n-j-1}\\
&-\int_M  \langle\theta_\varphi, dd^c\dot{\varphi}\rangle_{\varphi}\mathbb{G}_{\varphi}(\Lambda_\varphi\theta_0)\omega_\varphi^{[n]} +\underbrace{\frac{\left(\beta\cdot\Ome^{[n-1]}\right)}{\Ome^{[n]}}}_{\A\cdot \Ome^{n-1}=0}\int_M\mathbb{G}_{\varphi} \langle\theta_\varphi, dd^c\dot{\varphi}\rangle_{\varphi}\omega_0^{[n]}\\
&-\int_M \dot\varphi dd^c\mathbb{G}_0\left(\Lambda_0\theta_\varphi\right)\wedge \omega_\varphi^{[n-2]}\wedge\theta_0.
\end{split}  \]
By the $dd^c$-Lemma (see e.g. \cite[Lemma~1.17.1]{gauduchon-book})  and the K\"ahler identity $[\Lambda, d^c]= \delta$, we have
\[\theta_{\varphi} = \theta_0 + d\delta_0 \mathbb{G}_0 \theta_{\varphi} = \theta_0 - dd^c\mathbb{G}_0\left(\Lambda_0\al_{\varphi}\right) =\theta_0 + dd^c \mathbb{G}_{\varphi}\left(\Lambda_{\varphi} \theta_0\right).\]
Using the above, we continue the computation
\[
\begin{split}
-(d\mathcal{P}^{\A})_{\varphi}(\dot \varphi)=& \frac{1}{(n-1)!}\sum_{j=0}^{n-2}\int_M(j+1)\dot\varphi\omega_\varphi^{j}\wedge\theta_0^{2}\wedge\omega_0^{n-j-2}  -\int_M j\dot\varphi\omega_\varphi^{j-1}\wedge\theta_0^{2}\wedge\omega_0^{n-j-1}\\
&-\int_M  \langle\theta_\varphi, dd^c\dot{\varphi}\rangle_{\varphi}\mathbb{G}_{\varphi}(\Lambda_\varphi\theta_0)\omega_\varphi^{[n]} +\int_M \dot\varphi (\theta_\varphi-\theta_0)\wedge \omega_\varphi^{[n-2]}\wedge\theta_0\\
=&\int_M \dot\varphi \theta_\varphi\wedge \omega_\varphi^{[n-2]}\wedge\theta_0
+\int_M \mathbb{G}_{\varphi}(\Lambda_\varphi\theta_0) \theta_\varphi\wedge dd^c\dot{\varphi}\wedge\omega_\varphi^{[n-2]} \\
=&\int_M \dot\varphi \theta_\varphi\wedge \omega_\varphi^{[n-2]}\wedge\theta_0
+\int_M \dot{\varphi} \theta_\varphi\wedge dd^c \mathbb{G}_{\varphi}(\Lambda_\varphi\theta_0)\wedge\omega_\varphi^{[n-2]}\\
=&\int_M \dot\varphi \theta_\varphi\wedge \omega_\varphi^{[n-2]}\wedge\theta_0
+\int_M \dot{\varphi} \theta_\varphi\wedge(\theta_\varphi-\theta_0)\wedge\omega_\varphi^{[n-2]}
=\int_M \dot{\varphi} \theta_\varphi^2\wedge\omega_\varphi^{[n-2]}.
\end{split}
\]
\end{proof}
\begin{rmk}
Using $\mathbb{G}_0\left(\Lambda_0\theta_\varphi\right)=-\mathbb{G}_\varphi\left(\Lambda_\varphi\theta_0\right)+\frac{1}{\Vol (\Ome)}\int_M\mathbb{G}_\varphi\left(\Lambda_\varphi\theta_0\right)\omega_0^{[n]}$ and $\A\cdot \Ome^{[n-1]}=0$ we 
get the equivalent expression for $\mathcal{P}^{\A}$:
\begin{equation}
    \begin{split}\label{eq:H-bis}
    \mathcal{P}^{\A}(\varphi)=&\frac{1}{(n-1)!}\sum_{j=0}^{n-2}\int_M \varphi\omega_\varphi^{j}\wedge\theta_0^{2}\wedge\omega_0^{n-j-2}+\int_M\mathbb{G}_\varphi\left(\Lambda_\varphi\theta_0\right)\theta_0\wedge\omega_\varphi^{[n-1]} 
\end{split}
\end{equation}
\end{rmk}

Away from the central fibre ${M}_0$ there exists a smooth function $\Phi$ on $\mathcal{\tstM}^\times:={\tstM}\setminus{M}_0$ such that
\begin{equation}\label{eq:Tot-pot} \Omega=\widehat{\omega}_0+dd^c\Phi,\qquad \widehat{\omega}_0:=(p_M\circ\Pi)^*(\omega_0)
\end{equation}
where $\Pi$ is the domination map \eqref{dominate} and $p_M$ is the projection on the first factor of $M\times\mathbb{P}^1$. It follows from \eqref{eq:ray-met} that
\begin{equation}\label{eq:ray-pot}
\varphi_t:= \lambda(e^{-t + i s})^*(\Phi)_{|_{M\times\{1\}}}
\end{equation} 
is a ray of K\"ahler potentials in $\mathcal{H}_{\omega_0}$ such that $\omega_t=\omega_0+dd^c\varphi_t$.

We next use the expressions \eqref{eq:CTE0} and \eqref{eq:CTH} to compute the  slopes at infinity $\mathcal{E}^{0}_{\rm NA}$, $\mathcal{P}^{\A}_{\rm NA}$ of the functionals $\mathcal{E}^{0}$, $\mathcal{P}^{\A}$,  respectively, along the ray $\varphi_t$. In what follows, $\omega_{\rm FS}$ denotes a Fubini--Study metric of scalar curvature $2$ on $\PP^1$  so that $\rho_{\omega_{\rm FS}} =\omega_{\rm FS}$,  and  the upper symbol $\widehat{\cdots}$ stands for forms on $M\times \PP^1$ which are pullbacks of forms one of the factors ($M$ or $\PP^1$) via the projection maps $(p_M$ or $\pi$). Using the K\"ahler metric $\widehat{\omega}_0+\widehat\omega_{\rm FS}$ on $\mathcal{\tstM}^\times$, we have
\begin{equation}\label{eq:tot-Ric}
    \rho_\Omega-\widehat\omega_{\rm FS}-\widehat{\rho}_{0}=-\frac{1}{2}dd^c\Psi, \quad \Psi:=\log\left(\frac{\Omega^{n+1}}{\widehat{\omega}_0^{n}\wedge\widehat{\omega}_{\rm FS}}\right).
\end{equation}
We denote $\psi_t$ the ray of smooth functions on $M$ defined by
\begin{equation}\label{eq:psi_t}
\psi_t:=\lambda(e^{-t + i s})^*(\Psi)_{|_{M\times\{1\}}}.
\end{equation}
Let $\Theta_\Omega$ be the Harmonic representative of $2\pi \mathcal{B}$ relative to $\Omega$ and $\theta_0$ the harmonic representative of $2\pi \A$ relative to $\omega_0$. There exists a smooth function $U$ on $\mathcal{\tstM}^\times:=\mathcal{\tstM}\setminus{M_0}$ such that
\begin{equation}\label{eq:Theta-pot} \Theta_\Omega=\widehat{\theta}_0+dd^cU.
\end{equation}
Taking the $\Omega$-trace of both sides of \eqref{eq:Theta-pot} yields 
\begin{equation}
    \begin{split}\label{eq:U-pot}
\Lambda_{\Omega}(\Theta_{\Omega}-{\widehat\theta}_0) =\Delta_{\Omega} U, \qquad U=-\mathbb{G}_{\Omega}\left(\Lambda_{\Omega}\left(\Theta_\Omega-\widehat{\theta}_0\right)\right).
    \end{split}
\end{equation}
Define
\begin{equation}\label{eq:u_t}
u_t:=\lambda(e^{-t + i s})^*(U)_{|_{M\times\{1\}}}.
\end{equation}
In what follows we let 
\[ \tau := e^{-t + i s} \in \C^*\]
and introduce the ${S}^1$-invariant functions on $\mathbb{C}^*$
\begin{align}\label{eq:CTE0-Psi}
\mathcal{E}_\Psi^0(\tau):=&-\frac{1}{(n-1)!}\sum_{j=0}^{n-2}\int_M \varphi_t\omega_{t}^{j}\wedge\rho_{0}^{2}\wedge\omega_0^{n-j-2} \\ \nonumber &+\int_M\psi_{t}\rho_{0}\wedge\omega_{t}^{[n-1]} 
-\frac{1}{4}\int_M\psi_t dd^c\psi_t\wedge\omega^{[n-1]}_{t}, \\
\label{eq:CTH-U}
\mathcal{P}^{\A}_U(\tau) :=&-\frac{1}{(n-1)!}\sum_{j=0}^{n-2}\int_M \varphi_t\omega_t^{j}\wedge\theta_0^{2}\wedge\omega_0^{n-j-2} \\ \nonumber&-2\int_M u_t\theta_0\wedge\omega_t^{[n-1]}
-\int_M u_t dd^cu_t\wedge\omega_t^{[n-1]}.
\end{align}
\begin{lemma}\label{l:currents} The functions $\mathcal{E}^0_{\Psi}(\tau)$ and $\mathcal{P}^{\A}_U(\tau)$ satisfy the following identities of currents on $\PP^1\setminus \{0\}$
\begin{align}\label{eq:ddcE0}
  -dd^c\left(\mathcal{E}_\Psi^0(\tau)\right) &=\pi_*\left((\rho_\Omega-\pi^*\omega_{\rm FS})^2\wedge\Omega^{[n-1]}\right), \\\label{eq:ddcH}
  -dd^c\left(\mathcal{P}^{\A}_U(\tau)\right)&= \pi_*\left( \Theta_\Omega^2\wedge\Omega^{[n-1]}\right).
\end{align}
\end{lemma}
\begin{proof} The computation of \eqref{eq:ddcE0} is by now standard and follows from \cite{zach, Dervan}. We 
include it here for the sake of completeness. We work on $\tstM^{\times} \cong M\times (\PP^{1}\setminus\{0\})$. For a test function $\chi$ with support in $\mathbb{P}^1\setminus\{0\}$ we denote $\hat{\chi}:=\pi^*\chi$ and compute
\[ \begin{split} &-\left\langle dd^c\left(\mathcal{E}_\Psi^0(\tau)\right),\chi\right\rangle=-\int_{\mathbb{P}^1}\left(\mathcal{E}_\Psi^0(\tau)\right)dd^c\chi \\=&\frac{1}{(n-1)!}\sum_{j=0}^{n-2}\int_{\mathbb{P}^1}dd^c\chi\int_M \varphi_t\omega_t^{j}\wedge\rho_0^{2}\wedge\omega_0^{n-j-2}\int_{\mathbb{P}^1}dd^c\chi\int_M\psi_{t}\rho_0\wedge\omega_{\tau}^{[n-1]} \\&+\frac{1}{4} \int_{\mathbb{P}^1}dd^c\chi\int_M\psi_t dd^c\psi_t\wedge\omega^{[n-1]}_{t}\\=&\frac{1}{(n-1)!}\sum_{j=0}^{n-2}\int_{\mathbb{P}^1}dd^c\chi\int_{M\times\{\tau\}}\left( \Phi \Omega^{j}\wedge\widehat{\rho}_0^{2}\wedge\widehat{\omega}_0^{n-j-2}\right)_{|M\times\{\tau\}}-\int_{\mathbb{P}^1}dd^c\chi\int_{M\times\{\tau\}}\left(\Psi\widehat{\rho}_0\wedge\Omega^{[n-1]}\right)_{M\times\{\tau\}} \\&+\frac{1}{4} \int_{\mathbb{P}^1}dd^c\chi\int_{M\times\{\tau\}}\left(\Psi dd^c\Psi\wedge\Omega^{[n-1]}\right)_{M\times\{\tau\}}\\=&\frac{1}{(n-1)!}\sum_{j=0}^{n-2}\int_{{\tstM}^\times} \Phi dd^c\hat{\chi}\wedge \Omega^{j}\wedge\widehat{\rho}_0^{2}\wedge\widehat{\omega}_0^{n-j-2} -\int_{\mathcal{\tstM}^\times}\Psi dd^c\hat{\chi}\wedge\widehat{\rho}_0\wedge\Omega^{[n-1]} +\frac{1}{4} \int_{\mathcal{\tstM}^\times}\Psi dd^c\hat{\chi}\wedge dd^c\Psi\wedge\Omega^{[n-1]}\\=&\frac{1}{(n-1)!}\sum_{j=0}^{n-2}\int_{\mathcal{\tstM}^\times} \hat{\chi} dd^c\Phi\wedge \Omega^{j}\wedge\widehat{\rho}_0^{2}\wedge\widehat{\omega}_0^{n-j-2} -\int_{\mathcal{\tstM}^\times}\hat{\chi} dd^c\Psi\wedge\widehat{\rho}_0\wedge\Omega^{[n-1]} \\&+\frac{1}{4} \int_{\mathcal{\tstM}^\times}\hat{\chi} dd^c\Psi\wedge dd^c\Psi\wedge\Omega^{[n-1]}\\=&\frac{1}{(n-1)!}\sum_{j=0}^{n-2}\int_{\mathcal{\tstM}^\times} \hat{\chi} (\Omega-\widehat{\omega}_0)\wedge \Omega^{j}\wedge\widehat{\rho}_0^{2}\wedge\widehat{\omega}_0^{n-j-2} +2\int_{\mathcal{\tstM}^\times}\hat{\chi} ( \rho_\Omega-\widehat{\omega}_{\rm FS}-\widehat{\rho}_{0})\wedge\widehat{\rho}_0\wedge\Omega^{[n-1]}\\&+ \int_{\mathcal{\tstM}^\times}\hat{\chi}  ( \rho_\Omega-\widehat{\omega}_{\rm FS}-\widehat{\rho}_{0})^2\wedge\Omega^{[n-1]} \\
=&\frac{1}{(n-1)!}\sum_{j=0}^{n-2}\int_{\mathcal{\tstM}^\times} \hat{\chi} (\Omega^{j+1}\wedge\widehat{\rho}_0^{2}\wedge\widehat{\omega}_0^{n-j-2}-\Omega^{j}\wedge\widehat{\rho}_0^{2}\wedge\widehat{\omega}_0^{n-j-1}) -\int_{\mathcal{\tstM}^\times}\hat{\chi} \widehat{\rho}_0^2\wedge\Omega^{[n-1]}+ \int_{\mathcal{\tstM}^\times}\chi  ( \rho_\Omega-\widehat{\omega}_{\rm FS})^2\wedge\Omega^{[n-1]}\\
=& \int_{\mathcal{\tstM}^\times}\hat{\chi}  ( \rho_\Omega-\widehat{\omega}_{\rm FS})^2\wedge\Omega^{[n-1]}.\end{split}\] 
We now detail \eqref{eq:ddcH}. 
\[
\begin{split}
   - \left\langle dd^c\mathcal{P}^{\A}_U,\chi\right\rangle =&-\int_{\mathbb{P}^1}dd^c\chi \left(\mathcal{P}^{\A}_U(\tau)\right)\\
=&\frac{1}{(n-1)!}\sum_{j=0}^{n-2}\int_{\mathbb{P}^1}dd^c\chi\int_M \varphi_t\omega_t^{j}\wedge\theta_0^{2}\wedge\omega_0^{n-j-2} \\
&+2\int_{\mathbb{P}^1}dd^c\chi\int_M u_{t}\theta_0\wedge\omega_{t}^{[n-1]}
+\int_{\mathbb{P}^1}dd^c\chi\int_M u_{t}dd^cu_t\wedge\omega_{t}^{[n-1]}\\
=&\frac{1}{(n-1)!}\sum_{j=0}^{n-2}\int_{\mathbb{P}^1}dd^c\chi\int_{M\times\{\tau\}}\left( \Phi \,  \Omega^{j}\wedge\widehat{\theta}_0^{2}\wedge\widehat{\omega}_0^{n-j-2}\right)_{|M\times\{\tau\}}\\
&+2\int_{\mathbb{P}^1}dd^c\chi\int_{M\times\{\tau\}}\left(U\widehat{\theta}_0\wedge\Omega^{[n-1]}\right)_{M\times\{\tau\}} \\
&+\int_{\mathbb{P}^1}dd^c\chi\int_{M\times\{\tau\}}\left(Udd^cU\wedge\Omega^{[n-1]}\right)_{M\times\{\tau\}}\\
=&\frac{1}{(n-1)!}\sum_{j=0}^{n-2}\int_{\mathcal{M}^\times} \Phi dd^c\hat{\chi}\wedge \Omega^{j}\wedge\widehat{\theta}_0^{2}\wedge\widehat{\omega}_0^{n-j-2} +2\int_{\mathcal{M}^\times}U dd^c\hat{\chi}\wedge\widehat{\theta}_0\wedge\Omega^{[n-1]} \\
&+\int_{\mathcal{\tstM}^\times}U dd^c\hat{\chi}\wedge dd^cU\wedge\Omega^{[n-1]}\\
=&\frac{1}{(n-1)!}\sum_{j=0}^{n-2}\int_{\mathcal{\tstM}^\times} \hat{\chi} dd^c\Phi\wedge \Omega^{j}\wedge\widehat{\theta}_0^{2}\wedge\widehat{\omega}_0^{n-j-2} +2\int_{\mathcal{\tstM}^\times}\hat{\chi} dd^cU\wedge\widehat{\theta}_0\wedge\Omega^{[n-1]} \\
&+\int_{\mathcal{\tstM}^\times}\hat{\chi} dd^cU\wedge dd^cU\wedge\Omega^{[n-1]}\\
=&\frac{1}{(n-1)!}\sum_{j=0}^{n-2}\int_{\mathcal{\tstM}^\times} \hat{\chi} (\Omega-\widehat{\omega}_0)\wedge \Omega^{j}\wedge\widehat{\theta}_0^{2}\wedge\widehat{\omega}_0^{n-j-2} \\
&+2\int_{\mathcal{\tstM}^\times}\hat{\chi} ( \Theta_\Omega-\widehat{\theta}_{0})\wedge\widehat{\theta}_0\wedge\Omega^{[n-1]}+\int_{\mathcal{\tstM}^\times}\hat{\chi} ( \Theta_\Omega-\widehat{\theta}_{0})^2\wedge\Omega^{[n-1]}\\
=&\frac{1}{(n-1)!}\sum_{j=0}^{n-2}\int_{\mathcal{\tstM}^\times} \hat{\chi} \left(\Omega^{j+1}\wedge\widehat{\theta}_0^{2}\wedge\widehat{\omega}_0^{n-j-2}-\Omega^{j}\wedge\widehat{\theta}_0^{2}\wedge\widehat{\omega}_0^{n-j-1}\right)\\
&+\int_{\mathcal{\tstM}^\times}\hat{\chi} \Theta_\Omega^2\wedge\Omega^{[n-1]}-\int_{\mathcal{\tstM}^\times}\hat{\chi} \widehat{\theta}_{0}^2\wedge\Omega^{[n-1]}=\int_{\mathcal{\tstM}^\times}\hat{\chi} \Theta_\Omega^2\wedge\Omega^{[n-1]}.
\end{split}
\] 
The equality \eqref{eq:ddcH} follows.
\end{proof}
Using the ${S}^1$-invariance of the functions $\mathcal{E}_\Psi^0(\tau)$ and $\mathcal{P}^{\A}_U(\tau)$ and the Green--Riesz formula, Lemma~\ref{l:currents} yields
\begin{cor}\label{c:slopes} \begin{align}\label{eq:Slope_E0Psi}
  \underset{t\to \infty}{\lim}\left(\frac{d}{dt} \mathcal{E}_\Psi^0\right) &=-\int_{\mathcal{\tstM}}( \rho_\Omega-\pi^*\omega_{\rm FS})^2\wedge\Omega^{[n-1]}, \\\label{eq:Slope_HU}
 \underset{t\to +\infty}{\lim}\left(\frac{d}{dt} \mathcal{P}^{\A}_U\right) &=-\int_{\mathcal{\tstM}} \Theta_\Omega^2\wedge\Omega^{[n-1]}.
\end{align}
\end{cor}
\begin{proof} Let $\mathbb{D}_\epsilon\subset\mathbb{C}$ be a disc of radius $\epsilon$ centred at the origin and $(t,s)$ the coordinates given by $\tau=e^{-t+is}\in \mathbb{C}^*$. We compute using the ${S}^1$-invariance of $\mathcal{E}_\Psi^0(\tau)$, Lemma~\ref{l:currents} and the Green--Riesz formula:
\[
\begin{split}
    &\int_{\mathcal{\tstM}}( \rho_\Omega-\pi^*\omega_{\rm FS})^2\wedge\Omega^{[n-1]}=\underset{\epsilon\to 0}{\lim} \int_{\mathcal{\tstM}\setminus\pi^{-1}(\mathbb{D}_\epsilon)} ( \rho_\Omega-\pi^*\omega_{\rm FS})^2\wedge\Omega^{[n-1]}\\
    &=\underset{\epsilon\to 0}{\lim} \int_{\mathbb{P}^1\setminus\mathbb{D}_\epsilon} \pi_*\left( ( \rho_\Omega-\pi^*\omega_{\rm FS})^2\wedge\Omega^{[n-1]}\right)
    =-\underset{\epsilon\to 0}{\lim} \int_{\mathbb{P}^1\setminus\mathbb{D}_\epsilon}  dd^c\left(\mathcal{E}_\Psi^0(\tau)\right)\\
    &=-\underset{\epsilon\to 0}{\lim}\left(\frac{d}{dt}_{|t=-\log\epsilon}\mathcal{E}_\Psi^0(\varphi_t)\right)  =-\underset{t\to +\infty}{\lim}\frac{d}{dt} \mathcal{E}_\Psi^0(\varphi_t).
\end{split}
\]
The computation for \eqref{eq:Slope_HU} is similar. \end{proof}
 
To evaluate the slopes at infinity of $\mathcal{E}^0(\varphi_t)$ and $\mathcal{P}^{\A}(\varphi_t)$, we compare them  with the functions $\mathcal{E}_{\Psi}^0(t)$  and $\mathcal{P}_U^{\A}(t)$, respectively. 
\begin{lemma}\label{l:asymptot} If the central fibre $M_0$ of $\tstM$ is reduced, then we have
\[ \lim_{t\to \infty}\left(\frac{\mathcal{E}^0(\varphi_t)}{t}\right) = \lim_{t\to \infty} \left(\frac{\mathcal{E}_{\Psi}^0(t)}{t}\right), \qquad \lim_{t\to \infty}\left(\frac{\mathcal{P}^{\A}(\varphi_t)}{t}\right) = \lim_{t\to \infty} \left(\frac{\mathcal{P}_{U}^{\A}(t)}{t}\right). \]
\end{lemma}
\begin{proof}
\[
\begin{split}
\mathcal{E}_\Psi^0(t)-\mathcal{E}^0(\varphi_t)=&\int_M\left(\psi_{t} -\log\left(\frac{\omega_t^n}{\omega_0^n}\right)\right)\rho_0 \wedge \omega_t^{[n-1]}+\frac{1}{4}\int_M\log\left(\frac{\omega_t^n}{\omega_0^n}\right)dd^c\left(\log\left(\frac{\omega_t^n}{\omega_0^n}\right)-\psi_{t} \right)\wedge \omega^{[n-1]}_t  \\
&+\frac{1}{4}\int_M\left(\log\left(\frac{\omega_t^n}{\omega_0^n}\right)-\psi_{t}\right) dd^c\psi_t\wedge\omega^{[n-1]}_{t}\\
=&\int_{M\times\{\tau\}}\left(\Psi -\log\left(\frac{\Omega^n\wedge\widehat{\omega}_{\rm FS}}{\widehat{\omega}_0^n\wedge\widehat{\omega}_{\rm FS}}\right)\right)\lambda(\tau^{-1})^*(\rho_0) \wedge \lambda(\tau^{-1})^*\omega_t^{[n-1]}\\
&+\frac{1}{4}\int_{M\times\{\tau\}}\left[\log \left(\frac{\Omega^n\wedge\widehat{\omega}_{\rm FS}}{\widehat{\omega}^n\wedge\widehat{\omega}_{\rm FS}}\right)+\Psi\right]dd^c\left(\log\left(\frac{\Omega^n\wedge\widehat{\omega}_{\rm FS}}{\widehat{\omega}_0^n\wedge\widehat{\omega}_{\rm FS}}\right)-\Psi\right)\wedge \lambda(\tau^{-1})^*\omega^{[n-1]}_t \\
=&\int_{M\times\{\tau\}}\log\left(\frac{\Omega^{n+1}}{\Omega^n\wedge\widehat{\omega}_{\rm FS}}\right)\left(\widehat{\rho}_0 \wedge \Omega^{[n-1]}\right)_{\mid M\times\{\tau\}}\\
&+\frac{1}{4}\int_{M\times\{\tau\}}\left[\log\left(\frac{\Omega^{n+1}}{\Omega^n\wedge\widehat{\omega}_{\rm FS}}\right)-2\Psi\right]\left(dd^c\left(\log\left(\frac{\Omega^{n+1}}{\Omega^n\wedge\widehat{\omega}_{\rm FS}}\right)\right)\wedge \Omega^{[n-1]}\right)_{\mid M\times\{\tau\}} \\
=&\int_{M\times\{\tau\}}\log\left(\frac{\Omega^{n+1}}{\Omega^n\wedge\widehat{\omega}_{\rm FS}}\right)\left(\left(\widehat{\rho}_0 -\frac{1}{2}dd^c\Psi\right)\wedge \Omega^{[n-1]}\right)_{\mid M\times\{\tau\}}\\
&+\frac{1}{4}\int_{M\times\{\tau\}}\log\left(\frac{\Omega^{n+1}}{\Omega^n\wedge\widehat{\omega}_{\rm FS}}\right) \left[dd^c\left(\log\left(\frac{\Omega^{n+1}}{\Omega^n\wedge\widehat{\omega}_{\rm FS}}\right)\right)\wedge \Omega^{[n-1]}\right]_{\mid M\times\{\tau\}} \\
=& \int_{M\times\{\tau\}}\log\left(\frac{\Omega^{n+1}}{\Omega^n\wedge\widehat{\omega}_{\rm FS}}\right)\left(\rho_\Omega \wedge\Omega^{[n-1]}\right)_{\mid M\times\{\tau\}}\\
&+\frac{1}{4}\int_{M\times\{\tau\}}\log\left(\frac{\Omega^{n+1}}{\Omega^n\wedge\widehat{\omega}_{\rm FS}}\right) \left[dd^c\left(\log\left(\frac{\Omega^{n+1}}{\Omega^n\wedge\pi^*\omega_{\rm FS}}\right)\right)\wedge \Omega^{[n-1]} \right]_{\mid M\times\{\tau\}}
\end{split}
\]
Letting  
\[Z:=\frac{\Omega^{n+1}}{\Omega^n\wedge\pi^*\omega_{\rm FS}}, \qquad 
\Gamma(\tau):=\int_{M\times\{\tau\}}\log (Z)\left(\rho_\Omega +\frac{1}{4}dd^c\log(Z)\right)\wedge\Omega^{[n-1]}, 
\]
we have shown above
\[\mathcal{E}_\Psi^0(t)-\mathcal{E}^0(\varphi_t) = \Gamma(\tau), \qquad \tau = e^{-t+is}.\]
The proof of the first equality in Lemma~\ref{l:asymptot} will then follow if we show that $\Gamma(\tau)$ is bounded as $\tau \to 0$. To this end, since $\Xi:=\rho_\Omega +\frac{1}{4}dd^c\log(Z)$ is a smooth  2-form on  $\mathcal{\tstM}$, there are positive constants $C,c$ such that 
\begin{equation}\label{eq:Xi<Om}
-c\Omega<\Xi<C\Omega.    
\end{equation}
We write
\[
\Gamma(\tau)=\Gamma_1(\tau)-(cn)\Gamma_2(\tau),
\]
where
\[ \Gamma_1(\tau):= \int_{M\times\{\tau\}}\log(Z) \, (\Xi+ c\Omega)\wedge\Omega^{[n-1]}, \qquad \Gamma_2(\tau):=\int_{M\times\{\tau\}}\log(Z)\Omega^{[n]}.\]
It is well-known (see \cite[Rem.4.12]{Dervan-Ross}) that $\Gamma_2(\tau)$ is  bounded as $\tau\to0$,  provided that ${M}_0$ is reduced. As for $\Gamma_1(\tau)$, using \eqref{eq:Xi<Om}, we have 
\[
|\Gamma_1(\tau)|\leq \int_{M\times\{\tau\}}|\log(Z)|(\Xi + c \Omega)\wedge\Omega^{[n-1]} \leq n(c+C)\int_{M\times\{\tau\}}|\log(Z)|\Omega^{[n]},
\]
which is bounded from above. This completes the proof of the  first equality of Lemma~\ref{l:asymptot}.

\bigskip 
Now we evaluate the asymptotic slope at infinity of $\mathcal{P}^{\A}(\varphi_t)$. Let $S(t)$ be the function
\[
S(t):= \mathcal{P}^{\A}_U(t)- \mathcal{P}^{\A}(\varphi_t).
\]
We compute (using $\A\cdot \Ome^{n-1}=0$):
\[
\begin{split}
S(t)=&\int_M\left(\mathbb{G}_{\varphi_t}\left(\Lambda_{\varphi_t}\theta_{0}\right) - u_t\right)\theta_0\wedge\omega_t^{[n-1]}-\int_M u_t (\theta_0+dd^cu_t)\wedge\omega_t^{[n-1]}\\
=&\int_M\mathbb{G}_{\varphi_t}\Lambda_{\varphi_t}\left(\theta_{0} +dd^c u_t\right)\theta_0\wedge\omega_t^{[n-1]}-\int_M u_t (\theta_0+dd^cu_t)\wedge\omega_t^{[n-1]}\\
&-\frac{\left(\beta\cdot\Ome^{[n-1]}\right)}{\Vol (\Ome)}\int_M\
u_t\omega^{[n]}_t\\
=&\int_M\mathbb{G}_{\varphi_t}\Lambda_{\varphi_t}\left(\theta_{0} +dd^c u_t\right) \left(\theta_0+dd^c u_t\right)\wedge\omega_t^{[n-1]}-\int_M u_t (\theta_0+dd^cu_t)\wedge\omega_t^{[n-1]}\\
&-\int_M \mathbb{G}_{\varphi_t}\Lambda_{\varphi_t}\left(\theta_0 +dd^c u_t\right) dd^c u_t\wedge\omega_t^{[n-1]}  \\
=&\int_M\mathbb{G}_{\varphi_t}\Lambda_{\varphi_t}\left(\theta_0 +dd^c u_t\right) \left(\theta_0+dd^c u_t\right)\wedge\omega_t^{[n-1]}-\int_M u_t(\theta_0+dd^cu_t)\wedge\omega_t^{[n-1]}\\
&+\int_M\mathbb{G}_{\varphi_t}\Lambda_{\varphi_t}\left(\theta_0 +dd^c u_t\right)  \Delta_{\varphi_t}(u_t)\omega_t^{[n]}\\
=&\int_M\mathbb{G}_{\varphi_t}\left(\Lambda_{\varphi_t}\left(\theta_0 +dd^c u_t\right)\right) \left(\theta_0+dd^c u_t\right)\wedge\omega_t^{[n-1]}-\int_M u_t (\theta_0+dd^cu_t)\wedge\omega_t^{[n-1]}\\
&+\int_M \left(u_t-\frac{1}{\Vol(\Ome)}\int_Mu_t\omega_t^{[n]}\right)\Lambda_{\varphi_t}\left(\theta_0 +dd^c u_t\right)  \omega_t^{[n]}\\
=&\int_M\mathbb{G}_{\varphi_t}\left(\Lambda_{\varphi_t}\left(\theta_0 +dd^c u_t\right)\right)  \left(\theta_0 +dd^c u_t\right)\wedge\omega_t^{[n-1]} -\frac{\left(\beta\cdot\Ome^{[n-1]}\right)}{\Vol (\Ome)}\int_Mu_t\omega^{[n]}_t\\
=&\int_M\mathbb{G}_{\varphi_t}\left(\Lambda_{\varphi_t}\left(\theta_0 +dd^c u_t\right)\right)  \left(\theta_0 +dd^c u_t\right)\wedge\omega_t^{[n-1]}
\end{split}
\]
We will show  again that  $S(\tau)$ is bounded as $\tau \to 0$.  Identifying $\tstM^{\times} \cong  M \times (\PP^1\setminus\{0\})$ under $\Pi$ and letting 
\[\Omega_\tau:=\Omega_{|M\times\{\tau\}}, \qquad  f_\Omega:=\frac{\widehat{\omega}_{\rm FS}\wedge\Theta_{\Omega}\wedge\Omega^{[n-1]}}{\widehat{\omega}_{\rm FS} \wedge \Omega^{[n]}},\]
we further compute
\[\begin{split}
 S(\tau) =&\int_{M\times\{\tau\}}\lambda(\tau^{-1})^*\left(\mathbb{G}_{\varphi_t}\Lambda_{\varphi_t}\left(\theta_0 +dd^c u_t\right)\right)  \lambda(\tau^{-1})^*\left(\theta_0 +dd^c u_t\right)\wedge\lambda(\tau^{-1})^*\omega_t^{[n-1]}\\
=&\int_{M\times\{\tau\}}\left(\lambda(\tau^{-1})^*\mathbb{G}_{\varphi_t}\lambda(\tau)^*\right)\left(\lambda(\tau^{-1})^*\Lambda_{\varphi_t}\left(\theta_0 +dd^c u_t\right)\right)  \left(\Theta_{\Omega}\wedge\Omega^{[n-1]}\right)_{|M\times\{\tau\}}\\
=&\int_{M\times\{\tau\}}\left(\lambda(\tau^{-1})^*\mathbb{G}_{\varphi_t}\lambda(\tau)^*\right)\left(\frac{\left(\Theta_{\Omega}\wedge\Omega^{[n-1]}\right)_{|M\times\{\tau\}}}{\left(\Omega^{[n]}\right)_{|M\times\{\tau\}}}\right)  \left(\Theta_{\Omega}\wedge\Omega^{[n-1]}\right)_{|M\times\{\tau\}}\\
=&\int_{M\times\{\tau\}} \mathbb{G}_{\Omega_\tau}\left(\frac{\pi^*\omega_{\rm FS}\wedge\Theta_{\Omega}\wedge\Omega^{[n-1]}}{\pi^*\omega_{\rm FS}\wedge\Omega^{[n]}}\right)  \left(\Theta_{\Omega}\wedge\Omega^{[n-1]}\right)_{|M\times\{\tau\}}\\
=&\int_{M\times\{\tau\}} \mathbb{G}_{\Omega_\tau}(f_\Omega)f_\Omega\,\Omega^{[n]}_\tau =\int_{M\times\{\tau\}} |d\mathbb{G}_{\Omega_\tau}(f_\Omega)|_{\Omega_\tau}^2 \Omega^{[n]}_\tau  
\end{split}
\]
The last equality shows that $S(\tau)\geq0$. Using the  basic property of the Green function
\[
\mathbb{G}_{\Omega_\tau}(f_\Omega)(x)=\int_{M\times\{\tau\}} \mathbb{G}_{\Omega_\tau}(x, \cdot)f_\Omega\Omega_\tau^{[n]}\quad \forall x\in M\times\{\tau\}
\]
we have
\[
\begin{split}
    S(\tau)=&\int_{M\times\{\tau\}} \mathbb{G}_{\Omega_\tau}(f_\Omega)f_\Omega\,\Omega^{[n]}_\tau\\
    =&\int_{M\times\{\tau\}} \left(\int_{M\times\{\tau\}} \mathbb{G}_{\Omega_\tau}(x,y)f_\Omega(y)\Omega_\tau^{[n]}(y)\right)f_\Omega(x)\,\Omega^{[n]}_\tau(x)
\end{split}
\]
Let $A>0$ be a positive constant such that $\Theta_\Omega+A\Omega>0$ on $\mathcal{\tstM}$ and thus
\[
\tilde f_\Omega:=\frac{\widehat{\omega}_{\rm FS}\wedge(\Theta_{\Omega}+A\Omega)\wedge\Omega^{[n-1]}}{\widehat{\omega}_{\rm FS}\wedge\Omega^{[n]}}=f_\Omega+nA>0.
\]
Using $\int_{M\times\{\tau\}} \mathbb{G}_{\Omega_\tau}(x,y) \Omega_\tau^{[n]}(y)=0$, we get
\[
\begin{split}
    S(\tau)
    =&\int_{M\times\{\tau\}} \left(\int_{M\times\{\tau\}} \mathbb{G}_{\Omega_\tau}(x,y)\tilde{f}_\Omega(y)\Omega_\tau^{[n]}(y)\right){\tilde f}_\Omega(x)\,\Omega^{[n]}_\tau(x)
\end{split}
\]
By \cite[Thm.A-(1)]{GT} (which in turm is based on \cite{GPS,GPSS}), we have a uniform upper bound for the Green kernel, i.e.
\[
\forall x\in M\times\{\tau\}, \sup_{y\in M\times\{\tau\}}\mathbb{G}_{\Omega_\tau}(x,y)<C_0
\]
independent of $x,\tau$. Hence,
\[
\begin{split}
   0 \leq S(\tau)
    \leq& C_0\left(\int_{M\times\{\tau\}} {\tilde f}_\Omega\Omega^{[n]}_{\tau}\right)^2
    \leq C_0 \sup_{\mathcal{\tstM}}({\tilde f}_\Omega){\rm vol}(\Ome).
\end{split}
\]
This yields the second equality of the Lemma. 
\end{proof}
\begin{proof}[Proof of Theorem~\ref{E-Slope}] We have $ \mathcal{E}_{\rm NA} = \mathcal{E}_{\rm NA}^0 + \mathcal{P}^{\beta}_{\rm NA} + \frac{c_{\Ome, \A}}{2} \mathcal{I}_{\rm NA}$ 
where 
\[
\begin{split}
\mathcal{E}_{\rm NA}^0 &:=-\int_{\tstM}(\rho_{\Omega}-\pi^*\omega_{\rm FS})^2\wedge \Omega^{[n-1]}= -(2\pi)^{n+1}\left(c_1(\tstK_{\tstM/\PP^1})^2\cdot \tstA^{[n-1]}\right) \\
\mathcal{P}^{\beta}_{\rm NA} &:= -\int_{\tstM} \Theta_{\Omega}^2 \wedge \Omega^{[n-1]}= -(2\pi)^{n+1}\left(\tstB^2\cdot \tstA^{[n-1]}\right)
\end{split} \]
compute respectively the slope of $\mathcal{E}^0(\varphi_t)$ and $\mathcal{P}^{\A}(\varphi_t)$ (see Corollary~\ref{c:slopes} and Lemma~\ref{l:asymptot})  whereas 
\[ \mathcal{I}_{\rm NA}:= \int_{\tstM}\Omega^{[n+1]}= (2\pi)^{n+1}\tstA^{[n+1]}\]
computes the slope of $\mathcal{I}(\varphi_t)$   (see e.g. \cite{Dervan-Ross, zach}). The claim follows.
\end{proof}
\begin{rmk} Theorem~\ref{E-Slope} provides an analog of the Donaldson--Futaki invariant of a K\"ahler test configuration (originally associated with the cscK problem on $(M, \Ome)$), but which is relevant to the study of solutions of \eqref{PDEAn} on $(M, \Ome, \A)$. For instance, one can check that for a product K\"ahler test configuration (i.e. the case when $M_0=M$), $\mathcal{E}_{\rm NA}$ computes, up to a non-zero multiplicative constant,  the Futaki invariant of $(M,\Ome, \A)$, introduced in \eqref{d:futaki},   with respect to the generator $X$ of the $\C^*$-action on $M_0=M$. It will be interesting to relate in general the non-negativity of $\mathcal{E}_{\rm NA}$--- a condition which would correspond to the  ``$\mathcal{E}_{\rm NA}$-semistability'' of $(M, \Ome, \A)$ on smooth/orbifold test configurations---with the existence of a solution of \eqref{PDEAn}. Such a relation will follow directly from Theorem~\ref{E-Slope} if one shows that a solution of \eqref{PDEAn} realizes a global minimum of $\mathcal{E}$ (or at least that $\mathcal{E}$ is bounded below should a solution exists). By \cite[Thm.~1.2]{SW},  this actually holds true in the Fano case with $\Ome=2\pi c_1(M), \A=0$, assuming that $M$ admits a K\"ahler--Einstein metric $\omega_{\rm KE} \in \Ome$ (in this case $\omega_{\rm KE}$ solves \eqref{PDEAn}). In general,  a major difficulty for obtaining the minimizing property of the solution comparing to the arguments in the cscK case~\cite{BB} is that the Mabuchi functional $\mathcal{E}$ associated to \eqref{PDEAn} is not easy to define on weak
$C^{1,1}$ geodesics of $\omega_0$-relative K\"ahler potentials; additionally,  the Mabuchi functional $\mathcal{E}$ is convex only along (smooth) geodesics of K\"ahler metrics of \emph{positive scalar curvature}. 
\end{rmk}

\bigskip
\noindent
{\bf Acknowledgements.} The first named author  is  grateful to D. Martelli and M. K. Crisafio for an illuminating discussion which led to the examples in Theorem~\ref{thm:main3}, as well as for their consequent comments on the manuscript and clarifications. The authors thank R. Dervan, M. Garcia Fernandez, G. Grantcharov, E. Legendre, F. Rochon and J. Streets for sharing with them their expertise, as well as the referee for their useful remarks.

\bigskip
\noindent
{\bf Statements and Declarations.} V.A. was supported in part by an NSERC Discovery Grant RGPIN-2023-03316.
K.-H. L. and A.L. were supported by  CIRGET post-doctoral fellowships and FRQ Team Grants. No automatic tools or data sharing were used and no datasets were generated or analysed during the current study.

\end{document}